\documentclass[12pt]{amsart}

\usepackage{currvita, amssymb, amsmath, amsthm, 
graphicx, subfigure, pifont, wrapfig, pinlabel, amscd,
latexsym, exscale, enumerate, amsfonts, times, hyphenat, mathtools, a4wide, fullpage, array, bbm, url}
\usepackage{float}
\usepackage[colorlinks=false, pdfborder={0.9 0.9 0.9}]{hyperref}
\usepackage{stmaryrd}
\usepackage{multirow}
\usepackage{xcolor}
\usepackage[all,color]{xy}
\usepackage[vcentermath]{youngtab}
\usepackage{thmtools}

\newcommand{\Ob}{\mathrm{Ob}}
\newcommand{\Hom}{\mathrm{Hom}}
\newcommand{\Mat}{\boldsymbol{\mathrm{Mat}}}
\newcommand{\Mati}{\boldsymbol{\mathrm{Mat}}^{\mathrm{fs}}_{\infty}}
\newcommand{\Kar}{\boldsymbol{\mathrm{Kar}}}
\newcommand{\End}{\mathrm{End}}
\newcommand{\Endbf}{\boldsymbol{\mathrm{End}}}

\newcommand{\circv}{\circ_{\boldsymbol{v}}}
\newcommand{\circh}{\circ_{\boldsymbol{h}}}

\newcommand{\bN}{\mathbb{Z}_{\geq 0}}
\newcommand{\bZ}{\mathbb{Z}}
\newcommand{\bQ}{\mathbb{Q}}

\newcommand{\bC}{\mathbb{Q}(q)}
\newcommand{\bV}{\mathbb{T}_m}
\newcommand{\bVi}{\mathbb{T}_{\infty}}

\newcommand{\hooklongrightarrow}{\lhook\joinrel\longrightarrow}
\newcommand{\twoheadlongrightarrow}{\relbar\joinrel\twoheadrightarrow}

\newcommand{\qbin}[2]{
\left[
 \begin{array}{c}
 #1 \\
 #2 \\
 \end{array}
 \right]
}

\newcommand{\Uu}{\boldsymbol{\mathrm{U}}}
\newcommand{\Uq}{\boldsymbol{\mathrm{U}}_q}
\newcommand{\Uv}{\boldsymbol{\mathrm{U}}_v}
\newcommand{\Uvd}{\dot{\boldsymbol{\mathrm{U}}}_v}
\newcommand{\Ua}{\boldsymbol{\mathrm{U}}_{\mathrm{A}}}

\newcommand{\bbi}{\boldsymbol{\mathrm{i}}}
\newcommand{\bbj}{\boldsymbol{\mathrm{j}}}

\newcommand{\bbin}{\boldsymbol{\infty}}
\newcommand{\bbiu}{\boldsymbol{\mathrm{i}+1}}
\newcommand{\bbid}{\boldsymbol{\mathrm{i}-1}}

\newcommand{\bbz}{\boldsymbol{0}}
\newcommand{\bbo}{\boldsymbol{1}}
\newcommand{\bbt}{\boldsymbol{2}}

\newcommand{\bbmd}{\boldsymbol{\mathrm{m}-1}}
\newcommand{\bbmdd}{\boldsymbol{\mathrm{m}-2}}
\newcommand{\bbm}{\boldsymbol{\mathrm{m}}}
\newcommand{\bbmu}{\boldsymbol{\mathrm{m}+1}}
\newcommand{\JWr}{\mathrm{JW}^r}
\newcommand{\JWg}{\mathrm{JW}^g}
\newcommand{\JWx}{\mathrm{JW}x}

\newcommand{\Vn}[1]{\Delta_q({#1})}
\newcommand{\dVn}[1]{\nabla_q({#1})}
\newcommand{\Ln}[1]{L_q({#1})}
\newcommand{\Tn}[1]{T_q({#1})}
\newcommand{\Am}{Q_{m}}
\newcommand{\Ai}{Q_{\infty}}
\newcommand{\At}{Q_m^{\mathrm{triv}}}
\newcommand{\Ait}{Q_{\infty}^{\mathrm{triv}}}

\newcommand{\Att}{Q_{\infty}^{\mathrm{triv}}}

\newcommand{\Endo}{p\boldsymbol{\mathrm{End}}}

\newcommand{\pMod}{p\boldsymbol{\mathrm{Mod}}}
\newcommand{\pModgr}{p\boldsymbol{\mathrm{Mod}}_{\mathrm{gr}}}

\newcommand{\ModAm}{\boldsymbol{\mathrm{Mod}}\text{-}Q_m}
\newcommand{\pModAm}{p\boldsymbol{\mathrm{Mod}}\text{-}Q_m}
\newcommand{\ModAi}{\boldsymbol{\mathrm{Mod}}\text{-}Q_{\infty}}
\newcommand{\pModAi}{p\boldsymbol{\mathrm{Mod}}\text{-}Q_{\infty}}

\newcommand{\ModA}{\boldsymbol{\mathrm{Mod}}\text{-}A}
\newcommand{\AModgr}{A\text{-}\boldsymbol{\mathrm{Mod}}_{\mathrm{gr}}}
\newcommand{\pModgrA}{p\boldsymbol{\mathrm{Mod}}_{\mathrm{gr}}\text{-}A}

\newcommand{\ModgrA}{\boldsymbol{\mathrm{Mod}}_{\mathrm{gr}}\text{-}A}

\newcommand{\ModgrAm}{\boldsymbol{\mathrm{Mod}}_{\mathrm{gr}}\text{-}Q_m}

\newcommand{\AipModgrAi}{Q_{\infty}\text{-}p\boldsymbol{\mathrm{Mod}}_{\mathrm{gr}}\text{-}Q_{\infty}}
\newcommand{\T}{\boldsymbol{\mathfrak{T}}}
\newcommand{\Tgr}{\boldsymbol{\mathfrak{T}}^{\mathrm{gr}}}
\newcommand{\Tgrl}{\boldsymbol{\mathfrak{T}}_{\lambda}^{\mathrm{gr}}}
\newcommand{\D}{\boldsymbol{\mathfrak{D}}(\infty)}
\newcommand{\QD}{\boldsymbol{\mathfrak{QD}}(\infty)}
\newcommand{\Dk}{\widehat{\boldsymbol{\mathfrak{D}}}(\infty)}
\newcommand{\QDk}{\widehat{\boldsymbol{\mathfrak{QD}}}(\infty)}
\newcommand{\ModgrAt}{\boldsymbol{\mathrm{Mod}}_{\mathrm{gr}}\text{-}Q_m^{\mathrm{triv}}}
\newcommand{\pModgrAt}{p\boldsymbol{\mathrm{Mod}}_{\mathrm{gr}}\text{-}Q_m^{\mathrm{triv}}}

\newcommand{\pModAt}{p\boldsymbol{\mathrm{Mod}}\text{-}Q_m^{\mathrm{triv}}}

\newcommand{\pModAtt}{p\boldsymbol{\mathrm{Mod}}\text{-}Q_{\infty}^{\mathrm{triv}}}
\newcommand{\pModgrAm}{p\boldsymbol{\mathrm{Mod}}_{\mathrm{gr}}\text{-}Q_m}

\definecolor{mycolor}{rgb}{0.9,0,0}
\definecolor{colormy}{rgb}{0,0.85,0}
\definecolor{mycolormy}{rgb}{0.9,0,0.85}

\theoremstyle{definition}
\newtheorem{thm}{Theorem}[section]
\newtheorem{cor}[thm]{Corollary}
\newtheorem{lem}[thm]{Lemma}
\newtheorem{prop}[thm]{Proposition}

\declaretheorem[style=definition,name=Notation,qed=$\blacktriangle$,numberlike=thm]{nota}
\declaretheorem[style=definition,name=Remark,qed=$\blacktriangle$,numberlike=thm]{rem}
\declaretheorem[style=definition,name=Example,qed=$\blacktriangle$,numberlike=thm]{ex}
\declaretheorem[style=definition,name=Definition,qed=$\blacktriangle$,numberlike=thm]{defn}
\declaretheorem[style=definition,name=Definition,numberlike=thm]{defnn}

\newcommand{\makeqed}{\hfill\ensuremath{\square}}

\newcommand{\maketriqed}{\hfill\ensuremath{\blacktriangle}}

\hypersetup{
  colorlinks = false,
  urlcolor = blue,
  pdfauthor = {Henning Haahr Andersen and Daniel Tubbenhauer},
  pdfkeywords = {},
  pdftitle = {Diagram categories for \texorpdfstring{$\Uq$}{Uq}-tilting modules at roots of unity},
  pdfsubject = {},
  pdfpagemode = UseNone,
  bookmarksopen = true,
  bookmarksopenlevel = 3,
  pdfdisplaydoctitle = true,
  linktocpage=true,
}

\title[Diagram categories for $\Uq$-tilting modules at roots of unity]{Diagram categories for $\Uq$-tilting modules at roots of unity}
\author{Henning Haahr Andersen}
\address{H.H.A.: Centre for Quantum Geometry of Moduli Spaces, Aarhus University, Ny Munkegade 118, building 1530, room 327, 8000 Aarhus C, Denmark}
\email{mathha@qgm.au.dk}

\author{Daniel Tubbenhauer}
\address{D.T.: Centre for Quantum Geometry of Moduli Spaces, Aarhus University, Ny Munkegade 118, building 1530, room 316, 8000 Aarhus C, Denmark}
\email{dtubben@qgm.au.dk}

\thanks{The authors were partially supported by the center of excellence grant ``Centre for Quantum Geometry of Moduli Spaces (QGM)'' from the ``Danish National Research Foundation (DNRF)''}
\begin{document}
\begin{abstract}
We give a diagrammatic presentation of the category 
of $\Uq(\mathfrak{sl}_2)$-tilting modules $\T$ 
for $q$ being a root of unity and 
introduce a grading on $\T$. This grading is 
a ``root of unity phenomenon'' and might lead 
to new insights about link and 
$3$-manifold invariants deduced from $\T$.
We also give a diagrammatic 
category for the (graded) projective 
endofunctors on $\T$, indicate how 
our results could generalize 
and collect some ``well-known'' 
facts to give a reasonably self-contained exposition.
\end{abstract}
\maketitle

\vspace*{-.5cm}

\tableofcontents

\vspace*{-.5cm}

\section{Introduction}\label{sec-intro}
\subsection{The framework}\label{sub-intropart1}

In this paper we study the quantum group 
$\Uq=\Uq(\mathfrak{sl}_2)$, where $q$ is 
a root of unity, its category of 
tilting modules $\T$ and the category 
of projective endofunctors $\Endo(\T)$ 
combinatorially and diagrammatically.

Everything we do is completely explicit and ``down to earth'', 
but motivated and deduced from a general machinery that 
comes, from the side of representation theory, from 
pioneering work of Soergel~\cite{soe1}, 
Beilinson-Ginzburg-Soergel~\cite{bgs}, 
Kazhdan-Lusztig~\cite{kalu} and 
Stroppel~\cite{st2}, and from the side of 
combinatorics and diagrammatics, from work of
Soergel~\cite{soe2}, Khovanov-Lauda~\cite{kl5} 
and Elias-Khovanov~\cite{ek1}.

Let us motivate and explain our approach.

\subsubsection{Quantum groups at roots of unity: non-semisimplicity, modular representation theory and affine Weyl groups}\label{subsub-intropart1}
Fix a simple complex Lie algebra $\mathfrak{g}$. 
Then the finite-dimensional\footnote{If not otherwise stated: 
modules in this paper are assumed to be finite-dimensional. 
There are only few exceptions in this paper, e.g. $\Tn{\bbin}$ 
from Definition~\ref{defn-tiltgen}. We hope it is clear from the context.} representation 
theory of the quantum deformation $\Uv(\mathfrak{g})$ 
of $\Uu(\mathfrak{g})$ is, for \textit{generic} parameter $v$, 
semisimple and very similar to the classical representation theory for $\Uu(\mathfrak{g})$.

This drastically changes when specializing $v$ 
to an $l$-th (we allow any $l>2$ throughout the paper) 
root of unity $q$: the representation theory 
of $\Uq(\mathfrak{g})$ is non-semisimple. 
This is mostly due to the fact that the so-called 
Weyl modules at roots of unity are, in 
general, \textit{not simple} and filtrations by 
Weyl modules \textit{do not necessarily split}. 
A lot of questions remain open about the 
representation theory at roots of unity. 
In fact, the representation 
theory of $\Uq(\mathfrak{g})$ over $\mathbb{C}$ has 
many similarities to the representation theory 
of a corresponding almost simple, simply connected 
algebraic group $G$ over an algebraically closed 
field $K$ of prime characteristic, see 
for example~\cite{ajs} or~\cite{lu1}.

It turns out, when studying the 
representation theory of $\Uq(\mathfrak{g})$, 
a certain category of \textit{tilting modules} $\T$ 
comes up naturally. The category $\T$ is inspired 
by the corresponding category of tilting modules 
for reductive algebraic groups due to Donkin~\cite{don} 
(see also Ringel~\cite{ringel}) and shares most of its 
properties, see for example~\cite{an1}. It is our main object under study in this paper.

Note that the ``combinatorics'' (the irreducible characters) 
of $G$ and $\Uq(\mathfrak{g})$ was conjectured by 
Lusztig (see~\cite{lu3} for $G$ and~\cite{lu1} for 
$\Uq(\mathfrak{g})$) to be related to values at $1$ 
of the Kazhdan-Lusztig polynomial associated to the 
\textit{affine} Weyl group for $\mathfrak{g}$. In addition, 
Kazhdan and Lusztig proved later that the category 
of finite-dimensional $\Uq(\mathfrak{g})$-modules 
(of type $1$) is equivalent to a category of 
modules for the corresponding affine Kac-Moody algebra, see~\cite{kalu}.

This combined with the solution of Kashiwara 
and Tanisaki~\cite{kt} of the Kazhdan-Lusztig 
conjecture in the affine Kac-Moody case then 
solved the above mentioned conjecture for the 
irreducible characters of $\Uq(\mathfrak{g})$. 
Even closer related to our work: Soergel first 
conjectured in~\cite{soe3} and later proved 
in~\cite{soe4} a corresponding statement about 
indecomposable tilting modules. In our 
little $\mathfrak{sl}_2$ case we do not need 
these deep results because we can work out 
both, the irreducible characters and the 
indecomposable tilting modules, \textit{``by hand''}.

\subsubsection{Categorification and graded categories}\label{subsub-intropart3}

A ground-breaking development towards 
giving an 
\textit{algebraic} proof of the Kazhdan-Lusztig conjectures 
was initiated by Soergel in~\cite{soe2}. He defines a 
combinatorial category $\mathcal S$ consisting of 
objects that are bimodules over a polynomial ring $R$. 
These bimodules are nowadays commonly called 
Soergel bimodules and are indecomposable 
direct summands of tensor products of certain 
algebraically/combinatorially defined $R$-bimodules.

His category is additive, monoidal and graded 
and he proves that its Grothendieck group $K_0$ 
is isomorphic to an integral form of the Hecke algebra 
$H_v(W)$ associated to the Weyl group $W$ of the simple 
Lie algebra $\mathfrak{g}$ in question. Here the grading 
and the corresponding shifting functors give on the level 
of Grothendieck groups rise to the indeterminate $v$ of $H_v(W)$.
Thus, we can say that Soergel's has 
\textit{categorified} $H_v(W)$ (actually, the 
categorification works for any Coxeter group $W$ 
and its associated Hecke algebra $H_v(W)$).

In fact, in the spirit of categorification 
outlined by Crane and Frenkel in the mid-nineties, 
\textit{graded} categories $\mathcal C$ 
(or \textit{graded} $2$-categories) give rise to a structure 
of a $\bZ[v,v^{-1}]$-module on $K_0(\mathcal C)$ 
(where ``shifting decategorifies'' to multiplication by $v$). 
Many examples of this kind of categorification are 
known. For example, Khovanov's categorification 
of the Jones polynomial (sometimes also called 
$\mathfrak{sl}_2$-polynomial)~\cite{kh1}, Khovanov-Rozansky's categorification 
of the $\mathfrak{sl}_n$-polynomial in~\cite{kr1} (all of these, 
although originally defined differently, 
can be obtained using $2$-functors on graded $2$-categories), 
or Khovanov-Lauda and Rouquier's categorification of 
$\Uv(\mathfrak{g})$ and its highest weight modules, 
see for example~\cite{kl5} and~\cite{rou} using the so-called 
Khovanov-Lauda Rouquier algebra, are among 
the more popular ones and have opened new directions of research.

Thus, it is natural to ask if we can introduce 
a \textit{non-trivial grading} on $\T$ as well. 
We do this in Section~\ref{sec-quiver} by using 
an argument pioneered by Soergel (see~\cite{soe1}) 
in the ungraded and Stroppel (see~\cite{st2}) in 
the graded case for category $\mathcal O$. Namely, 
the usage of Soergel's combinatorial functor $\mathbb{V}$ 
that gives rise to an equivalence of a block of 
$\mathcal O$ (for $\mathfrak{g}$) and a certain 
full subcategory of $\ModA$. The algebra $A$ is 
the endomorphism ring of the \textit{anti-dominant projective} 
in the block and it can be explicitly (when the block is regular) 
identified with the algebra of coinvariants for the Weyl group 
associated to $\mathfrak{g}$. This algebra can be given a 
\textit{$\bZ$-grading} and, as Stroppel explains in~\cite{st2}, 
this set-up gives rise to \textit{graded} versions of 
blocks of category $\mathcal O$ and the categories of 
\textit{graded} endofunctors on these blocks.
In fact, as Stroppel explains in~\cite{st2}, her 
approach is a combinatorial alternative to the 
approach of Beilinson, Ginzburg and Soergel given in~\cite{bgs}.

As in the other cases above, the \textit{grading} 
is the crucial point: in category $\mathcal O$ the 
multiplicity $[\Vn{\lambda}:\Ln{\mu}]$ of the 
simple module $\Ln{\mu}$ inside of the Verma 
$\Vn{\lambda}$ is given by evaluating the 
corresponding Kazhdan-Lusztig polynomial at $1$. 
The Kazhdan-Lusztig polynomial is a polynomial 
and not just a number and the grading of $\mathcal O$ 
``explains'' now the \textit{individual coefficients} of these polynomials as well.

In our case the role of $A$ is, as we explain 
in Section~\ref{sec-quiver}, played by an 
``infinite version'' $\Ai$ of a quiver algebra 
$\Am$ that Khovanov and Seidel introduced in~\cite{ks1} 
in their study of Floer homology\footnote{Around the 
same time this quiver algebra was considered 
independently by Braden in~\cite{bra}. We stay with Khovanov and Seidel's formulation 
(but write $\Am$ instead of $A_m$ as they do)
because the functors they considered on the category 
of representations of this algebra are closely 
connected to translation functors in our context.}. 
Its ``Koszul version'' appears in various contexts 
related to symplectic topology, algebraic geometry and representation theory. 
In particular, it appears as a subquotient of Khovanov's arc algebra that 
he introduced in~\cite{kh4} to give an algebraic structure 
underlying Khovanov homology and whose representation 
theory is known to be highly interesting as outlined 
in a series of papers by Brundan and Stroppel, 
see~\cite{bs1},~\cite{bs2},~\cite{bs3},~\cite{bs4} 
and~\cite{bs5}. The grading can be seen as coming from the 
corresponding grading of some Khovanov-Lauda Rouquier algebra and we use it 
to introduce the grading on $\T$ (for each block $\T_{\lambda}$) 
and we obtain a graded category $\Tgr$.

We stress that this grading is a ``\textit{root of unity phenomenon}'': 
the category of finite-dimensional $\Uv$-modules is 
semisimple and has therefore no interesting grading. 
On the other hand, the grading on $\Tgr$ is non-trivial 
and gives for example rise (as mentioned above) to a 
grading for similar modules of reductive algebraic groups over 
algebraical closed fields $K$ of prime characteristic.

Thus, an intriguing question is if one can use 
the grading on $\Tgr$ to obtain new information 
about link invariants or about the 
Witten-Reshetikhin-Turaev invariants 
of $3$-manifolds that can be deduced from $\T$, see~\cite{an1} 
or~\cite{tur}. Note 
that, as we deduce in Remark~\ref{rem-grothendieck1a}, 
each block $\Tgrl$ of $\Tgr$ decategorifies to the 
Burau representation of the braid group 
$B_{\infty}$ with $\infty$-many strands 
(the split Grothendieck group $K_0^{\oplus}$ carries an action of $B_{\infty}$), 
which already gives a hint how these topological invariants could be related to our work.

\subsubsection{Diagram categories, biadjoint functors and diagrammatic categorification}\label{subsub-intropart4}

Khovanov and Lauda have categorified 
$\Uv(\mathfrak{sl}_k)$, see~\cite{kl5}. 
Their approach, that has turned out to be very fruitful, was to 
use \textit{diagrammatic categorification}: 
they defined a certain $2$-category $\mathcal U(\mathfrak{sl}_k)$ 
consisting of a certain type of so-called string diagrams 
whose (split) Grothendieck group 
$K_0^{\oplus}(\mathcal U(\mathfrak{sl}_k))$ gives the idempotented, integral form $\Uvd(\mathfrak{sl}_k)_{\bZ}$ of $\Uv(\mathfrak{sl}_k)$. One of 
their main observations was that the $E$'s and $F$'s of 
$\Uvd(\mathfrak{sl}_k)_{\bZ}$ behave like biadjoint 
induction and restriction functors on certain categories. 
As outlined in an even more general framework by Khovanov 
in~\cite{kh5} (although it was folklore knowledge for some 
years and appears in a more rigorous form in for 
example~\cite{js} or~\cite{mueg}), biadjoint functors 
have a \textit{``built-in topology''} since, roughly, 
biadjointness means that we can straighten out diagrams.

Let us denote by $B_i$ the bimodules from Soergel's 
categorification of the Hecke algebra $H_v(W)$ from above.
A main feature of the $B_i$'s is 
that tensoring with $B_i$ is a self-adjoint endofunctor 
and, even stronger, a \textit{Frobenius object}, i.e. there are morphisms
\[
B_i\to R,\quad R\to B_i,\quad B_i\to B_i\otimes_R B_i\quad\text{and}\quad B_i\otimes_R B_i \to B_i
\]
pictured as (we read from bottom to top and right to left)
\[
\xy
(0,-3)*{\includegraphics[scale=0.8]{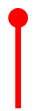}};
(0,8)*{\scriptstyle R};
(0,-9)*{\scriptstyle B_i};
\endxy\quad,\quad
\raisebox{-0.025cm}{\xy
(0,3)*{\rotatebox{180}{\includegraphics[scale=0.8]{res/figs/diacath/iso3}}};
(0,9)*{\scriptstyle B_i};
(0,-8)*{\scriptstyle R};
\endxy}\quad,\quad
\xy
(0,0)*{\includegraphics[scale=0.8]{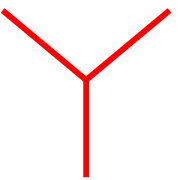}};
(-6.5,8)*{\scriptstyle B_i};
(0,8)*{\scriptstyle \otimes};
(6.5,8)*{\scriptstyle B_i};
(0,-9)*{\scriptstyle B_i};
\endxy
\quad,\quad
\xy
(0,0)*{\rotatebox{180}{\includegraphics[scale=0.8]{res/figs/diacath/iso4}}};
(-6.5,-9)*{\scriptstyle B_i};
(0,-9)*{\scriptstyle \otimes};
(6.5,-9)*{\scriptstyle B_i};
(0,8)*{\scriptstyle B_i};
\endxy
\]
that satisfy the \textit{Frobenius relations} (plus reflections of these)
\[
\text{\textbf{Frob1}}:\quad
\xy
(0,0)*{\includegraphics[scale=0.85]{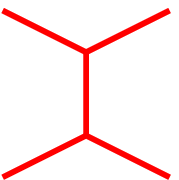}};
\endxy=
\xy
(0,0)*{\includegraphics[scale=0.85]{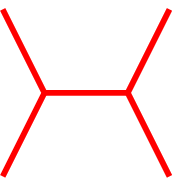}};
\endxy\quad,\quad\text{\textbf{Frob2}}:\quad
\xy
(0,0)*{\includegraphics[scale=0.85]{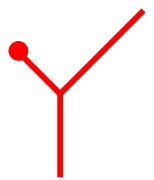}};
\endxy\;=\;
\xy
(0,0)*{\includegraphics[scale=0.85]{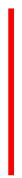}};
\endxy
\;=\;
\xy
(0,0)*{\includegraphics[scale=0.85]{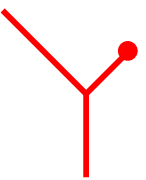}};
\endxy
\]
It is tempting to ask if one can give a 
\textit{diagrammatic} categorification in 
the spirit of $\Uu_v(\mathfrak{sl}_n)$-string 
diagrams of Khovanov-Lauda (see~\cite{kl5}) for
Soergel's categorification as well. The observations 
from above, as Khovanov explains in~\cite[Section~3]{kh5}, 
were the main reason why Elias and Khovanov started to look for such a description.

They were (very) successful in their search 
and their diagrammatic categorification given 
in~\cite{ek1} (in fact, the diagrams 
above are the ones they use) has inspired many successive works. 
Most mentionable for this paper: Elias' 
categorification $\D$ of the Hecke algebra 
$H_v(D_{\infty})$ 
(where $D_{\infty}$ is the infinite dihedral group) 
from~\cite{el1} called 
the \textit{dihedral cathedral} or, alternatively, the 
presentation in terms of
generators and relations 
of the same category given by Libedinsky~\cite{lib}.

This has already led to seminal results: 
as Elias and Williamson explain in~\cite[Subsection~1.3]{ew2}, 
their algebraic proof that the Kazhdan-Lusztig polynomials have positive 
coefficients for \textit{arbitrary} Coxeter 
systems was discovered using the diagrammatic framework 
from~\cite{ew1} and~\cite{ew3}.

In our context: the combinatorics of the 
blocks $\T_{\lambda}$, as explained in Subsection~\ref{sub-translation}, 
is mostly governed by two functors $\Theta_s$ 
and $\Theta_t$ called translation through the 
$s$ and $t$-wall respectively. Here 
$s$ and $t$ are the two reflections that generate 
the \textit{affine} Weyl group 
$W_l\cong D_{\infty}$ of $\mathfrak{sl}_2$.
These functors, motivated from the category $\mathcal O$ 
analogs, are \textit{biadjoint} and satisfy 
\textit{Frobenius relations}. Moreover, we show in 
Lemma~\ref{lem-adjoint} and Theorem~\ref{thm-struktursatz2} 
that this \textit{still holds} in the graded setting.
Thus, it seems reasonable to expect that 
$\Tgrl$ and its category of (graded) projective endofunctors 
$\Endo(\Tgrl)$ have a \textit{diagrammatic} 
description as well. And, since $W_l\cong D_{\infty}$, 
it seems reasonable to expect a relation to Elias' \textit{dihedral cathedral} 
$\D$ from~\cite{el1}. We prove this in Section~\ref{sec-dia} by 
introducing certain quotients of $\D$ giving rise to a diagrammatic presentation 
of $\Tgrl$ and $\Endo(\Tgrl)$.

\subsection{An outline of the paper}\label{sub-intropart2}
The paper is organized as follows.
\begin{itemize}
\item Section~\ref{sec-tilting} contains 
the definition and some basic properties of the category 
of $\Uq$-tilting modules $\T$ as well as the category of projective 
endofunctors of it. Most of Section~\ref{sec-tilting} is known, 
but we have included some new observations.
\item In Section~\ref{sec-quiver} we recall Khovanov-Seidel's $m$-quiver 
algebra $\Am$ and introduce the ``infinite version'' $\Ai$. 
We show that the regular blocks of 
$\T$ embed into the
module category of $\Ai$. 
Since $\Ai$ is naturally graded, we thus, define 
blockwise a grading on $\T$ and $\Endo(\T_{\lambda})$.
\item Section~\ref{sec-dia} finally provides the diagrammatic interpretation of the results from before.
That is, we introduce a diagram category $\QD$ which is a quotient of the diagram 
category studied by Elias~\cite{el1} and show that the additive closure of $\QD$ and 
$\Tgrl$ are equivalent as graded categories. Similarly, we enrich $\QD$ and obtain a 
diagrammatic description of the (graded) category of endofunctors as well.
\end{itemize}
Everything is very explicit and we try to 
illustrate this with examples along the way.

\begin{rem}\label{remark-colors}
We use colors in this paper. Some of them are not important and only 
for illustration. But in Section~\ref{sec-dia} we 
need red and green and it is crucial 
that they are different.
The reader 
who has a black-and-white version can distinguish these 
since green appears lightly shaded.
\end{rem}
\noindent \textbf{Acknowledgements:} We thank 
Christian Blanchet, Ben Elias, Gregor Masbaum, Catharina Stroppel, Pedro Vaz and 
Geordie Williamson for many helpful discussions. 
Special thanks to Geordie Williamson for pointing 
out a connection of the tilting category to an 
anti-spherical quotient of Elias' dihedral cathedral. We are also 
grateful to the anonymous referees for many careful comments and corrections.
D.T. thanks the Danish summer for, 
positively formulated, not distracting him from typing this paper.
\section{The tilting category \texorpdfstring{$\T$}{T}}\label{sec-tilting}
In this section we describe the category $\T$ of 
$\Uq=\Uq(\mathfrak{sl}_2)$-tilting modules for the quantum enveloping 
algebra\footnote{We like to work over the cyclotomic field $\bC$, but any field of characteristic zero containing $q$ would work.} of $\mathfrak{sl}_2$ at a fixed root of unity $q$.
We give a hopefully self-contained summary of the results in the $\mathfrak{sl}_2$ case 
since most results are either spread out over the 
literature or only mentioned implicitly. But 
we have also included some new observations related to our context.

We mostly follow~\cite[Chapters~1,~2,~3]{ja} with our 
notation, $v$ denotes an indeterminate, $q\in\bC$ denotes a 
fixed root of unity and all module categories 
in this section are categories of \textit{left} 
modules.

\subsection{Quantum groups at roots of unity}\label{sub-qroot}
Before we start, let us fix some notions.
Given $a\in\bZ$ and $b\in\bN$, let
\begin{gather*}
[a]=\frac{v^a-v^{-a}}{v-v^{-1}}=v^{a-1}+v^{a-3}+\dots+v^{-a+1}+v^{-a+1},\quad\quad[b]!=[1][2]\cdots [b-1][b],\\
\qbin{a}{b}=\frac{[a][a-1]\cdots[a-b+2][a-b+1]}{[b]!}\in\bZ[v,v^{-1}].
\end{gather*}
be the \textit{quantum integer}, the \textit{quantum factorial} and 
the \textit{quantum binomial}.
By convention, $[0]!=1$.

\begin{defn}\label{def-qv} The \textit{quantum special linear algebra} $\Uv(\mathfrak{sl}_2)$ 
is the associative, unital $\bQ(v)$-algebra 
generated by $K,K^{-1},E$ and $F$ subject to the relations
\begin{gather*}
KK^{-1}=K^{-1}K=1,\quad\quad KE=v^2EK,\quad\quad KF=v^{-2}FK,\quad\quad
EF-FE =\dfrac{K-K^{-1}}{v-v^{-1}}.
\end{gather*}
We write $\Uv=\Uv(\mathfrak{sl}_2)$ for short.
\end{defn}

Recall that $\Uv$ is a Hopf algebra with coproduct $\Delta$, 
antipode $S$ and the counit $\varepsilon$ given by
\begin{gather*}
\Delta(E)=E\otimes 1+K\otimes E,\quad\quad\Delta(F)=F\otimes K^{-1}+1\otimes F,\quad\quad\Delta(K)=K\otimes K,\\
S(E)=-K^{-1}E,\;\;\;\; S(F)=-FK,\;\;\;\; S(K)=K^{-1},\quad\quad\varepsilon(E)=\varepsilon(F)=0,\;\;\;\;\varepsilon(K)=1.
\end{gather*}
This allows to extend actions 
to tensor products and 
to duals, and there is a trivial $\Uv$-module.

We want to ``specialize'' the $v$ to be a root of unity $q\in\bC$. 
To this end, 
let $\mathrm{A}=\bZ[v,v^{-1}]$ and we consider 
Lusztig's $\mathrm{A}$-form $\Ua$ from~\cite{lu2}.

\begin{defn}(\textbf{Lusztig's $\mathrm{A}$-form $\Ua$})\label{def-qA} Define for 
$j\in\bZ_{>0}$ the \textit{$j$-th divided powers}
\[
E^{(j)}=\frac{E^{j}}{[j]!},\quad\quad F^{(j)}=\frac{F^{j}}{[j]!}.
\]
$\Ua$ is 
defined as the $\mathrm{A}$-subalgebra of 
$\Uv$ generated by $K,K^{-1}$, $E^{(j)}$ and $F^{(j)}$.
\end{defn}

\begin{defn}\label{def-qq} Fix a root of unity 
$q\in\bC$, $q\neq \pm 1$ and denote by 
$l$ the order\footnote{The square is only important for roots 
of unity of even order, but not 
for roots of unity of odd order (for these the order of $q^2$ is the same as 
the order of $q$), e.g. we have $l=3$ for third as well as for sixth roots of unity.} of $q^2$. Consider 
$\bC$ as an $\mathrm{A}$-module by specializing $v$ to $q$. Define
\[
\Uq=\Ua\otimes_{\mathrm{A}}\bC.
\]
We abuse notation and write $E^{(j)}$ instead 
of $E^{(j)}\otimes 1$. Analogously for the other generators.
\end{defn}

\begin{rem}\label{rem-indhopf}
It is easy to check that $\Ua$ is a Hopf 
subalgebra of $\Uv$. Thus, $\Uq$ inherits a Hopf algebra structure from $\Uv$. 
In particular, if one has a $\Uq$-module $M$, then 
$M^*=\Hom_{\Uq}(M,\bC)$ has the usual induced action. 
It follows that, if $m\in M$ is an eigenvector of $K$ with 
eigenvalue $\alpha$ and $m^*\in M^*$ a dual vector with respect to 
some $K$-stable complement of $\bC m$ in $M$, then
\begin{equation}\label{eq-eigen}
Km=\alpha m\Longleftrightarrow Km^*=\alpha^{-1}m^*.
\end{equation}
Moreover, note that $[j]=0\in\Uq$ iff $l|j$. 
This implies 
$E^l=[l]!E^{(l)}=0$ and $F^l=[l]!F^{(l)}=0$. 
It is also true that $K^{2l}=1$, see~\cite[Lemma~4.4~(a)]{lu1}.
\end{rem}

\subsection{Weyl modules, dual Weyl modules and simple modules}\label{sub-wsmod}

\begin{defn}(\textbf{Weyl, dual Weyl and simple modules})\label{defn-weyl}
Let $i\in\bN$ and denote 
by $\Vn i$ the $i$\textit{-th Weyl module}. 
This is the $\bC$-vector space 
with basis $m_0,\dots,m_i$ and an $\Uq$-action defined by
\[
Km_k=q^{i-2k}m_k,\quad\quad E^{(j)}m_k=\qbin{i-k+j}{j}m_{k-j},\quad\quad F^{(j)}m_k=\qbin{k+j}{j}m_{k+j},
\]
with the convention that $m_{<0}=m_{>i}=0$. 
The \textit{$i$-th dual Weyl module}, 
denoted by $\dVn i$, is obtained from 
$\Vn i$ by taking the dual, that is, 
$\Hom_{\mathbb{Q}(q)}(\Vn i,\bC)=(\Delta_q(i))^*=\dVn i$.
\end{defn}

It is easy to check directly that 
$\Vn i$ has a \textit{unique simple head} $\Ln i$ which 
is also the \textit{unique simple socle} of $\dVn i$. 
See also~\cite[Section~4]{apw}.

\begin{ex}\label{ex-weyl}
Let $q=\frac{-1+\sqrt{-3}}{2}$ (we use this $q$ in all examples with $l=3$ in what follows).

The Weyl module $\Vn 3$ can be visualized as
\begin{equation}\label{eq-example1}
\xymatrix{
  m_3  \ar@<2pt>[r]^{+1}\ar@(ul,ur)|{q^{-3}}  &  {\color{colormy}m_2}  \ar@[mycolor]@<2pt>[l]^{{\color{mycolor} 0}}\ar@[colormy]@<2pt>[r]^{{\color{colormy} -1}}\ar@[colormy]@(ul,ur)|{{\color{colormy}q^{-1}}} & {\color{colormy}m_1}  \ar@[mycolor]@<2pt>[r]^{{\color{mycolor} 0}}\ar@[colormy]@<2pt>[l]^{{\color{colormy}-1}}\ar@[colormy]@(ul,ur)|{{\color{colormy}q^{+1}}}  &  m_0.  \ar@<2pt>[l]^{+1}\ar@(ul,ur)|{q^{+3}} \\
}
\end{equation}
Here the action of 
$E$ points right, the action of 
$F$ left, $E^{(3)}m_3=m_0$, $F^{(3)}m_0=m_3$, $K$ acts as a loop 
and all other actions are zero. 
Thus, the $\mathrm{A}$-span of $\{m_1,m_2\}$ is now, 
in contrast to the classical case, stable under the action of $\Uq$. The complement 
however is not an $\Uq$-submodule.

Another example (where we have excluded the actions of the divided powers) is $\Vn{4}$:
\begin{equation}\label{eq-example2}
\xymatrix{
  m_4  \ar@<2pt>[r]^{+1}\ar@(ul,ur)|{q^{-4}}  &  m_3  \ar@<2pt>[l]^{+1}\ar@<2pt>[r]^{-1}\ar@(ul,ur)|{q^{-2}} & {\color{green} m_2}  \ar@[mycolor]@<2pt>[r]^{{\color{mycolor} 0}}\ar@[red]@<2pt>[l]^{{\color{mycolor} 0}}\ar@[colormy]@(ul,ur)|{{\color{colormy}q^{0}}} & m_1  \ar@<2pt>[r]^{+1}\ar@<2pt>[l]^{-1}\ar@(ul,ur)|{q^{+2}} &  m_0.  \ar@<2pt>[l]^{+1}\ar@(ul,ur)|{q^{+4}} \\
}
\end{equation}
We note that $\Vn{4}$ has the 
trivial $\Uq$-module spanned by $m_2$ as a $\Uq$-submodule.
\end{ex}

$K$ acts on $\Vn i$ via the eigenvalues $q^{-i},q^{-i+2},\dots,q^{i-2},q^{+i}$. 
Thus, by~\eqref{eq-eigen}, the same is 
true for $\dVn i$. Moreover, the 
$\Ln i$ are self-dual, see e.g.~\cite[Section~4]{apw}.

\begin{prop}\label{prop-weyl}
We have the following.
\begin{itemize}
\item[(a)] $\Vn i\cong\Ln i$ iff $i<l$ or $i\equiv -1\mod l$.
\item[(b)] Suppose $i=al+b$ for 
some $a,b\in\bN$ with $b\leq l-2$. 
Set $i^{\prime}=(a+2)l-b-2$. Then there exists a short exact sequence
\[
0\longrightarrow \Ln i\hooklongrightarrow\Vn{i^{\prime}}\twoheadlongrightarrow \Ln{i^{\prime}}\longrightarrow 0.
\]
Moreover, $\Ln{i^{\prime}}$ is the 
head and $\Ln i$ is the socle of $\Vn{i^{\prime}}$.\makeqed
\end{itemize}
\end{prop}

\begin{proof}
This is~\cite[Corollary~4.6]{apw}.
\end{proof} 
\begin{cor}\label{cor-weyl}
We have the following.
\begin{itemize}
\item[(a)] $\dVn i\cong\Ln i$ iff $i<l$ or $i\equiv -1\mod l$.
\item[(b)] Suppose $i=al+b$ for some 
$a,b\in\bN$ with $b\leq l-2$. Set $i^{\prime}=(a+2)l-b-2$. 
Then there exists a short exact sequence
\[
0\longrightarrow L_q(i^{\prime})\hooklongrightarrow\nabla_q(i^{\prime})\twoheadlongrightarrow \Ln i\longrightarrow 0.
\]
Moreover, $\Ln{i^{\prime}}$ is the socle and 
$\Ln i$ is the head of $\dVn{i^{\prime}}$.\makeqed
\end{itemize}
\end{cor}

\begin{proof}
This follows from $\Ln i\cong(\Ln i)^*$ and 
the fact that ${}^*$ is an exact, contravariant functor.
\end{proof}

\begin{ex}\label{ex-weyl2}
For $i=0$ and $l=3$ we have $i^{\prime}=4$. 
The trivial $\Uq$-module $\Ln 0$ appears as a 
submodule of $\Vn 4$, compare to~\eqref{eq-example2}.
\end{ex}

We should note here that 
there are two different 
types of $\Uq$-modules known as \textit{type} $1$ \textit{and} $-1$, 
see for example~\cite[Section 1]{apw}. 
For us the difference between the two types is not important 
in this paper and we \textit{only} consider $\Uq$-modules of type $1$. 
The (more general) treatment in~\cite[Section 1]{apw} 
ensures that the consideration of only type 
$1$ is still enough to get results for both types.

\begin{cor}\label{cor-simples}
The set $\{\Ln i\mid i\in\bN\}$ is a  
complete set of simple, pairwise non-isomorphic, 
finite-dimensional $\Uq$-modules (of type $1$).\makeqed
\end{cor}

\begin{proof}
See~\cite[Corollary~6.3]{apw}.
\end{proof}

\begin{rem}\label{rem-semisimple}
The category of $\Uq$-modules is 
\textit{far} from being semisimple: Proposition~\ref{prop-weyl} 
says that $\Vn i$ is never simple whenever we are in the case (b). 
See also~\eqref{eq-example1}.
\end{rem}

\subsection{Tilting modules and the tilting category \texorpdfstring{$\T$}{T}}\label{sub-tiltmod}
\begin{defn}(\textbf{$\Delta$- and $\nabla$-filtration})\label{defn-filt}
We say that a $\Uq$-module $M$ has 
a $\Delta$\textit{-filtration} if there exists a descending sequence of 
$\Uq$-submodules
\[
M=F_{0}\supset F_1\supset\cdots\supset F_i\supset\dots,\quad\quad\bigcap_{i=0}^{\infty}F_i=0,
\]
such that for all $i=0,1\dots$ we have 
$F_{i}/F_{i+1}\cong \Vn{i^{\prime}}$ for 
some $i^{\prime}\in\bN$. A $\nabla$\textit{-filtration} is 
defined similarly, but using 
$\dVn{i^{\prime}}$ instead of 
$\Vn{i^{\prime}}$ and an ascending sequence of $\Uq$-submodules, that is
\[
0=F_{0}\subset F_1\subset\cdots\subset F_i\subset\dots,\quad\quad\bigcup_{i=0}^{\infty}F_i=M,
\]
such that for all $i=0,1\dots$ we 
have $F_{i+1}/F_{i}\cong \dVn{i^{\prime}}$ for some $i^{\prime}\in\bN$.
\end{defn}

One can prove, following similar arguments as in~\cite[Proposition~II.4.16]{jarag},
that such filtrations are unique up to reordering. It is clear that 
a finite-dimensional $\Uq$-module $M$ has a 
$\Delta$-filtration iff $M^*$ has a $\nabla$-filtration.

\begin{defn}(\textbf{Tilting modules})\label{defn-tilt}
We call a $\Uq$-module $M$ a $\Uq$-\textit{tilting module} 
if it has a $\Delta$- and a $\nabla$-filtration. 
We say for short that $M$ is \textit{tilting}.
\end{defn}

\begin{ex}\label{ex-tilt}
Recall that the category of finite-dimensional $\Uv$-modules is 
semisimple. It follows that all finite-dimensional 
$\Uv$-modules are tilting. Moreover, for $i<l$ or $i\equiv -1 \mod l$, 
by Proposition~\ref{prop-weyl} and 
Corollary~\ref{cor-weyl}, we see that $\Vn i\cong\Ln i\cong\dVn i$ is tilting.
\end{ex}

It is easy to see 
(in our $\mathfrak{sl}_2$ case - in general this 
is non-trivial, see e.g.~\cite[Theorem~3.3]{par}) 
that $\Vn i\otimes_{\bC}\Vn{i^{\prime}}$ has a $\Delta$-filtration with factors 
(likewise for $\dVn i\otimes_{\bC}\dVn{i^{\prime}}$)
\begin{equation}\label{eq-tiltfact}
\Vn{|i-i^{\prime}|},\,\Vn{|i-i^{\prime}|+2},\,\dots, \Vn{i+i^{\prime}-2},\,\Vn{i+i^{\prime}},\quad\quad i,i^{\prime}\geq 0.
\end{equation}

\begin{defn}(\textbf{Full tilting category})\label{defn-ftiltcat}
We denote by $\T^{\mathrm{all}}$ the full 
subcategory of all $\Uq$-modules that are tilting. 
We denote its class of objects by $\Ob(\T^{\mathrm{all}})$.
\end{defn}

\begin{prop}\label{prop-basicp}
The category $\T^{\mathrm{all}}$ has the following properties.
\begin{itemize}
\item[(a)] If $M,M^{\prime}\in\Ob(\T^{\mathrm{all}})$, then $M\oplus M^{\prime}\in\Ob(\T^{\mathrm{all}})$.
\item[(b)] If $M\oplus M^{\prime}\in\Ob(\T^{\mathrm{all}})$, then $M,M^{\prime}\in\Ob(\T^{\mathrm{all}})$.
\item[(c)] If $M,M^{\prime}\in\Ob(\T^{\mathrm{all}})$, then $M\otimes_{\bC} M^{\prime}\in\Ob(\T^{\mathrm{all}})$.\makeqed
\end{itemize}
\end{prop}

\begin{proof}
Part (a) is immediate. For part (b) one can, for 
example, argue as in~\cite[Proposition~3.7]{hum}, and the non-trivial 
part (c) can be deduced from~\eqref{eq-tiltfact} 
and the corresponding statement for the dual Weyl modules $\dVn i$.
\end{proof}

\begin{defn}(\textbf{The tilting category $\T$})\label{defn-tiltcat}
Let $\T$ be the full subcategory of $\T^{\mathrm{all}}$ consisting of:
\begin{itemize}
\item The objects $\Ob(\T)$ are all 
finite-dimensional $\Uq$-tilting modules $M\in\Ob(\T^{\mathrm{all}})$.
\item The morphisms $\Hom_{\T}(M,N)$ are all $\Uq$-intertwiners $f\in\Hom_{\Uq}(M,N)$. 
\end{itemize}
The objects of $\T$ have finite filtrations, 
i.e. $F_N=F_{N+i}$ for some $N\in\bN$ and all $i\in\bN$.
\end{defn}

Recall that an additive category is \textit{Krull-Schmidt}, 
if each object can be uniquely decomposed (up to permutation) 
into a finite direct sum of \textit{indecomposable objects}, i.e. objects 
$O$ such that $O\cong O^{\prime}\oplus O^{\prime\prime}$ implies that 
$O^{\prime}\cong 0$ or $O^{\prime\prime}\cong 0$.

\begin{lem}\label{lem-tiltcat}
We have the following.
\begin{itemize}
\item[(a)] The category $\T$ is additive (but not abelian), closed under 
finite direct sums, finite tensor products and under duals.
\item[(b)] $\T$ is a Krull-Schmidt category whose 
indecomposable objects are parametrized by $\bN$.
\item[(c)] $\Vn i$, $\dVn i$ and $\Ln i$ are in $\T$ 
iff $i<l$ or $i\equiv-1\mod l$.\makeqed
\end{itemize}
\end{lem}

\begin{proof}
%
\textbf{(a).} From Proposition~\ref{prop-basicp} and the fact that 
$\Hom_{\Uq}(\cdot,\bC)$ commutes with finite sums.

\textbf{(b).} 
The Krull-Schmidt property follows from finite-dimensionality and 
(b) of Proposition~\ref{prop-basicp}.
By Proposition~\ref{prop-tilt} below, all finite-dimensional 
indecomposable tiltings are of the form $\Tn i$.

\textbf{(c).} A direct consequence of Proposition~\ref{prop-weyl} and Corollary~\ref{cor-weyl}.
\end{proof}

\begin{rem}\label{rem-ribbon}
It follows from Lemma~\ref{lem-tiltcat} part (a) 
that $\T$ is a \textit{rigid category} 
(a monoidal category with duals and certain compatibility properties). 
Moreover, $\T$ is even a \textit{ribbon (tensor) category} and 
$\T$ gives rise to a \textit{modular category} 
(roughly: one mods out by tiltings whose quantum trace is zero) 
and thus, gives a $2+1$-dimensional TQFT and can 
be used to define the \textit{Witten-Reshetikhin-Turaev invariants} of $3$-manifolds. 
Good treatments of this are~\cite[Section~4]{an1} or~\cite[Section~7]{saw}.
\end{rem}

Define, using Propositions~\ref{prop-weyl} 
and~\ref{prop-basicp} and~\eqref{eq-tiltfact}, 
a family $(\Tn i)_{i\in\bN}$ of indecomposable tiltings as follows. 
We start by setting $\Tn 0=\Vn 0$ and $\Tn 1=\Vn 1$. 
Then we denote by $m_q\in\Tn 1$ any eigenvector for $K$ with eigenvalue $q$. 
For each $i>1$ we define $\Tn i$ to be the indecomposable 
summand of $(\Tn 1)^{\otimes i}$ which contains 
the vector $m_q\otimes\cdots\otimes m_q\in (\Tn 1)^{\otimes i}$.

Denote by $(T:\Vn{i})\in\bN$ the \textit{filtration multiplicity} for an $\Uq$-tilting module $T$. It is clear, by using~\eqref{eq-tiltfact}, that
\begin{equation}\label{eq-tiltfact2}
((\Tn 1)^{\otimes i}:\Vn i)=1,\quad\quad((\Tn 1)^{\otimes i}:\Vn{i^{\prime}})=0,\quad\quad\text{for } i^{\prime}>i.
\end{equation}
Hence, $\Tn i$ may also be described as 
the unique indecomposable summand of 
$(\Tn 1)^{\otimes i}$ that contains $\Vn i$. We note the following more precise statement.

\begin{prop}\label{prop-tilt}
\item[(a)] $\Vn i\cong\Ln i\cong\Tn i\cong\dVn i$ iff $i<l$ or $i\equiv -1\mod l$.
\item[(b)] Suppose $i=al+b$ for some $a,b\in\bN$ with 
$b\leq l-2$. Set $i^{\prime}=(a+2)l-b-2$. Then there exist short exact sequences
\[
0\longrightarrow \Vn{i^{\prime}}\hooklongrightarrow\Tn{i^{\prime}}\twoheadlongrightarrow \Vn{i}\longrightarrow 0,\quad\quad
0\longrightarrow \dVn{i}\hooklongrightarrow\Tn{i^{\prime}}\twoheadlongrightarrow \dVn{i^{\prime}}\longrightarrow 0.
\]
\item[(c)] We have $\Tn i\cong \Tn{i^{\prime}}$ 
iff $i=i^{\prime}$. Moreover, if $M\in\Ob(\T)$ 
is indecomposable, then there 
exists an $i\in\bN$ such that $M\cong \Tn i$.\makeqed
\end{prop}

\begin{proof}
\textbf{(a).} Clear by Proposition~\ref{prop-weyl} and Corollary~\ref{cor-weyl}.

\textbf{(b).} First let $T=\Vn{(a+1)l-1}\otimes \Vn 1$.  
Then $T=\Tn{(a+1)l-1}\otimes \Tn 1$ by (a), and, by~\eqref{eq-tiltfact}, 
we see that $T$ has a $\Delta_q$-filtration giving us the short exact sequence
\[
0\longrightarrow \Vn{(a+1)l}\hooklongrightarrow T\twoheadlongrightarrow \Vn{(a+1)l-2}\longrightarrow 0.
\] 
By duality, $T$ has an analogous $\nabla_q$-filtration, showing that $T$ is tilting. 
It is also indecomposable since otherwise the summand $\Tn{(a+1)l}$ 
of $T$ would be equal to $\Vn{(a+1)l}$, forcing it to be simple, which contradicts 
(b) of Proposition~\ref{prop-weyl}.
This proves the claim in case $b=l-2$.

Now, if $b<l-2$, then consider the diagram
\[
\xy
\xymatrix{
0\ar[r] & \Vn{i^{\prime}}\ar@{^{(}->}[r] & \Vn{i^{\prime}-1}\otimes \Vn 1\ar@{->>}[r]\ar@{^{(}->}[d] & \Vn{i^{\prime}-2}\ar[r] & 0\\
 &  & \Tn{i^{\prime}-1}\otimes \Tn 1\ar@{->>}[d] &  & \\
0\ar[r] & \dVn{i^{\prime}-2}\ar@{^{(}->}[r] & \dVn{i^{\prime}-1}\otimes \dVn 1\ar@{->>}[r] & \dVn{i^{\prime}}\ar[r] & 0.\\
}
\endxy
\]
Here the bottom and top sequences come from~\eqref{eq-tiltfact}. 
A computation shows that these split. This 
gives an inclusion $\Vn{i^{\prime}-2}\hookrightarrow\Tn{i^{\prime}-1}\otimes \Tn 1$ 
which extends to $\Tn{i^{\prime}-2}$, and a surjection 
$\Tn{i^{\prime}-1}\otimes \Tn 1\twoheadrightarrow\dVn{i^{\prime}-2}$ 
factoring through $\Tn{i^{\prime}-2}$. The resulting endomorphism of $\Tn{i^{\prime}-2}$
is non-zero on the 
$i^{\prime}-2$ weight space (the reader unfamiliar 
with this notion might consider Remark~\ref{rem-tria} 
where the notation is $M_i$ for the $i$-th weight space) 
and thus, an isomorphism. 
Hence, 
$\Tn{i^{\prime}-1}\otimes \Tn 1\cong\Tn{i^{\prime}-2}\oplus T$ 
for some $\Uq$-module $T$, which 
is tilting by (a) of Lemma~\ref{lem-tiltcat}. 
Using an inductive argument, we have that $T$ has two $\Delta_q$- 
and $\nabla_q$-factors as in the statement of (b). As above, we see that 
$T=\Tn{i^{\prime}}$, which proves the claim.

\textbf{(c).} To see the first statement note that it suffices to show 
that $\Tn i\not\cong \Tn{i^{\prime}}$ for $i\neq i^{\prime}$. To this end, assume without 
loss of generality that $i^{\prime}>i$. 
By~\eqref{eq-tiltfact2}, we see 
that $\Vn{i^{\prime}}$ appears with multiplicity 
$1$ in $\Tn{i^{\prime}}$, but not in $\Tn i$. Thus,  
$\Tn i\not\cong \Tn{i^{\prime}}$.
That every finite-dimensional 
indecomposable tilting is of this 
form needs a little bit more treatment. 
But it is a standard argument that appears 
in various contexts and can be 
adopted for example 
from~\cite[Section~11.2]{hum} or, alternatively, 
from~\cite[Chapter~II, Section~E.6]{jarag}.
\end{proof}

\begin{ex}\label{ex-tilt1}
Let $l=3$ again. The $\Uq$-tilting module $\Tn 2$ can be 
visualized as
\begin{equation*}
\begin{gathered}
\xymatrix{
 \otimes \ar@{.}[r]\ar@{.}[d]& m_1 \ar@<2pt>[r]^{1}\ar@(ul,ur)|{q^{-1}}  &  m_0  \ar@<2pt>[l]^1\ar@(ul,ur)|{q^{+1}}  \\
 m_1\ar@<2pt>[d]^{1}\ar@(ul,dl)|{q^{-1}} & m_{11} \ar@<2pt>[r]^{1}\ar@<2pt>[d]^{1}\ar@(ul,dl)|{q^{-2}}  &  m_{01}  \ar@<2pt>[l]^1\ar@<2pt>[d]^{1}\ar@{.}[dl]|{+}\ar@(ur,dr)|{q^{0}}  \\
 m_0\ar@<2pt>[u]^{1}\ar@(ul,dl)|{q^{+1}} & m_{10} \ar@<2pt>[r]^{1}\ar@<2pt>[u]^{1}\ar@(ul,dl)|{q^{0}}  &  m_{00}  \ar@<2pt>[l]^1\ar@<2pt>[u]^{1}\ar@(ur,dr)|{q^{+2}}  \\
}
\end{gathered},
\end{equation*}
where we use $m_{ij}=m_i\otimes m_j$. 
By construction, $\Tn 2$ contains $m_{00}$ and it therefore has to be the span of $\{m_{00},q^{-1}m_{10}+m_{01},m_{11}\}$ as indicated above (which is isomorphic to $\Ln 2$).
\end{ex}

\begin{rem}\label{rem-proj}
One can 
show that the category of all 
finite-dimensional projective $\Uq$-modules
is a full 
subcategory of $\T$. 
The same is true for finite-dimensional 
injective $\Uq$-modules: $\T$ has enough injectives and projectives. 
In fact, $\Tn{i}$ is injective and projective 
for all $i\geq l-1$. Hence, 
$\{\Tn{i}\mid i\geq l-1\}$ is a complete set of 
pairwise non-isomorphic, projective-injective
indecomposable, finite-dimensional $\Uq$-modules. 
See for example~\cite[Section~5]{an1}
\end{rem}

\subsection{The linkage principle and blocks}\label{sub-blocks}
Consider the alcove $\mathcal A_0$ 
and its closure $\bar{\mathcal A}_0$ given by
\[
\mathcal A_0=\{k\in\bZ\mid -1<k<l-1\},\quad\quad\bar{\mathcal A}_0=\{k\in\bZ\mid -1\leq k \leq l-1\}.
\]
We call $\mathcal A_0$ the \textit{fundamental alcove}. 
Any other alcove in $\bZ_{\geq -1}$ is 
clearly of the form $\mathcal A_0+il$ for 
some $i\in\bN$. Moreover, we call 
$-1,l-1\in\bar{\mathcal A}_0-\mathcal A_0$ \textit{walls of $\mathcal A_0$}.

The \textit{affine Weyl group} is 
$W_l=\langle s_{r},t_{r}\mid r\in\bZ\rangle$ where $s_{r}$ and $t_r$ act on $\bZ$ via
\[
s_{r}.k=-k-2+(4r+2)l,\quad\quad t_{r}.k=-k-2+4rl,\quad k\in\bZ
\]
Note that $W_l$ is isomorphic 
to the infinite dihedral group $D_{\infty}$: 
we have $W_l\cong \bZ/2\bZ\ltimes l\bZ\cong D_{\infty}$ 
(with $\bZ/2\bZ$ being the Weyl group of $\mathfrak{sl}_2$) and we denote 
the two generators of $D_{\infty}$ by $s$ and $t$ (with the evident association). 
They act ``alcove-wise'' as indicated in~\eqref{eq-alcove}.

We denote by $W_l.x$ the orbits 
in $\bZ_{\geq -1}$ under the 
action of the affine Weyl 
group acting on $x\in\bZ_{\geq -1}$. 
For $l=3$ this can be visualized as
\begin{equation}\label{eq-alcove}
\xy
(0,0)*{\includegraphics[scale=0.75]{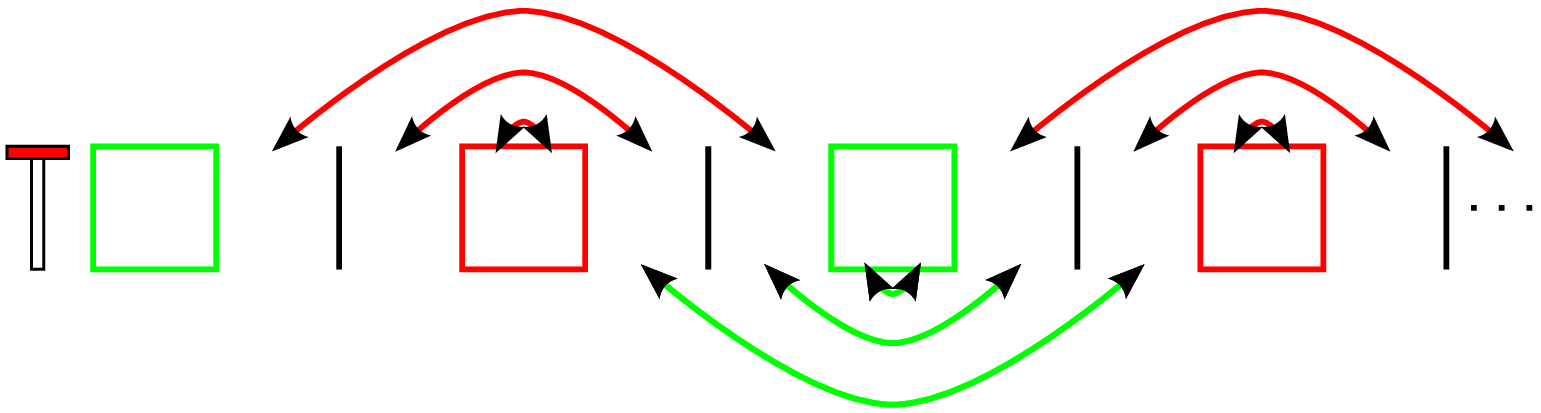}};
(-47,0)*{-1};
(-37.5,0)*{0};
(-28,0)*{1};
(-18.5,0)*{2};
(-9.75,0)*{3};
(0,0)*{4};
(9.5,0)*{5};
(18.5,0)*{6};
(28,0)*{7};
(37.5,0)*{8};
(47,0)*{9};
(-18.5,-7.5)*{s\text{-wall}};
(37.5,-7.5)*{s\text{-wall}};
(9.5,7.5)*{t\text{-wall}};
(-18.5,12)*{r=0};
(37.5,12)*{r=1};
(9.5,-12)*{r=1};
(-50,-7.5)*{\text{dead-end}};
\endxy
\end{equation}
with red (or top) action 
by $s$ and green (or bottom) 
action by $t$.

We say $i\in\bN$ is \textit{linked} to $i^{\prime}\in\bN$ 
if there exists $w\in W_l$ such that $w.i=i^{\prime}$. We, by 
convention, set all $\Uq$-modules indexed by negative numbers to be zero.

The following is known as the 
\textit{$\mathfrak{sl}_2$-linkage principle}. 
The more 
general (and highly non-trivial) 
statement can be found in~\cite[Corollaries~4.4 and~4.6]{an2}.

\begin{thm}(\textbf{The linkage principle})\label{thm-link}
All composition factors of $\Tn i$ have maximal weights $i^{\prime}$ linked to 
$i$. Moreover, $\Tn i$ is a simple $\Uq$-module if $i\in W_l.(\bar{\mathcal A}_0-\mathcal{A}_0)$.

In addition, if $i$ is linked to an element of $\mathcal{A}_0$, then 
$\Tn i$ is a simple $\Uq$-module iff $i\in\mathcal{A}_0$.\makeqed
\end{thm}

\begin{proof}
Use Propositions~\ref{prop-weyl} and~\ref{prop-tilt} and Corollary~\ref{cor-weyl}.
\end{proof}

\begin{nota}\label{no-link}
We \textit{always} use $\lambda$ for 
elements in the fundamental alcove 
$\mathcal{A}_0$ and $\mu$ for elements 
on the walls of the fundamental alcove 
$\bar{\mathcal A}_0-\mathcal A_0$. Moreover, set 
\[
\mathcal A_i=\{i^{\prime}\mid il-1<i^{\prime}<(i+1)l-1\}=\mathcal A_0+il,\quad
\bar{\mathcal A}_i=\{i^{\prime}\mid il-1\leq i^{\prime}\leq (i+1)l-1\}=\bar{\mathcal A}_0+il,
\] 
for each $i\in\bN$. Suppose $\lambda\in\mathcal A_0$ 
and $\mu\in\bar{\mathcal A}_0-\mathcal A_0$. We write 
$\lambda_i$ and $\mu_i$ for the unique elements in 
$\mathcal A_i\cap W_l.\lambda$ and $\bar{\mathcal A}_i\cap W_l.\mu$. 
Note that, if $\mu=\mu_0=-1$, then $\mu_1=\mu_2$, $\mu_3=\mu_4$ and so 
forth, while for $\mu=\mu_0=l-1$ 
we have $\mu_0=\mu_1$, $\mu_2=\mu_3$ etc. See also~\eqref{eq-alcove}.
\end{nota}

We define the tilting generator $\Tn{\bbin}$ and 
its cut-off's $\Tn{\leq \bbm}$ that will play a 
crucial role in Section~\ref{sec-quiver}. By 
abuse of notation, we denote them in the same 
way for all $\lambda$ and $\mu$ and hope that 
it is clear from the context which ones we consider.

\begin{defn}(\textbf{The tilting generator})\label{defn-tiltgen}
Let $m\in\bN\cup\{\infty\}$. We call the $\Uq$-modules given by
\[
\Tn{\bbin}=\bigoplus_{i=0}^{\infty}\Tn{\lambda_i}\quad\text{and}\quad\Tn{\leq \bbm}=\bigoplus_{i=0}^m\Tn{\lambda_i}
\] 
the \textit{tilting generator} and its \textit{$m$-th cut-off}, respectively. 
Likewise for $\mu$'s instead of $\lambda$'s.
\end{defn}

Denote by $\T_{\lambda}$ for $\lambda\in\mathcal A_0$ 
the $\lambda$ \textit{block of} $\T$: all indecomposable summands of 
objects of $\T_{\lambda}$ should be of 
the form $\Tn i$ for $i\in W_l.\lambda$. Note that 
blocks are in general not indecomposable as categories.
We, using a similar notation on walls, note the following.

\begin{lem}\label{lem-block}
We have
\[
\T=\bigoplus_{\lambda\in\mathcal A_0}\T_{\lambda}\oplus \T_{-1}\oplus \T_{l-1}
\]
with semisimple categories 
$\T_{-1}$ and $\T_{l-1}$ equivalent 
to the corresponding $\Uv$-module categories. 
We have for all $\lambda\in\mathcal A_0$
(similar for $\T_{-1}$ and $\T_{l-1}$):
\begin{itemize}
\item[(a)] The categories $\T_{\lambda}$ 
are additive, closed under finite direct sums and under duals.
\item[(b)] The categories $\T_{\lambda}$ 
are full Krull-Schmidt subcategories of $\T$. 
Moreover, every indecomposable tilting 
in $\T_{\lambda}$ is of the form $\Tn i$ for some $i\in\bN$ and $i\in W_l.\lambda$.\makeqed
\end{itemize}
\end{lem}

\begin{proof}
This is now only a 
combination of Theorem~\ref{thm-link} and 
Lemma~\ref{lem-tiltcat}. Note that 
$\T_{-1}$ and $\T_{l-1}$ are equivalent 
to categories of $\Uv$-modules by Proposition~\ref{prop-tilt} part (a).
\end{proof}

\begin{nota}\label{no-linkb}
If it is clear from the context 
which $\lambda$ we consider, then we, 
by abuse of notation, write 
$\Tn{\lambda_i}=\Tn{\bbi}$ for short. Similarly for simple and (dual) 
Weyl modules, and on walls.
\end{nota}

\begin{ex}\label{ex-block}
Take $l=3$ again. Then we only have to consider
$\lambda=0,1$ and $\mu=-1,2$.

The indecomposable tiltings in $\T_0$ are 
$\Tn{i}$ for $i=0,4,6,\dots$ as a look 
at~\eqref{eq-alcove} indicates. For $\T_1$ 
they are $\Tn{i}$ for $i=1,3,7,\dots$. The 
two blocks $\T_{-1}$ and $\T_2$ are semisimple 
and consist of direct sums of $\Ln i\cong\Tn{i}$ 
for $i=5,11,\dots$ and for $i=2,8,\dots$ respectively.
\end{ex}

\begin{rem}\label{rem-tria}
As in the usual case for an 
indeterminate $v$, we have a 
triangular decomposition
$\Uq=\Uq^-\Uq^0\Uq^+$, see for example~\cite[Section~1]{apw}. 
There is a character $\chi_i\colon \Uq^0\to\bC$ for any $i\in\bN$ 
(how the character is 
determined by $i$ can be found in~\cite[Lemma~1.1]{apw}).
We note that the Weyl module $\Vn{i}$ 
has the following \textit{universal property}: 
for any $\Uq$-module $M$ there is an isomorphism of vector spaces
\[
\Hom_{\Uq}(\Vn{i},M)\cong\{m\in M_i\mid E^{(j)}m=0\text{ for all }j\in\bN\},
\]
where $M_i=\{m\in M\mid um=\chi_i(u)m,\; u\in\Uq^0\}$ is the $i$ weight space 
of $M$.
This together with Definition~\ref{defn-weyl}, 
Proposition~\ref{prop-weyl} and Corollary~\ref{cor-weyl} imply that
\[
\Hom_{\Uq}(\Vn{i},\Ln{j})\cong\Hom_{\Uq}(\Vn{i},\dVn{j})\cong\delta_{ij}\bC
\]
for all $i,j\in\bN$. In particular,
\begin{equation}\label{eq-SchurLemma}
\Hom_{\Uq}(\Ln{i},\Ln{j})\cong\delta_{ij}\bC,
\end{equation}
i.e. Schur's Lemma holds in our set-up (although $\bC$ is not 
algebraically closed). Note that this is true for general quantum groups over 
arbitrary fields, see~\cite[Corollary~7.4]{apw}.
In addition\footnote{We point out that this fails in general. An explicit counterexample can be found for example in~\cite[Section~5]{ak}.}, we get
\[
\Hom_{\Uq}(\dVn{i},\Vn{i})\cong\bC.
\]
for all $i\in\bN$ (by using 
Proposition~\ref{prop-weyl} and Corollary~\ref{cor-weyl}).
\end{rem}

We call the following maps 
\textit{up} $u^{\lambda}_i$, 
\textit{down} $d^{\lambda}_i$ and 
\textit{loop} $\varepsilon^{\lambda}_i$ 
(\textit{or simply} $u_i,d_i$ and $\varepsilon_i$) respectively.

\begin{prop}\label{prop-link}
There exist up to scalars unique 
$\Uq$-intertwiners $u^{\lambda}_i,d^{\lambda}_i$ with
\begin{gather*}
u^{\lambda}_i\colon \Tn{\lambda_i}\to\Tn{\lambda_{i+1}},\; i=0,1,\dots,\quad\quad d^{\lambda}_i\colon \Tn{\lambda_i}\to\Tn{\lambda_{i-1}},\; i=1,2,\dots,
\\
u^{\lambda}_{i}\circ u^{\lambda}_{i-1}=0=d^{\lambda}_i\circ d^{\lambda}_{i+1},\; i=1,2,\dots,\quad\quad d^{\lambda}_{i+1}\circ u^{\lambda}_i=\varepsilon_i^{\lambda}=u^{\lambda}_{i-1}\circ d^{\lambda}_{i},\; i=1,2,\dots
\\
\varepsilon_i^{\lambda}\circ\varepsilon_i^{\lambda}=0,\; i=1,2,\dots
,\quad\quad d^{\lambda}_1\circ u^{\lambda}_0=0,
\end{gather*}
for $\lambda\in\mathcal A_0$. 
The equation on the bottom right is called 
the \textit{dead-end relation}.\makeqed
\end{prop}

\begin{proof}
First let us assume that we are not in case of the 
dead-end relation (in fact, we leave it to 
the reader to verify this special case), 
that is, the indices $i$ are at least $1$. 
We want to use Proposition~\ref{prop-weyl} 
and Corollary~\ref{cor-weyl}. Moreover, we 
note that the $\Uq$-morphisms below will 
only be unique up to scalars and their 
precise form does not matter. We only 
assume that they are non-zero. In fact, 
by abuse of notation, we always use the 
same symbols, but the maps are of course different in general.

We consider
\[
\xymatrix{
    & \Ln{\lambda_i}\phantom{.} \ar@{^{(}->}[r] & \dVn{\lambda_i}\phantom{.} \ar@{->>}[r]^a & \Ln{\lambda_{i-1}}
                \ar@{->} `r/8pt[d] `/10pt[l] `^dl[ll]|b `^r/3pt[dll] [dll] \\
             & \Ln{\lambda_{i-1}} \ar@{^{(}->}[r]^c & \Vn{\lambda_{i}} \ar@{->>}[r]^d & \Ln{\lambda_{i}}
                \ar@{->} `r/8pt[d] `/10pt[l] `^dl[ll]|e `^r/3pt[dll] [dll] \\
             & \Ln{\lambda_{i}} \ar@{^{(}->}[r]^f & \dVn{\lambda_{i}} \ar@{->>}[r] & \Ln{\lambda_{i-1}}.
    }
\]
Hence, we see that $\Hom_{\Uq}(\dVn{\lambda_i},\Vn{\lambda_i})$ is 
one dimensional: as noted in Remark~\ref{rem-tria}, the 
hom-spaces between Weyl and dual Weyl modules are at 
most $1$-dimensional. The composite $cba$ ensures 
that the dimension is exactly $1$ since it spans the hom-space.

Likewise for $\Hom_{\Uq}(\Vn{\lambda_i},\dVn{\lambda_i})$ by 
using $fed$. Note that the composite $fedcba=0$ due to 
the fact that the middle row is exact. The same holds when the roles 
of $\Vn{\lambda_i}$ and $\dVn{\lambda_i}$ are exchanged. 
Moreover, note that the morphism from $\dVn{\lambda_i}$ to 
$\Vn{\lambda_i}$ uses only $\Ln{\lambda_{i-1}}$ while the 
other way around uses only $\Ln{\lambda_{i}}$. Thus, up 
to scalars, morphisms from $\dVn{\lambda_i}$ to $\Vn{\lambda_i}$ are 
the ``same'' as morphisms from $\Vn{\lambda_{i-1}}$ to $\dVn{\lambda_{i-1}}$.
We can now construct, by using Proposition~\ref{prop-tilt}, 
up $u_i^{\lambda}$ and 
down $d_i^{\lambda}$ 
as a composition of the following maps.
\[
\xymatrix{
    & \Vn{\lambda_i}\phantom{.} \ar@{^{(}->}[r] & \Tn{\lambda_i}\phantom{.} \ar@{->>}[r]^a & \Vn{\lambda_{i-1}}
                \ar@{->} `r/8pt[d] `/10pt[l] `^dl[ll]|b `^r/3pt[dll] [dll] \\
             & \Vn{\lambda_{i-1}} \ar@{^{(}->}[r]^c & \Tn{\lambda_{i-1}} \ar@{->>}[r]^d & \Vn{\lambda_{i-2}}
             \ar@{->} `r/8pt[d] `/10pt[l] `^dl[ll]|e `^r/3pt[dll] [dll] \\
             & \Vn{\lambda_{i-2}} \ar@{^{(}->}[r]^f & \Tn{\lambda_{i-2}} \ar@{->>}[r] & \Vn{\lambda_{i-3}}.
    }
\]
We define $d_i^{\lambda}$ to be the 
composite $cba$. Similar for $u_i^{\lambda}$, 
but using the right-hand side of part (b) of 
Proposition~\ref{prop-tilt}. By the same 
reasoning as above we see that they are unique up to scalars.

We have to check the relations between the 
various up $u_i^{\lambda}$ and down $d_i^{\lambda}$ maps 
now. To see that $u^{\lambda}_{i}\circ u^{\lambda}_{i-1}=0$ 
and $d^{\lambda}_i\circ d^{\lambda}_{i+1}=0$ we can simply 
use the second diagram above and its dual 
counterpart and the fact that the rows are exact. Now consider
\[
\xymatrix{
\Tn{\lambda_{i-1}}\ar@{<->}[d]|{\mathrm{id}} & \Vn{\lambda_{i-1}}\ar@{_{(}->}[l]\ar[d]|{\Ln{\lambda_{i-1}}\to\Ln{\lambda_{i-1}}} & \Tn{\lambda_{i}}\ar@{<->}[d]|{\mathrm{id}}\ar@{->>}[r]\ar@{->>}[l] & \dVn{\lambda_{i}}\ar@{^{(}->}[r]\ar[d]|{\Ln{\lambda_{i-1}}\to\Ln{\lambda_{i-1}}} & \Tn{\lambda_{i+1}}\ar@{<->}[d]|{\mathrm{id}}\\
\Tn{\lambda_{i-1}}\ar@{->>}[r] & \dVn{\lambda_{i-1}}\ar@{^{(}->}[r] & \Tn{\lambda_{i}} & \Vn{\lambda_{i}}\ar@{_{(}->}[l] & \Tn{\lambda_{i+1}}\ar@{->>}[l].
}
\]
Combining everything, we see that 
(up to scalars) $d^{\lambda}_{i+1}\circ u^{\lambda}_i=\varepsilon_i^{\lambda}=u^{\lambda}_{i-1}\circ d^{\lambda}_{i}\neq 0$. Moreover, by 
the reasoning above, the $\varepsilon_i^{\lambda}$ 
is an up to scalars unique non-zero 
$\Uq$-morphism $\Tn{\lambda_i}$ to $\Tn{\lambda_i}$ that squares to zero.
\end{proof}

\newpage

\begin{cor}\label{cor-homspaces}
Let $i,i^{\prime}\in\bN$. Then we have the following.
\begin{itemize}
\item[(a)] Outside of walls:
\[
\Hom_{\Uq}(\Tn{\lambda_i},\Tn{\lambda_{i^{\prime}}})\cong\begin{cases}\bC[\varepsilon],  & \text{if }|i-i^{\prime}|=0\text{ and }i=i^{\prime}\neq 0,\\
  \bC,  & \text{if }|i-i^{\prime}|=1 \text{ or }i=i^{\prime}=0,\\
  0,  & \text{if }|i-i^{\prime}|>1,\end{cases}
\]
where $\bC[\varepsilon]\cong \bC[X]/(X^2)$ denotes the $\bC$-algebra of dual numbers.
\item[(b)] On walls:
\[
\Hom_{\Uq}(\Tn{\mu_i},\Tn{\mu_{i^{\prime}}})\cong\begin{cases}\bC,  & \text{if }|i-i^{\prime}|=0,\\
  0,  & \text{if }|i-i^{\prime}|>0,\end{cases}
\]
where $\bC$ is the trivial $\Uq$-module.\makeqed
\end{itemize}
\end{cor}

\begin{proof}
Because of Proposition~\ref{prop-link}, we only 
need to verify (b). But since the $\Tn{\mu_i}$ 
are simple on walls (by Proposition~\ref{prop-tilt}), 
part (b) follows from ``Schur's Lemma''~\eqref{eq-SchurLemma}.
\end{proof}

\subsection{Translation functors}\label{sub-translation}
Fix $\lambda\in\mathcal A_0$ and 
$\mu\in\bar{\mathcal{A}}_0-\mathcal A_0$. Set $\nu=|\lambda-\mu|$ 
(in fact, $\nu=|\lambda_i-\mu_i|$ for all 
$i\in\bN$ and $\nu$ will stay in $\bar{\mathcal{A}}_0$). 
Recall that there is a unique simple 
$\Uq$-module $\Ln{\nu}\cong\Tn{\nu}$.
Denote by $p_{\lambda}\colon\T\to\T_{\lambda}$ the 
\textit{projection onto the block} $\T_{\lambda}$ functor. Similarly for $\mu$.

\begin{defn}(\textbf{Onto and out of the wall})\label{defn-onout}
Given $\lambda$ and $\mu$, we define 
two functors via
\[
\mathcal T_{\lambda}^{\mu}\colon\T_{\lambda}\to\T_{\mu},\;M\mapsto p_{\mu}(M\otimes_{\bC}\Tn{\nu}),\quad\quad\mathcal T^{\lambda}_{\mu}\colon\T_{\mu}\to\T_{\lambda},\;M\mapsto p_{\lambda}(M\otimes_{\bC}\Tn{\nu}).
\]
We call them \textit{translation onto} the $\mu$-wall $\mathcal T_{\lambda}^{\mu}$ and 
\textit{translation out of} the $\mu$-wall $\mathcal T^{\lambda}_{\mu}$.
\end{defn}

\begin{defn}(\textbf{Through the wall})\label{defn-trans}
Define $\Theta_s^{\lambda}=\mathcal T^{\lambda}_{-1}\circ \mathcal T_{\lambda}^{-1}$ and $\Theta_t^{\lambda}=\mathcal T^{\lambda}_{l-1}\circ \mathcal T_{\lambda}^{l-1}$. We 
call them \textit{translation through the wall functors}.
\end{defn}

If it is clear which $\lambda$ we are using we, abusing notation, denote them by $\Theta_s,\Theta_t$.

\begin{prop}\label{prop-functors}
For all $i\in\bN$ we have 
the following.
\begin{itemize}
\item[(a)] The functors 
$\mathcal T_{\lambda}^{\mu}$ and $\mathcal T^{\lambda}_{\mu}$ are well-defined (their definition gives $\Uq$-tilting modules in the right blocks), 
adjoints (left and right) and 
exact. Thus, $\Theta^{\lambda}_s$ and 
$\Theta^{\lambda}_t$ are exact and self-adjoint.
\item[(b)] We have (recalling that $\Ln{\mu_i}\cong\Tn{\mu_i}$)
\[
\mathcal T_{\lambda}^{\mu}(\Tn{\lambda_i})\cong\begin{cases}\Tn{\mu_{i-1}}\oplus\Tn{\mu_{i+1}},  & \text{if }\mu_i>\lambda_i,\\
  \Tn{\mu_i}\oplus\Tn{\mu_i},  & \text{if }\mu_i<\lambda_i,\end{cases}\quad
\mathcal T^{\lambda}_{\mu}(\Tn{\mu_i})\cong\begin{cases}\Tn{\lambda_{i+1}},  & \text{if }\mu_i>\lambda_i,\\
 \Tn{\lambda_{i}},  & \text{if }\mu_i<\lambda_i.\end{cases}
\]
\item[(c)] The dead-end relations $\Theta^{\lambda}_s(\Tn{\lambda_0})\cong0$, $\Theta^{\lambda}_s(\Tn{\lambda_1})\cong\Tn{\lambda_2}$, and $\Theta^{\lambda}_t(\Tn{\lambda_0})\cong\Tn{\lambda_1}$. Moreover, we have
\[
\Theta^{\lambda}_{s\text{ or }t}(\Tn{\lambda_i})\cong\begin{cases}\Tn{\lambda_{i-1}}\oplus\Tn{\lambda_{i+1}} ,  & \text{if }i>1\text{ is odd for }s\text{ and even for }t,\\
  \Tn{\lambda_{i}}\oplus\Tn{\lambda_{i}} ,  & \text{if }i>0\text{ is odd for }t\text{ and even for }s.\end{cases}
\]
\end{itemize}
Here we set $\Tn{-1}=\Tn{\lambda_{-1}}=\Tn{\mu_{-1}}=0$.\makeqed
\end{prop}

Note that biadjoint functors have in 
addition some other nice properties, see e.g.~\cite[Section~2]{kh5}. Moreover, 
in the 
case $\mu=-1$ we have $\mu_i>\lambda_i$ 
iff $i$ is odd whereas in the case $\mu=l-1$ 
we have $\mu_i>\lambda_i$ iff $i$ is even as 
a look at~\eqref{eq-alcove} should convince the reader.

\begin{proof}
\textbf{(a).} The functors are well-defined by 
(c) of Proposition~\ref{prop-basicp} and
part (b) of Lemma~\ref{lem-tiltcat}, i.e. tensor products of tiltings are tiltings. 
The other statements can be verified as in the 
case of category $\mathcal O$, i.e. we can adopt~\cite[Sections~7.1 and~7.2]{hum} without difficulties.

\textbf{(b).} Use 
Propositions~\ref{prop-weyl} 
and~\ref{prop-tilt}, Corollary~\ref{cor-weyl} and Theorem~\ref{thm-link}.

\textbf{(c).} This is just a direct application of the finer list of 
statements given in (b).
\end{proof}

We use the convention that 
$\Theta^{\lambda}_{k-tst}$ denotes an 
\textit{alternating} composition of length 
$k$ of translation through the wall functors starting with 
$\Theta^{\lambda}_{t}$. Likewise for $\Theta^{\lambda}_{k-sts}$.
Using (c) of Proposition~\ref{prop-functors}, we get the following.

\begin{cor}(\textbf{Combinatorics of the translation 
through the wall functors})\label{cor-theta}
We have:
\begin{itemize}
\item[(a)] $\Theta_s^{\lambda}\circ\Theta_s^{\lambda}\cong \Theta_s^{\lambda}\oplus \Theta_s^{\lambda}$ and $\Theta_t^{\lambda}\circ\Theta_t^{\lambda}\cong \Theta_t^{\lambda}\oplus \Theta_t^{\lambda}$ as functors.
\item[(b)] We have for $i\in\bN$ even respectively odd, 
that (with $k\geq 0$ terms $\Theta^{\lambda}_{t\text{ or }s}$) 
there exist multiplicities $m_j\in\bN$ (that can be zero) such that
\begin{gather*}
\Theta^{\lambda}_{k-tst}\Tn{\lambda_i}=(\Theta^{\lambda}_{s\text{ or }t}\circ\dots\Theta^{\lambda}_t\circ\Theta^{\lambda}_s\circ\Theta^{\lambda}_t)\Tn{\lambda_i}\cong \Tn{\lambda_{i+k}}\oplus \bigoplus_{j<i+k}\Tn{\lambda_{j}}^{\oplus m_j},
\\
\Theta^{\lambda}_{k-tst}\Tn{\lambda_i}=(\Theta^{\lambda}_{s\text{ or }t}\circ\dots\Theta^{\lambda}_t\circ\Theta^{\lambda}_s\circ\Theta^{\lambda}_t)\Tn{\lambda_i}\cong \Tn{\lambda_{i-1+k}}^{\oplus 2}\oplus \bigoplus_{j<i-1+k}\Tn{\lambda_{j}}^{\oplus m_j}.
\end{gather*}
Similarly for $i\in\bN$ odd respectively even and $\Theta^{\lambda}_{k-sts}$.\makeqed
\end{itemize}
\end{cor}

\begin{proof}
Clear from Proposition~\ref{prop-functors}.
\end{proof}

Define the following functors:
\begin{gather*}
\begin{aligned}
\mathcal T^{\mu,i}_{\lambda}=p_{\bbi}\circ\mathcal T^{\mu}_{\lambda},
\quad\quad&
\mathcal T_{\mu}^{\lambda,i}(\Tn{\mu_j})=\begin{cases}
\delta_{i,j}\Tn{\lambda_{i}},&\text{if }\mu=-1, i\text{ even},\\
\delta_{i,j}\Tn{\lambda_{i}},&\text{if }\mu=l-1, i\text{ odd},
\end{cases}
\\
\Theta_{s\text{ or }t}^{\lambda,i}&=\begin{cases}p_{\bbi}\circ\Theta^{\lambda}_{s}, &\text{if }i\text{ is even},\\
p_{\bbi}\circ\Theta^{\lambda}_{t}, &\text{if }i\text{ is odd}.
\end{cases}
\end{aligned}
\end{gather*}
Here $p_{\bbi}$ denotes the projection to the $\Tn{\bbi}$-part.

\begin{defn}\label{defn-tiltcatendo}
Given a $\Uq$-tilting module $T$, denote by 
$\mathcal F_{T}$ any functor $\mathcal F_{T}\cong \cdot\otimes_{\bC}T$.
A functor $\mathcal F\colon\T\to\T$ is called 
\textit{projective} if it is a direct summand 
of functors of the form $\mathcal F_{T}$.
We denote by $\Endo(\T)$ the \textit{category 
of projective endofunctors} $\mathcal F\colon\T\to\T$ 
whose morphisms $\Hom_{\Endo(\T)}(\mathcal F,\mathcal F^{\prime})$ 
are natural transformations $\eta\colon\mathcal F\to\mathcal F^{\prime}$. 
Moreover, we denote by $\Endo(\T_{\lambda})$ the \textit{category 
of projective endofunctors} $\mathcal F\colon\T_{\lambda}\to\T_{\lambda}$ 
that can be obtained via compositions and countable direct sums 
of functors of the form $\Theta_{s\text{ or }t}^{\lambda,i}$. Similarly for walls $\mu$. 
Here we count the empty sum as the identity functor. 
\end{defn}

It follows from Lemma~\ref{lem-profunc} below that these are actually categories. 
Moreover, by construction, the translation through the wall functors 
$\Theta_s^{\lambda}$ or $\Theta_t^{\lambda}$ are objects in $\Endo(\T_{\lambda})$, 
since
\begin{equation}\label{eq-projectivefunc}
\Theta_s^{\lambda}=\!\!\bigoplus_{i\in\bZ_{\geq 1}}\!\!\Theta_s^{\lambda,2i},\quad\quad\Theta_t^{\lambda}=\!\!\bigoplus_{i\in\bZ_{\geq 0}}\!\!\Theta_s^{\lambda,2i+1},
\end{equation}
as follows by Proposition~\ref{prop-functors}. Similarly,
\begin{equation}\label{eq-projectivefunc2}
\mathcal T_{\lambda}^{\mu}=\!\!\bigoplus_{i\in\bZ_{\geq 0}}\!\!\mathcal T_{\lambda}^{\mu,i},\quad\quad \mathcal T_{\mu}^{\lambda}=\begin{cases}\bigoplus_{i\in\bZ_{\geq 0}}\!\!\mathcal T_{\mu}^{\lambda,2i}, &\text{if }\mu=-1,\\
\bigoplus_{i\in\bZ_{\geq 0}}\!\!\mathcal T_{\mu}^{\lambda,2i+1}, &\text{if }\mu=l-1.\end{cases}
\end{equation}

\newpage

\begin{lem}\label{lem-profunc}
$\Endo(\T), \Endo(\T_{\lambda}), \Endo(\T_{\mu})$ are additive 
categories of exact functors.\makeqed
\end{lem}

\begin{proof}
That $\Endo(\T)$ is closed under composition 
follows from part (b) of Lemma~\ref{lem-tiltcat}, 
and the identity functor is projective because 
of $\mathrm{id}\cong\cdot\otimes_{\bC} \Tn 0$. Hence, $\Endo(\T)$ is a category.
That $\Endo(\T)$ is preserved under finite 
direct sums (showing that the category is 
additive) follows from part (a) of Lemma~\ref{lem-tiltcat}. 
The statement for the $\lambda$ and $\mu$ 
versions follow by construction.
Last, that functors in 
$\Endo(\T)$ are exact follows because tensoring over a 
field always respects exact sequences. 
Likewise for the $\lambda$ and $\mu$ versions which come 
from projections of such functors.
\end{proof}

The following lemma 
is true in more generality (see for example~\cite[Chapter~4, Section~6]{ass}), 
but we restrict to our case here. Recall that a 
functor between additive categories 
$\mathcal F\colon\mathcal C\to\mathcal D$ is 
called \textit{indecomposable} if any 
decomposition $\mathcal F\cong\mathcal F_1\oplus\mathcal F_2$ 
implies that $\mathcal F_1\cong 0$ or $\mathcal F_2\cong 0$.
Moreover, we call a category \textit{infinite} Krull-Schmidt 
if each object can be uniquely decomposed (up to permutation) 
into a countably infinite direct sum of indecomposable objects.

\begin{lem}\label{lem-profuncinde}
A functor $\mathcal F\cong \cdot\otimes_{\bC} T$ is 
indecomposable iff $T$ is an indecomposable 
$\Uq$-module. Thus, $\Endo(\T)$ is Krull-Schmidt. In contrast, the $\lambda$ and 
$\mu$ versions are infinite Krull-Schmidt.\makeqed
\end{lem}

\begin{proof}
Assume that $T$ decomposes 
into $T_1\oplus T_2$. Then $\mathcal{F}\cong (\cdot\otimes_{\bC}T_1)\oplus (\cdot\otimes_{\bC}T_2)$.
On the other hand, by Yoneda 
and the tensor-hom adjunction, 
the $T$ representing $\mathcal F$ is 
uniquely determined up to isomorphism. 
Thus, a decomposition of $\mathcal F$ 
induces a decomposition of $T$. This 
implies, by using 
part (a) of Lemma~\ref{lem-tiltcat}, that $\Endo(\T)$ is Krull-Schmidt.

The statement for the $\lambda$ and $\mu$ 
versions follow again by construction.
\end{proof}

\begin{prop}\label{prop-profunc}
The functors $\mathcal T^{\mu,i}_{\lambda}$, $\mathcal T_{\mu}^{\lambda,i}$ and 
$\Theta^{\lambda,i}_{s\text{ or }t}$ are 
all indecomposable.
\end{prop}

\begin{proof}
Any non-trivial decomposition 
$\Theta^{\lambda,i}_{s\text{ or }t}=F_1\oplus F_2$ 
gives rise to a non-trivial idempotent 
$\eta\in\Endo(\Theta^{\lambda,i}_{s\text{ or }t})$. 
By using Proposition~\ref{prop-functors} together 
with Corollary~\ref{cor-homspaces} we see that 
such a natural transformation can not exist 
(most of the $\Uq$-intertwiners are nilpotent). 
Let us illustrate this in an explicit calculation for $\Theta^{\lambda,i}_t$ 
with $i\neq 0$.
Applying $\Theta^{\lambda,i}_t$ on the sequence
\begin{equation}\label{eq-inde}
\xy
\xymatrix{
 \Tn{\bbid}
\ar@<2pt>[r]^/0.25cm/{u_{i-1}} & \Tn{\bbi}\ar@<2pt>[l]^/-0.25cm/{d_{i}}\ar@<2pt>[r]^/-0.25cm/{u_i} 
& \Tn{\bbiu}\ar@<2pt>[l]^/0.25cm/{d_{i+1}}
}
\endxy
\end{equation}
gives us (with bottom and top row given by applying 
$\Theta^{\lambda,i}_t$ to~\eqref{eq-inde})
\begin{gather}\label{eq-inde2}
\begin{aligned}
\xy
\xymatrix{
\Tn{\bbi}\ar@<2pt>[r]^/-1.0em/{\begin{pmatrix} 0\\ 1 \end{pmatrix}}
\ar@{<->}[dd]|{\eta_{\Tn{\bbid}}}
&
\Tn{\bbi}\oplus\Tn{\bbi}\ar@<2pt>[r]^/0.9em/{\begin{pmatrix} 1 & 0\end{pmatrix}}
\ar@<2pt>[l]^/0.9em/{\begin{pmatrix} 1 & 0 \end{pmatrix}}
\ar@{<->}[dd]|{\eta_{\Tn{\bbi}}}
&
\Tn{\bbi}\ar@<2pt>[l]^/-0.9em/{\begin{pmatrix} 0 \\ 1 \end{pmatrix}}
\ar@{<->}[dd]|{\eta_{\Tn{\bbiu}}}\\
& & &\\
\Tn{\bbi}\ar@<2pt>[r]^/-1.0em/{\begin{pmatrix} 0\\ 1 \end{pmatrix}}
&
\Tn{\bbi}\oplus\Tn{\bbi}\ar@<2pt>[r]^/0.9em/{\begin{pmatrix} 1 & 0\end{pmatrix}}
\ar@<2pt>[l]^/0.9em/{\begin{pmatrix} 1 & 0 \end{pmatrix}}
&
\Tn{\bbi}.\ar@<2pt>[l]^/-0.9em/{\begin{pmatrix} 0 \\ 1 \end{pmatrix}}
}
\endxy
\end{aligned}
\end{gather}
Here we used Proposition~\ref{prop-functors} to calculate 
$\Theta^{\lambda,i}_t(\Tn{i})$, while the images of the 
$\Uq$-intertwiners can be calculated as in Example~\ref{ex-howtofunctor}. 
Note that $\eta$ above is assumed to be a natural transformation, i.e. 
all squares in~\eqref{eq-inde2} commute.
Now, under the assumption that $\eta$ is an idempotent, we have $\eta_{\Tn{\bbid}}$ 
equals $0$ or $1$ because of Corollary~\ref{cor-homspaces}. 
In the first case we obtain that all others components 
of $\eta$ are zero as well, while in the second case we obtain 
that all others components 
of $\eta$ are the identity as well 
(both follow via direct computation using Corollary~\ref{cor-homspaces}, the assumption 
that all components of $\eta$ are idempotents 
and the commutativity of all squares in~\eqref{eq-inde2}). 
The same arguments and calculations show the 
claim for the functors $\Theta^{\lambda,i}_s$ 
and $\mathcal T_{\lambda}^{\mu,i}$ 
(except that our job is even easier for the latter since all 
involved tiltings are simple). 
The claim for the functor $\mathcal T^{\lambda,i}_{\mu}$ is left to the reader.
This shows 
the statement.
\end{proof}

\begin{rem}\label{rem-projective}
The category $\Endo(\T_{\lambda})$ is the Karoubi envelope 
(in the sense of Definition~\ref{defn-karoubi}) of the 
category $\Endo_{\Theta}(\T_{\lambda})$ generated 
(via direct sums and compositions) by the 
translation through the wall functors $\Theta^{\lambda}_{s\text{ or }t}$. 
This follows from~\eqref{eq-projectivefunc} and Proposition~\ref{prop-profunc}. 
We call $\Endo(\T_{\lambda})$ the category of projective endofunctors 
although it would be historical more accurate to call 
$\Endo_{\Theta}(\T_{\lambda})$ the category of projective endofunctors. 
This convention is due to the fact that, in contrast to the case of category $\mathcal O$ 
(see e.g.~\cite{bg}), our buildings blocks are the 
various $\Theta^{\lambda,i}_{s\text{ or }t}$ 
and not the functors $\Theta^{\lambda}_{s\text{ or }t}$ themselves.
\end{rem}
\section{The Khovanov-Seidel quiver algebra and gradings}\label{sec-quiver}
In the present section\footnote{We use \textit{right} modules in this 
section and read from right to left: we think of the paths 
in the quiver as applying morphisms/functors. The reason for this 
will become clear in Subsection~\ref{sub-ainfty}.} we identify 
the quiver for the blocks $\T_{\lambda}$. We show how the 
quiver description can be used to define a grading on $\T$. We 
obtain a grading on $\Endo(\T)$ as well.

Additionally, we show how the translation functors of $\T_{\lambda}$ 
can be used to get a categorification of the action of the 
braid group on the (split) Grothendieck group of $\T_{\lambda}$ which realizes the 
Burau representation on the (split) Grothendieck group. Moreover, we show how 
the category of projective endofunctors on $\T_{\lambda}$ can be used to 
categorify 
(parts of) the Temperley-Lieb $\bQ(v)$-algebras (where the grading decategorifies to the 
indeterminate $v$).

By an algebra we always mean an associative, but not necessarily unital $\bC$-algebra.

\subsection{The quiver algebras \texorpdfstring{$\Am$}{Qm} and \texorpdfstring{$\Ai$}{Ain}}\label{sub-qalg}
Let $m\in\bN$. We consider the following quiver
\begin{equation}\label{eq-quiver}
\xymatrix{
  \raisebox{0.1cm}{\xy(0,0)*{\bullet};(0,-2.5)*{\scriptstyle \bbm};\endxy}  \ar@<-2pt>[r]_/-0.25em/{d_{m}}  &  \raisebox{0.1cm}{\xy(0,0)*{\bullet};(0,-2.5)*{\scriptstyle \bbmd};\endxy}  \ar@<-2pt>[l]_{u_{m-1}}\ar@<-2pt>[r]_{d_{m-1}} & 
\raisebox{0.1cm}{\xy(0,0)*{\bullet};(0,-2.5)*{\scriptstyle \bbmdd};\endxy}  \ar@<-2pt>[l]_{u_{m-2}}\ar@<-2pt>[r]_{d_{m-2}} &  
  \cdots \ar@<-2pt>[l]_{u_{m-3}}\ar@<-2pt>[r]_/+0.2em/{d_{3}} & \raisebox{0.1cm}{\xy(0,0)*{\bullet};(0,-2.5)*{\scriptstyle \bbt};\endxy}  \ar@<-2pt>[l]_/-0.2em/{u_{2}}\ar@<-2pt>[r]_{d_{2}} & \raisebox{0.1cm}{\xy(0,0)*{\bullet};(0,-2.5)*{\scriptstyle \bbo};\endxy}  \ar@<-2pt>[l]_{u_{1}}\ar@<-2pt>[r]_{d_{1}} & \raisebox{0.1cm}{\xy(0,0)*{\bullet};(0,-2.5)*{\scriptstyle \bbz};\endxy}  \ar@<-2pt>[l]_{u_0}\\
}
\end{equation}
having $m+1$ vertices $\bbz,\bbo,\dots,\bbm$ and $2m$ arrows called \textit{up and down}:
\[
u_i\colon \bbi\to \bbiu,\;i=0,\dots,m-1,\quad\quad d_i\colon \bbi\to \bbid,\;i=1,\dots,m.
\]

The \textit{path algebra} $P_m$ of the quiver 
from~\eqref{eq-quiver} is defined to 
be the $\bC$-algebra whose underlying 
$\bC$-module is the $\bC$-vector 
space spanned by all finite paths 
with multiplication given by composition 
of paths if possible and zero otherwise.

Recall that a $\bZ$\textit{-graded} 
$\bC$-algebra $A$ (or simply \textit{graded}) is a 
$\bC$-algebra $A=\oplus_{i\in\bZ} A^i$ such that $A^iA^j\subset A^{i+j}$ 
for all $i,j\in\bZ$.
The path algebra $P_m$ is a graded 
$\bC$-algebra\footnote{Note that Khovanov and Seidel 
use free $\bZ$-modules instead of $\bC$-vector spaces. 
In order to avoid too many different overlapping notations, 
we only use the $\bC$-vector space version, but it is not a big problem to work over $\bZ$.}, 
that is
\[
P_m=\bigoplus_{k\in\bZ_{\geq 0}}P^k_m,\quad\text{with}\; P^k_m=\{\text{All paths in }P_m\text{ of length }k\},
\]
since we clearly have 
$P^k_m\circ P^{k^{\prime}}_m\subset P^{k+k^{\prime}}_m$. We 
write $l(\cdot)$ for the grading on $P_m$ (the ``length'').

\begin{defn}\label{def-ksquiver1}(\textbf{Khovanov-Seidel's $m$-quiver algebra})
Let $\Am$ denote the quotient algebra obtained from the path algebra 
$P_m$ for the quiver from~\eqref{eq-quiver} by the defining relations
\begin{gather*}
u_{i}\circ u_{i-1}=0=d_i\circ d_{i+1},\;i=1,\dots,m-1,\quad\quad d_{i+1}\circ u_i=u_{i-1}\circ d_{i},\;i=1,\dots,m-1,\\
d_1\circ u_0=0.
\end{gather*}
The latter is called the \textit{dead-end relation}.
Given two paths $p,p^{\prime}\in \Am$ we write $p^{\prime}p$ 
instead of $p^{\prime}\circ p$. The algebra $\Am$ 
inherits the grading $l(\cdot)$ from $P_m$ 
since all the relations are homogeneous. 
We call $\Am$ for short \textit{KS} $m$\textit{-quiver algebra}.
\end{defn}

We denote, by abuse of notation, the path of length $0$ that 
starts and ends at $\bbi$ also by $\bbi$. Note that 
the $\bbi$ are projectors or idempotents, because 
$\bbi\bbi=\bbi$, $p\bbi=p$ if $p$ starts in $\bbi$ 
and $0$ else and $\bbi p=p$ if $p$ ends in $\bbi$ 
and $0$ else. Moreover, they form a complete set of 
pairwise orthogonal idempotents, that 
is $1=\bbz+\bbo+\dots+\bbm$ and $\bbi\bbj=\delta(\bbi,\bbj)$, 
where $1\in \Am$ is the unit. Note that $0\neq \bbz$, 
$1\neq \bbo$ and the $\bbi$'s are not central 
for $m>0$, since e.g. $0=\bbi u_i\neq u_i\bbi=u_i$ for $i=0,\dots,m-1$.

Moreover, we denote for $i=1,\dots,m$ by $\varepsilon_i=u_{i-1}d_i$ 
the loop that starts at $\bbi$ and goes via $d_i$ to $\bbid$ 
and then back via $u_{i-1}$. Note that the relations imply 
that $\varepsilon_i=d_{i+1}u_i$ (if possible, i.e. if $i+1\leq m$). 
Thus, the $\bC$-algebra $\Am$ has a basis given by $\bbi$ 
(for $i=0,\dots,m$) and $u_i$ (for $i=0,\dots,m-1$) and 
$d_i, \varepsilon_i$ (for $i=1,\dots,m$) with $l(\bbi)=0,l(u_i)=l(d_i)=1$ and $l(\varepsilon_i)=2$.

We call an algebra homomorphism $f\colon A\to B$ between 
graded algebras $A$ and $B$ a 
\textit{homomorphism of graded algebras}, if $f(A^k)\subset B^k$ for all $k\in\bZ$.
We have the following.

\begin{lem}\label{lem-inclusion}
There is a sequence of (non-unital!) inclusions of graded algebras
\[
\xymatrix{
  Q_0 \ar@^{(->}[r]^{\iota_0} &  Q_1  \ar@^{(->}[r]^{\iota_1}  & Q_2  \ar@^{(->}[r]^{\iota_2} & \dots,
}
\]
where $\iota_m\colon Q_m\to Q_{m+1}$ is defined by $\bbi,u_i,d_i,\varepsilon_i\mapsto \bbi,u_i,d_i,\varepsilon_i$ for all suitable indices $i$.\makeqed
\end{lem}

\begin{proof}
This follows because $P_m$ includes 
into $P_{m+1}$ and the set of relations for 
$\Am$ is included in the ones for $Q_{m+1}$. 
That these morphisms respect the grading and are injective is immediate.
\end{proof}

This motivates the following definition.
\begin{defn}\label{def-ksquiver2}(\textbf{Khovanov-Seidel's $\infty$-quiver algebra})
Define $\Ai$ to be the inductive limit 
of the sequence of inclusions of graded 
algebras from Lemma~\ref{lem-inclusion}, that is $\Ai=\varinjlim \Am$. 
We call it for short 
\textit{KS} $\infty$\textit{-quiver algebra}.
\end{defn}

Note that $\Ai$ is a graded $\bC$-vector 
space of countable dimension. Moreover, 
$\Ai$ is an algebra with a 
complete set of pairwise orthogonal 
idempotents $\{\bbi\mid i\in\bN\}$, 
but $\Ai$ is a \textit{non-unital} algebra, 
since the unit would have to be an infinite sum of the $\bbi$'s. Such 
an algebra is sometimes called \textit{idempotented}.

\begin{ex}\label{ex-quiver}
$Q_0$ consists just of $\bC$-multiples of 
$\bbz$. The algebra $Q_{m>0}$ can be visualized as
\[
\xymatrix{
  \raisebox{0.1cm}{\xy(0,0)*{\bullet};(0,-2.5)*{\scriptstyle \bbm};\endxy} \ar@(ul,ur)|{\varepsilon_m} \ar@<-2pt>[r]_{d_{m}} & \cdots \ar@<-2pt>[l]_{u_{m-1}}\ar@<-2pt>[r]_/+0.2em/{d_{i+1}} &  \raisebox{0.1cm}{\xy(0,0)*{\bullet};(0,-2.5)*{\scriptstyle \bbi};\endxy} \ar@(ul,ur)|{\varepsilon_i} \ar@<-2pt>[l]_/-0.2em/{u_{i}}\ar@<-2pt>[r]_{d_{i}} & \cdots \ar@<-2pt>[l]_{u_{i-1}}\ar@<-2pt>[r]_/+0.2em/{d_{3}} & \raisebox{0.1cm}{\xy(0,0)*{\bullet};(0,-2.5)*{\scriptstyle \bbt};\endxy} \ar@(ul,ur)|{\varepsilon_2} \ar@<-2pt>[l]_/-0.2em/{u_{2}}\ar@<-2pt>[r]_{d_{2}} & \raisebox{0.1cm}{\xy(0,0)*{\bullet};(0,-2.5)*{\scriptstyle \bbo};\endxy} \ar@(ul,ur)|{\varepsilon_1} \ar@<-2pt>[l]_{u_{1}}\ar@<-2pt>[r]_{d_{1}} & \raisebox{0.1cm}{\xy(0,0)*{\bullet};(0,-2.5)*{\scriptstyle \bbz};\endxy}  \ar@<-2pt>[l]_{u_0}\\
}
\]
The KS $\infty$-quiver algebra $\Ai$ can be visualized as
\[
\xymatrix{
   \cdots\ar@<-2pt>[r]_/+0.2em/{d_{i+1}} &  \raisebox{0.1cm}{\xy(0,0)*{\bullet};(0,-2.5)*{\scriptstyle \bbi};\endxy} \ar@(ul,ur)|{\varepsilon_i} \ar@<-2pt>[l]_/-0.2em/{u_{i}}\ar@<-2pt>[r]_{d_{i}} & \cdots \ar@<-2pt>[l]_{u_{i-1}}\ar@<-2pt>[r]_/+0.2em/{d_{3}} & \raisebox{0.1cm}{\xy(0,0)*{\bullet};(0,-2.5)*{\scriptstyle \bbt};\endxy} \ar@(ul,ur)|{\varepsilon_2} \ar@<-2pt>[l]_/-0.2em/{u_{2}}\ar@<-2pt>[r]_{d_{2}} & \raisebox{0.1cm}{\xy(0,0)*{\bullet};(0,-2.5)*{\scriptstyle \bbo};\endxy} \ar@(ul,ur)|{\varepsilon_1} \ar@<-2pt>[l]_{u_{1}}\ar@<-2pt>[r]_{d_{1}} & \raisebox{0.1cm}{\xy(0,0)*{\bullet};(0,-2.5)*{\scriptstyle \bbz};\endxy}  \ar@<-2pt>[l]_{u_0}\\
}
\]
We point out that $\Ai$ only has one 
asymmetry coming from the dead-end 
relation. Moreover, $\Ai$ is graded by the 
path length $l(\cdot)$ due to the fact 
that $\Ai$ does not contain infinite paths.
\end{ex}

In order to be able to also consider the 
semisimple blocks $\T_{-1}$ and $\T_{l-1}$ of 
our category $\T$, we also introduce another 
quotient of the path algebra $P_m$, denoted 
by $\At$, (and take a limit as above). Since 
the corresponding module categories should be 
semisimple, the quotient $\At$ is rather trivial 
and we call it \textit{KS trivial $m$-quiver algebra}.

\begin{defn}\label{defn-trivial}(\textbf{Khovanov-Seidel's trivial $m$-quiver algebra})
Let $\At$ denote the quotient algebra 
obtained from the path algebra $P_m$ for 
the quiver from~\eqref{eq-quiver} by the defining relations
\[
u_{i}=0=d_{i+1},\;i=0,\dots,m-1.
\]
The algebra $\At$ inherits the grading 
$l(\cdot)$ from $P_m$, but this grading is trivial. 
$\At$ consists only of orthogonal idempotents $\bbi$ for $i=0,\dots, m$. These idempotents 
form a set of pairwise non-isomorphic, central, orthogonal idempotents which shows 
that $\At$ is a semisimple algebra. Moreover, the algebra $\At$ is clearly isomorphic to $\bC\times\dots\times\bC$ (with $m+1$ factors).

As before we define \textit{Khovanov-Seidel's trivial} $\infty$\textit{-quiver algebra} via $\Att=\varinjlim\At$. We note that this can be visualized as
\[
\xymatrix{
   \cdots &  \raisebox{0.1cm}{\xy(0,0)*{\bullet};(0,-2.5)*{\scriptstyle \bbi};\endxy}  & \cdots  & \raisebox{0.1cm}{\xy(0,0)*{\bullet};(0,-2.5)*{\scriptstyle \bbt};\endxy} & \raisebox{0.1cm}{\xy(0,0)*{\bullet};(0,-2.5)*{\scriptstyle \bbo};\endxy} & \raisebox{0.1cm}{\xy(0,0)*{\bullet};(0,-2.5)*{\scriptstyle \bbz};\endxy}\\
}
\]
where we take the inductive limit as above.
\end{defn}

For convenience we set $u_m=d_0=\varepsilon_0=0$ for all $\Am$ and $d_0=\varepsilon_0=0$ for $\Ai$ in the following.

\subsection{Combinatorics of the graded, right \texorpdfstring{$\Am$}{Qm}- and \texorpdfstring{$\Ai$}{Ain}-modules}\label{sub-ammod}
Recall that, if $A$ denotes some 
graded $\bC$-algebra, then a (right) 
$A$-module $M$ is called \textit{$\bZ$-graded} (or simply \textit{graded}), if
\[
M=\bigoplus_{k\in\bZ}M^k\quad\text{and}\quad M^k\cdot A^{k^{\prime}}\subset M^{k+k^{\prime}}.
\]
An $A$-module homomorphism between 
graded modules $f\colon M\to N$ is called 
\textit{degree preserving}, if $f(M^k)\subset N^k$ 
for all $k\in\bZ$ and \textit{homogeneous 
(of degree $d\in\bZ$)} if $f(M^k)\subset N^{k+d}$ for all $k\in\bZ$.
We denote by $\ModgrA$ the category of 
\textit{graded}, finitely generated $A$-modules 
whose morphisms from $M$ to $M^{\prime}$ are given by
\begin{equation}\label{eq-gradhom}
\Hom_{\ModgrA}(M,M^{\prime})=\bigoplus_{s\in\bZ}\Hom_{A}(M,M^{\prime}\langle s\rangle)_{0},
\end{equation}
where the zero should mean \textit{degree preserving} 
morphisms and $\langle s\rangle$ 
denotes a shift as below. Thus, all morphisms in $\ModgrA$ are finite 
direct sums of homogeneous morphisms.

The endofunctor $\cdot\langle s\rangle\colon\ModgrA\to\ModgrA$, 
called \textit{shift\footnote{In our convention a positive number shifts the degree \textit{up}.} 
by $s\in\bZ$} sends the $k-s$-th degree part of a module to the 
$k$-th of the shift $M\langle s\rangle$, that is, $M\langle s\rangle^k=M^{k-s}$.

We use similar notions for left modules 
(we denote such categories by e.g. $\AModgr$) 
or projective modules (we denote such categories 
by e.g. $\pModgrA$). We \textit{only} 
work in categories of (graded) finitely generated (projective) modules.

$\Am$ acts on itself from the right 
by \textit{pre}-composition and from the 
left by \textit{post}-composition of paths.
This makes a difference: 
let us denote by $P_i$ the 
left ideal of $\Am$ generated 
by $\bbi$. Similar, ${}_iP$ denotes 
the right ideal generated by $\bbi$. We have as $\bC$-vector spaces
\[
P_i=\bC\bbi\oplus\bC u_i\oplus\bC d_i\oplus\bC\varepsilon_i,\quad\quad{}_iP=\bC\bbi\oplus\bC u_{i-1}\oplus\bC d_{i+1}\oplus\bC\varepsilon_i,
\]
with homogeneous components of degree $0,1,1,2$ (from left to right).

The $P_i$ and the ${}_iP$ are graded 
$\Am$-modules (left and right respectively) that can be visualized as
\[
P_i=\xymatrix{
  \cdots  & \raisebox{0.1cm}{\xy(0,0)*{\bullet};(0,-2.5)*{\scriptstyle \bbi};\endxy}\ar@(ul,ur)|{\varepsilon_i}  \ar@<-2pt>[l]_/-0.2em/{u_i}\ar@<-2pt>[r]_/-0.2em/{d_i} &\cdots \\
},\quad\quad{}_iP=\xymatrix{
   \cdots \ar@<-2pt>[r]_{d_{i+1}} & \raisebox{0.1cm}{\xy(0,0)*{\bullet};(0,-2.5)*{\scriptstyle \bbi};\endxy}\ar@(ul,ur)|{\varepsilon_i}   &\cdots \ar@<-2pt>[l]_{u_{i-1}}.\\
}
\]
Thus, the $P_i$'s and the ${}_iP$'s 
are projective, since we have (as left, respectively right, $\Am$-modules)
\[
\Am=\bigoplus_{i=0}^m P_i,\quad\quad \Am=\bigoplus_{i=0}^m {}_iP.
\]
Moreover, they are all indecomposable because 
they are the projective covers of the simple left, respectively right, $\Am$-modules $L_i$ 
and ${}_iL$ 
obtained from $P_i$ 
and ${}_iP$ by killing $u_i,d_i$ and $\varepsilon_i$ 
respectively $u_{i-1}$, $d_{i+1}$ and $\varepsilon_i$.
Furthermore, it is easy to see that all 
indecomposable left, respectively right, 
$\Am$-modules are of the form $P_i\langle s\rangle$, 
respectively ${}_iP\langle s\rangle$, for some $s\in\bZ$: 
this follows directly from the fact that $\Am$ is 
finite-dimensional. Thus, all indecomposable projective $\Am$-modules are 
direct summands of $\Am$ considered as a (left or right) $\Am$-module.

\begin{ex}\label{ex-decomp}
One easily checks that the same is true for 
the corresponding notions for $\Ai$ and 
that one gets a decomposition into ${}_iP$'s as 
follows (we have indicated the neighbouring ${}_iP$'s 
using alternating colors and we for simplicity assume that the $i$ below is odd)
\[
\xymatrix{
   \cdots\ar@[blue]@<-2pt>[r]_/+0.2em/{{\color{blue}d_{i+1}}} &  {\color{blue}\raisebox{0.1cm}{\xy(0,0)*{\bullet};(0,-2.5)*{\scriptstyle \bbi};\endxy}} \ar@[blue]@(ul,ur)|{{\color{blue}\varepsilon_i}} \ar@[green]@<-2pt>[l]_/-0.2em/{{\color{green}u_{i}}}\ar@[green]@<-2pt>[r]_{{\color{green}d_{i}}} & \cdots \ar@[blue]@<-2pt>[l]_{{\color{blue}u_{i-1}}}\ar@[green]@<-2pt>[r]_/+0.2em/{{\color{green}d_{3}}} & {\color{green}\raisebox{0.1cm}{\xy(0,0)*{\bullet};(0,-2.5)*{\scriptstyle \bbt};\endxy}} \ar@[green]@(ul,ur)|{{\color{green}\varepsilon_2}} \ar@[blue]@<-2pt>[l]_/-0.2em/{{\color{blue}u_{2}}}\ar@[blue]@<-2pt>[r]_{{\color{blue}d_{2}}} & {\color{blue}\raisebox{0.1cm}{\xy(0,0)*{\bullet};(0,-2.5)*{\scriptstyle \bbo};\endxy}} \ar@[blue]@(ul,ur)|{{\color{blue}\varepsilon_1}} \ar@[green]@<-2pt>[l]_{{\color{green}u_{1}}}\ar@[green]@<-2pt>[r]_{{\color{green}d_{1}}} & {\color{green}\raisebox{0.1cm}{\xy(0,0)*{\bullet};(0,-2.5)*{\scriptstyle \bbz};\endxy}} \ar@[blue]@<-2pt>[l]_{{\color{blue}u_0}}\\
}
\]
We encourage the reader to draw the 
decomposition into $P_i$'s. Again, all 
finitely generated, graded projective, 
indecomposable right $\Ai$-modules are, up 
to a shift, of the form ${}_iP$. Same for 
the left modules and the $P_i$'s. This can 
be deduced directly, but is 
true in more generality for idempotented 
algebras, see for example~\cite[Proposition~5.3.1]{klms}.
\end{ex}

Similarly for $\At$ and $\Att$ which decompose into the 
corresponding (left and right) modules denoted 
by $P_i^{\mathrm{triv}}$ and ${}_iP^{\mathrm{triv}}$.

\subsection{Endofunctors on \texorpdfstring{$\pModgr$-$\Am$}{pMod-Qm} and \texorpdfstring{$\pModgr$-$\Ai$}{pMod-Qin}}\label{sub-endo}
Set\footnote{Note that we use a different convention than Khovanov and Seidel for the gradings by shifting ${}_iP$ down by one.} 
$B_i=P_i\otimes_{\bC}{}_iP\langle -1\rangle$ for all 
indices $i=1,\dots,m$. Note that the $B_i$'s are 
graded $\Am$-bimodules with the tensor product 
$p\otimes p^{\prime}$ of degree $l(p\otimes p^{\prime})=l(p)+l(p^{\prime})$ 
for $p$ of degree $l(p)$ and $p^{\prime}$ of degree $l(p^{\prime})$. 
We define functors
\[
\mathcal U_i\colon \pModgrAm\to \pModgrAm,\;\mathcal U_i= \cdot \otimes_{\Am}B_i\quad\text{for}\quad i=1,\dots,m.
\]
Note that we have
\begin{equation}\label{eq-ufunctors}
{}_{i^{\prime}}P\otimes_{\Am}P_i\cong\biggl\{
\begin{matrix}
 \bC\bbi\oplus\bC\varepsilon_i,\; \text{if }|i-i^{\prime}|=0, & & \bC u_{i},\;\text{if }i^{\prime}-i=1,\\
 \bC d_{i},\;\text{if }i-i^{\prime}=1, & & 0,\;\text{if }|i-i^{\prime}|>1.
\end{matrix}
\end{equation}
This can be easily seen by considering (here $i=i^{\prime}$ as an example)
\[
\xymatrix{
   \cdots \ar@<-2pt>[r]_{d_{i+1}} & \raisebox{0.1cm}{\xy(0,0)*{{\color{green}\bullet}};(0,-2.5)*{{\color{green}\scriptstyle \bbi}};\endxy}\ar@[green]@(ul,ur)|{{\color{green}\varepsilon_i}}   &\cdots \ar@<-2pt>[l]_{u_{i-1}}\\
}\otimes_{\Am}\xymatrix{
  \cdots  & \raisebox{0.1cm}{\xy(0,0)*{{\color{green}\bullet}};(0,-2.5)*{{\color{green}\scriptstyle \bbi}};\endxy}\ar@[green]@(ul,ur)|{{\color{green}\varepsilon_i}}  \ar@<-2pt>[l]_/-0.3em/{u_i}\ar@<-2pt>[r]_/-0.3em/{d_i} &\cdots, \\
}
\]
where we have illustrated the 
overlapping pieces of ${}_iP$ and $P_i$. This clearly implies
\begin{equation}\label{eq-tlcat}
\mathcal U_i ({}_{i^{\prime}}P)\cong\begin{cases}{}_iP\langle -1\rangle\oplus {}_iP\langle +1\rangle,&\text{if }|i-i^{\prime}|=0,\\
{}_iP,&\text{if }|i-i^{\prime}|=1,\\
0,&\text{if }|i-i^{\prime}|>1.\end{cases}
\end{equation}
We note that we can see ${}_{i^{\prime}}P\otimes_{\Am}P_i$ as a 
graded, right $\At$-module where the action of $\bbi$ kills 
everything that does not start in $\bbi$. Note that, in this 
notation, $\bC\bbi,\bQ u_{i},\bC d_{i}$ and $\bC\varepsilon_i$ 
are all one dimensional $\At$-modules concentrated 
in degrees $0,1,1$ and $2$ respectively.
Hence, as graded, 
right $\At$-modules, 
$\bC u_{i}\cong\bC \bbi\langle+1\rangle$, $\bC d_{i}\cong\bC \bbi\langle+1\rangle$ and $\bC \varepsilon_{i}\cong\bC \bbi\langle+2\rangle$.

\begin{cor}\label{cor-tlfunctor}
The functors $\mathcal U_i\colon\pModgrAm\to\pModgrAm$ satisfy the following.
\begin{itemize}
\item[(a)] $\mathcal U_i\circ\mathcal U_i\cong \mathcal U_i\langle -1\rangle\oplus \mathcal U_i\langle +1\rangle$ for $i=1,\dots,m$.
\item[(b)] Let $i^{\prime}=i\pm 1$ such that $i,i^{\prime}\in\{1,\dots,m\}$. Then $\mathcal U_i\circ\mathcal U_{i^{\prime}}\circ\mathcal U_i\cong\mathcal U_i$.
\item[(c)] $\mathcal U_i\circ\mathcal U_{i^{\prime}}\cong 0$, if $|i-i^{\prime}|>1$.\makeqed
\end{itemize}
\end{cor}

\begin{proof}
This follows directly from~\eqref{eq-tlcat}.
\end{proof}

Finally we set
\[
\mathcal U_{\mathrm{even}}=\!\!\bigoplus_{1\leq i, 2i\leq m}\!\!\mathcal U_{2i},\quad\quad\mathcal U_{\mathrm{odd}}=\!\!\bigoplus_{0\leq i, 2i+1\leq m}\!\!\mathcal U_{2i+1}.
\]
By Corollary~\ref{cor-tlfunctor} we see that
\begin{equation}\label{eq-evenodd}
\mathcal U_{\mathrm{even}}\circ \mathcal U_{\mathrm{even}}\cong \mathcal U_{\mathrm{even}}\langle -1\rangle\oplus \mathcal U_{\mathrm{even}}\langle +1\rangle,\quad\quad\mathcal U_{\mathrm{odd}}\circ \mathcal U_{\mathrm{odd}}\cong \mathcal U_{\mathrm{odd}}\langle -1\rangle\oplus \mathcal U_{\mathrm{odd}}\langle +1\rangle.
\end{equation}
It is clear that we can easily adapt 
the definition of the $\mathcal U_i$'s 
to the $\Ai$ case. We denote these 
functors by $\mathcal U^{\infty}_i$ 
for all $i\in\bZ_{>0}$. These functors satisfy completely similar relations as before.

We also define $\mathcal U^{\infty}_{\mathrm{even}}$ 
and $\mathcal U^{\infty}_{\mathrm{odd}}$ completely 
analogously by taking the direct sum over all 
$i\in\bZ_{>0}$. The reader is still invited to 
check that these two satisfy similar relations as in~\eqref{eq-evenodd}.

We also refine 
the definitions of the functors $\mathcal U_i$ 
by \textit{factoring through} the category of 
graded (right) $\At$-modules which we denote by 
$\ModgrAt$. To be precise, we define $\mathcal U_i^t$ and $\mathcal U^i_t$ via
\begin{gather*}
\mathcal U_i^t\colon\ModgrAm\to\ModgrAt,\; \mathcal U_i^t=\cdot\otimes_{\Am}P_i
\\
\mathcal U^i_t\colon\ModgrAt\to\ModgrAm,\; \mathcal U^i_t=\cdot\otimes_{\bC}{}_iP\langle -1\rangle.
\end{gather*}
It is immediate that $\mathcal U_i=\mathcal U^i_t\circ\mathcal U^t_i$. 
We define even and odd versions of these refinements 
as before where we use $\mathcal U^t_{\mathrm{even}}$ 
and $\mathcal U_t^{\mathrm{even}}$ as notations (similar for the odd versions).

The following lemma can be seen, by Remark~\ref{rem-cato} below, 
as an analog of~\cite[Theorem~8.4]{st2}.

\begin{lem}\label{lem-adjoint}
As graded functors: $\mathcal U^t_i\langle+1\rangle$ is 
the right adjoint of $\mathcal U_t^i$, $\mathcal U^t_i$ 
is the left adjoint of $\mathcal U_t^i\langle-1\rangle$. 
Similar for the even/odd versions. The functors $\mathcal U_i$, $\mathcal U_{\mathrm{even}}$ 
and $\mathcal U_{\mathrm{odd}}$ are all graded self adjoint.\makeqed
\end{lem}

\begin{proof}
This is a case-by-case check. 
We use that, as graded $\bC$-vector spaces, we have
\begin{equation}\label{eq-homps}
\Hom_{\pModgrAm}({}_iP,{}_{i^{\prime}}P)\cong\biggl\{
\begin{matrix}
 \bC\bbi\oplus\bC\varepsilon_i,\; \text{if }|i-i^{\prime}|=0, & & \bC u_{i},\;\text{if }i^{\prime}-i=1,\\
 \bC d_{i},\;\text{if }i-i^{\prime}=1, & & 0,\;\text{if }|i-i^{\prime}|>1.
\end{matrix}
\end{equation}
The isomorphism is given by mapping a 
path $p$ on the right-hand side to the 
homomorphism ${}_iP\to {}_{i^{\prime}}P$ 
given by left multiplication 
(\textit{post}-composition) with $p$. 
Note that this is a homogeneous morphism of degree $l(p)$.

We only 
prove that $\mathcal U^t_i\langle+1\rangle$ is 
the right adjoint of $\mathcal U_t^i$ and leave the rest to the reader.
To this end, we have to show that there are isomorphisms of graded vector spaces
\[
\Hom_{\pModgrAm}(\mathcal U_t^i(\bC\bbi),{}_{i^{\prime}}P)\cong\Hom_{\pModgrAt}(\bC\bbi,\mathcal U^t_i\langle+1\rangle({}_{i^{\prime}}P)),
\]
where we can restrict to check only for objects 
as above by additivity. Note that we have three 
cases depending on $i-i^{\prime}$, since the case 
$|i-i^{\prime}|>1$ is clear.

$\boldsymbol{i=i^{\prime}}$. The left-hand side is 
$\Hom_{\pModgrAm}({}_iP\langle-1\rangle,{}_iP)\cong \bC\bbi\langle+1\rangle\oplus \bC\varepsilon_i\langle+1\rangle$, which follows from~\eqref{eq-homps}. 
Moreover, by~\eqref{eq-ufunctors} and the shift by $+1$, we get the same for the right side.

$\boldsymbol{i^{\prime}=i+1}$. The left-hand 
side gives $\Hom_{\pModgrAm}({}_iP\langle-1\rangle,{}_{i+1}P)\cong \bC u_i\langle+1\rangle$. This follows again by~\eqref{eq-homps}. The right-hand side gives, 
by~\eqref{eq-ufunctors} and the shift by 
$+1$, the same result.

$\boldsymbol{i^{\prime}=i-1}$. Similarly to the case $\boldsymbol{i^{\prime}=i+1}$.

Thus, we have an adjoint pair 
$(\mathcal U_t^i,\mathcal U^t_i\langle+1\rangle)$. 
The other statements follow analogously.
\end{proof}

The same holds, mutatis mutandis, for the infinity 
versions as well.


\subsection{\texorpdfstring{$\Ai$}{Ain} and the tilting category \texorpdfstring{$\T$}{T}}\label{sub-ainfty}
Recall that the category 
$\T_{\lambda}$ denotes the 
$\lambda$-block of $\T$ for $\lambda$ 
in the fundamental alcove $\mathcal A_0$, see Subsection~\ref{sub-blocks}. 
We fix any such $\lambda$ and denote everything 
using the simplified notation without the $\lambda$'s. In particular, recall our notation from Definition~\ref{defn-tiltgen} for the tilting generators $\Tn{\bbin}$ and $\Tn{\leq\bbm}$.
We denote by $\T_{\lambda}(\leq\bbm)$ the 
full subcategory of $\T_{\lambda}$ consisting of 
objects whose indecomposable summands are all from 
the set $\{\Tn{\lambda_0},\dots,\Tn{\lambda_m}\}$ for some fixed $m\in\bN$.
We also use completely similar notations for 
the walls $\mu$ by indicating this case with a superscript $t$. 

\begin{prop}\label{prop-isoofalg1}
We have
\[
\End_{\Uq}(\Tn{\leq\bbm})\cong \Am\quad\text{and}\quad\End_{\Uq}(\Tn{\leq\bbm}^t)\cong \At,
\]
which are isomorphisms of $\bC$-algebras.\makeqed
\end{prop}

\begin{proof}
We are going to construct these 
isomorphisms explicitly. Moreover, 
if we simply write $\bbi$ for $\Tn{\bbi}$, 
then we can visualize $\End_{\Uq}(\Tn{\leq\bbm})$ 
using Proposition~\ref{prop-link} for all $m>0$ as
\[
\xymatrix{
  \raisebox{0.1cm}{\xy(0,0)*{\bullet};(0,-2.5)*{\scriptstyle \bbm};\endxy} \ar@(ul,ur)|{\varepsilon_m} \ar@<-2pt>[r]_{d_{m}} & \cdots \ar@<-2pt>[l]_{u_{m-1}}\ar@<-2pt>[r]_/+0.2em/{d_{i+1}} &  \raisebox{0.1cm}{\xy(0,0)*{\bullet};(0,-2.5)*{\scriptstyle \bbi};\endxy} \ar@(ul,ur)|{\varepsilon_i} \ar@<-2pt>[l]_/-0.2em/{u_{i}}\ar@<-2pt>[r]_{d_{i}} & \cdots \ar@<-2pt>[l]_{u_{i-1}}\ar@<-2pt>[r]_/+0.2em/{d_{3}} & \raisebox{0.1cm}{\xy(0,0)*{\bullet};(0,-2.5)*{\scriptstyle \bbt};\endxy} \ar@(ul,ur)|{\varepsilon_2} \ar@<-2pt>[l]_/-0.2em/{u_{2}}\ar@<-2pt>[r]_{d_{2}} & \raisebox{0.1cm}{\xy(0,0)*{\bullet};(0,-2.5)*{\scriptstyle \bbo};\endxy} \ar@(ul,ur)|{\varepsilon_1} \ar@<-2pt>[l]_{u_{1}}\ar@<-2pt>[r]_{d_{1}} & \raisebox{0.1cm}{\xy(0,0)*{\bullet};(0,-2.5)*{\scriptstyle \bbz};\endxy}  \ar@<-2pt>[l]_{u_0}\\
}
\]
Thus, we consider the 
(by Proposition~\ref{prop-link} well-defined) $\bC$-algebra homomorphism
\[
\phi\colon \Am\to\End_{\Uq}(\Tn{\leq\bbm}), u_i\mapsto u_i, d_i\mapsto d_i,\varepsilon_i\mapsto \varepsilon_i
\]
for all suitable indices $i$. 
Using Corollary~\ref{cor-homspaces} 
we see that this is an isomorphism. 
The case $m=0$ and the 
``trivial case'' cases are immediate which proves the statement.
\end{proof}

The morphisms of $\End_{\Uq}(\Tn{\bbin})$ are 
``matrices'' $f=\prod_{i,i^{\prime}\in\bN} f_{i^{\prime}i}$ 
with $f_{i^{\prime}i}\colon \Tn{\bbi}\to\Tn{\bbi^{\prime}}$ and each 
row consists only of finitely many non-zero entries 
(this follows from Corollary~\ref{cor-homspaces}). 
Now $\End^{\mathrm{fs}}_{\Uq}(\Tn{\bbin})$ should be the 
(non-unital) subalgebra of $\End_{\Uq}(\Tn{\bbin})$ 
of all matrices with only finitely many non-zero rows, i.e.
\[
\End^{\mathrm{fs}}_{\Uq}(\Tn{\bbin})\cong \bigoplus_{i,i^{\prime}=0}^{\infty} \Hom_{\Uq}(\Tn{\bbi},\Tn{\bbi^{\prime}}).
\]
Similarly, of course, in the ``trivial'' case.

\begin{prop}\label{prop-isoofalg2}
We have
\begin{gather*}
\begin{aligned}
\End^{\mathrm{fs}}_{\Uq}(\Tn{\bbin})&\cong\bigoplus_{i,i^{\prime}=0}^{\infty}\Hom_{\Uq}(\Tn{\bbi},\Tn{\bbi^{\prime}})\cong \Ai,
\\
\End^{\mathrm{fs}}_{\Uq}(\Tn{\bbin}^t)&\cong\bigoplus_{i,i^{\prime}=0}^{\infty}\Hom_{\Uq}(\Tn{\bbi}^t,\Tn{\bbi^{\prime}}^t)\cong \Att,
\end{aligned}
\end{gather*}
which are isomorphisms of $\bC$-algebras.\makeqed
\end{prop}

\begin{proof}
We only show the first isomorphism since the other is clear.
By using Corollary~\ref{cor-homspaces} and the definition 
of $\End^{\mathrm{fs}}_{\Uq}(\Tn{\bbin})$ from above we 
see that, for each $f\in\End^{\mathrm{fs}}_{\Uq}(\Tn{\bbin})$, 
there exists some $m\in\bN$ such that $f_{i^{\prime}i}=0$ for $i,i^{\prime}>m$.
Those matrices describe elements in $\End_{\Uq}(\Tn{\leq\bbm})$ by ``forgetting'' the parts for $i,i^{\prime}>m$. Thus, we have
\[
\End^{\mathrm{fs}}_{\Uq}(\Tn{\bbin})\cong \varinjlim \End_{\Uq}(\Tn{\leq\bbm}).
\]
By Proposition~\ref{prop-isoofalg1} we see that this implies
\[
\End^{\mathrm{fs}}_{\Uq}(\Tn{\bbin})\cong \varinjlim \Am\cong\Ai.
\]
This finishes the proof.
\end{proof}

We are now able to relate 
$\T_{\lambda}(\leq \bbm)$ to 
$\pModAm$ and thus, $\T_{\lambda}$ to $\pModAi$. 
We define
\begin{gather*}
\bV\colon\T_{\lambda}(\leq \bbm)\to\ModAm, \bV(M)=\Hom_{\Uq}(\Tn{\leq\bbm},M),
\\
\bVi\colon\T_{\lambda}\to\ModAi, \bVi(M)=\Hom^{\mathrm{fs}}_{\Uq}(\Tn{\bbin},M)\cong \bigoplus_{i=0}^{\infty}\Hom_{\Uq}(\Tn{\bbi},M),
\end{gather*}
for $M$ in either $\Ob(\T(\leq \bbm))$ or $\Ob(\T_{\lambda})$ 
respectively where the isomorphism follows from 
Proposition~\ref{prop-isoofalg2}. Similarly for 
the walls $\mu$ where we indicate the functors 
$\bV^t$ and $\bVi^t$ (and everything else) with 
superscripts $t$. We note that the hom-spaces above 
are right $\Am$-modules (or $\Ai$, $\At$ or $\Ait$-modules) by \textit{pre}-composition.

Note that $\Theta_s,\Theta_t\in\Ob(\End(\T_{\lambda}))$ 
from Definition~\ref{defn-trans} clearly descent 
to $\T_{\lambda}(\leq\bbm)$ for all $m\in\bN$. 
We denote these restrictions, by abuse of notation, 
also by $\Theta_s,\Theta_t$. Furthermore, 
recall the notation $\Theta_{s\text{ or }t}^i$ for $i\in\bZ_{>0}$ 
which are defined via ($p_{\bbi}$ projects to the $\Tn{\bbi}$-part)
\[
\Theta_{s\text{ or }t}^i=\begin{cases}p_{\bbi}\circ\Theta_{s}, &\text{if }i\text{ is even },\\
p_{\bbi}\circ\Theta_{t}, &\text{if }i\text{ is odd }.
\end{cases}
\]
We use the 
same conventions and notations in the finite case by killing everything above $m$.

Similarly: the functors $\mathcal T_{\lambda}^{\mu}$ 
and $\mathcal T^{\lambda}_{\mu}$ from Definition~\ref{defn-onout} also 
descent to the finite case where we again denote them by the same symbols.

We consider the right side for the tilting embedding theorem 
(finite and infinite) as ungraded modules/functors. 
Moreover $\pMod$ always denotes categories of projective modules.

\begin{thm}(\textbf{Tilting embedding theorem - finite version})\label{thm-struktursatz1}
We have the following.
\begin{itemize}
\item[(a)] The functor $\bV$ is fully faithful.
\item[(b)] The functor $\bV$ sends objects 
$M\in\Ob(\T_{\lambda}(\leq\bbm))$ to projective 
$\Am$-modules and is an equivalence  between $\T_{\lambda}(\leq\bbm)$ and $\pModAm$.
\item[(c)] We have the following commuting diagram.
\[
\begin{xy}
  \xymatrix{
      \T_{\lambda}(\leq\bbm) \ar[r]^/-0.1em/{\bV} \ar[d]_{\Theta^i_{s\text{ or }t}}    &   \pModAm \ar[d]^{\mathcal U_{i}} \\
      \T_{\lambda}(\leq\bbm) \ar[r]_/-0.1em/{\bV}             &   \pModAm.  
  }
\end{xy}
\]
\item[(d)] We have the following commuting diagrams.
\[
\begin{xy}
  \xymatrix{
      \T_{\lambda}(\leq\bbm) \ar[r]^/-0.1em/{\bV} \ar[d]_{\Theta_s}    &   \pModAm \ar[d]^{\mathcal U_{\mathrm{even}}}  \\
      \T_{\lambda}(\leq\bbm) \ar[r]_/-0.1em/{\bV}             &   \pModAm   
  }
\end{xy}
\quad\raisebox{-0.75cm}{\text{and}}\quad
\begin{xy}
  \xymatrix{
      \T_{\lambda}(\leq\bbm) \ar[r]^/-0.1em/{\bV} \ar[d]_{\Theta_t}    &   \pModAm \ar[d]^{\mathcal U_{\mathrm{odd}}}  \\
      \T_{\lambda}(\leq\bbm) \ar[r]_/-0.1em/{\bV}             &   \pModAm.  
  }
\end{xy}
\]
\item[(e)] The same as in (a)+(b) works for the walls $\mu$ as well by using $\bV^t$.
\item[(f)] We have the following commuting diagrams.
\[
\begin{xy}
  \xymatrix{
      \T_{\lambda}(\leq\bbm) \ar[r]^/-0.1em/{\bV} \ar[d]_{\mathcal{T}_{\lambda}^{-1}}    &   \pModAm \ar[d]^{\mathcal U_{\mathrm{even}}^t}  \\
      \T_{-1}(\leq\bbm)\ar[r]_/-0.1em/{\bV^t}             &   \pModAt  
  }
\end{xy}
\quad\raisebox{-0.75cm}{\text{and}}\quad
\begin{xy}
  \xymatrix{
      \T_{\lambda}(\leq\bbm) \ar[r]^/-0.1em/{\bV} \ar[d]_{\mathcal{T}_{\lambda}^{l-1}}    &   \pModAm \ar[d]^{\mathcal U_{\mathrm{odd}}^t}  \\
      \T_{l-1}(\leq\bbm)\ar[r]_/-0.1em/{\bV^t}             &   \pModAt.  
  }
\end{xy}
\]
\item[(g)] We have the following commuting diagrams.
\[
\begin{xy}
  \xymatrix{
      \T_{-1}(\leq\bbm) \ar[r]^/-0.1em/{\bV^t} \ar[d]_{\mathcal{T}^{\lambda}_{-1}}    &   \pModAt \ar[d]^{\mathcal U^{\mathrm{even}}_t}  \\
      \T_{\lambda}(\leq\bbm) \ar[r]_/-0.1em/{\bV}             &   \pModAm  
  }
\end{xy}
\quad\raisebox{-0.75cm}{\text{and}}\quad
\begin{xy}
  \xymatrix{
      \T_{l-1}(\leq\bbm) \ar[r]^/-0.1em/{\bV^t} \ar[d]_{\mathcal{T}^{\lambda}_{l-1}}    &   \pModAt \ar[d]^{\mathcal U^{\mathrm{odd}}_t}  \\
      \T_{\lambda}(\leq\bbm) \ar[r]_/-0.1em/{\bV}             &   \pModAm.  
  }
\end{xy}
\hspace{1.95cm}
\raisebox{-1.55cm}{\makeqed}
\hspace{-1.95cm} 
\]
\end{itemize}
\end{thm}

\begin{proof}
We are careless with the indices in the proof 
and hope that this does not cause confusion. All 
appearing indices should be such that everything 
below $0$ or above $m$ is defined to be zero.

\textbf{(a).} By Proposition~\ref{prop-tilt} part 
(c) and Lemma~\ref{lem-tiltcat} part (a) it is enough 
to verify the statement for $M=\Tn\bbi$ for $i=0,\dots,m$. 
Moreover, by using Corollary~\ref{cor-homspaces} together 
with the discussion in Subsection~\ref{sub-ammod}, we 
get isomorphisms $\bV(\Tn\bbi)=\Hom_{\Uq}(\Tn{\leq \bbm},\Tn\bbi)\cong {}_iP$. 
Thus, $\bV$ is fully faithful.

\textbf{(b).} From (a) by using the discussion in 
Subsection~\ref{sub-ammod}.

\textbf{(c).} This follows by combining 
$\bV(\Tn\bbi)=\Hom_{\Uq}(\Tn{\leq \bbm},\Tn\bbi)\cong {}_iP$ 
with~\eqref{eq-tlcat} and part (c) of Proposition~\ref{prop-functors}.

\textbf{(d).} This follows directly from (c) 
and the definition of $\mathcal U_{\mathrm{even}}$ and $\mathcal U_{\mathrm{odd}}$.

\textbf{(e).} The ``trivial'' cases are clear.

\textbf{(f).} We do the left case here. Note that $\bV(\Tn{\bbi})\cong {}_iP$. 
Thus, using~\eqref{eq-ufunctors}, we see that 
$\mathcal U^{\mathrm{even}}_t\bV^t$ sends 
$\Tn{\bbi}$ to $\bC\bbi\oplus\bC\varepsilon_i$ 
for even $i$ and $\bC u_{i-1}\oplus \bC d_{i+1}$ for odd $i$.
This, by the discussion in Subsection~\ref{sub-endo}, 
is isomorphic as right $\At$-modules to 
$\bC\bbi\oplus\bC\bbi$ and $\bC\bbid\oplus\bC\bbiu$ 
respectively. In addition, $\mathcal T^{-1}_{\lambda}$ 
sends $\Tn{\bbi}$ to $\Tn{\mu_i}\oplus\Tn{\mu_i}$ 
(for $i$ even) or to $\Tn{\mu_{i-1}}\oplus\Tn{\mu_{i+1}}$ 
(for $i$ odd). These are then sent to $\bC\bbi\oplus\bC\bbi$ (for $i$ even) 
respectively $\bC\bbid\oplus\bC\bbiu$ (for $i$ odd).

\textbf{(g)}: Analogously to (f).
\end{proof}

\begin{thm}(\textbf{Tilting embedding theorem - infinite version})\label{thm-struktursatz2}
We have the following.
\begin{itemize}
\item[(a)] The functor $\bVi$ is fully faithful.
\item[(b)] The functor $\bVi$ sends objects $M\in\Ob(\T)$ to projective $\Ai$-modules and is an equivalence  between $\T_{\lambda}$ and $\pModAi$.
\item[(c)] We have the following commuting diagram.
\[
\begin{xy}
  \xymatrix{
      \T_{\lambda} \ar[r]^/-0.75em/{\bVi} \ar[d]_{\Theta^i_{s\text{ or }t}}    &   \pModAi \ar[d]^{\mathcal U^{\infty}_{i}}  \\
      \T_{\lambda} \ar[r]_/-0.75em/{\bVi}             &   \pModAi.   
  }
\end{xy}
\]
\item[(d)] We have the following commuting diagrams.
\[
\begin{xy}
  \xymatrix{
      \T_{\lambda} \ar[r]^/-0.75em/{\bVi} \ar[d]_{\Theta_s}    &   \pModAi \ar[d]^{\mathcal U^{\infty}_{\mathrm{even}}}  \\
      \T_{\lambda} \ar[r]_/-0.75em/{\bVi}             &   \pModAi   
  }
\end{xy}
\quad\raisebox{-0.75cm}{\text{and}}\quad
\begin{xy}
  \xymatrix{
      \T_{\lambda} \ar[r]^/-0.75em/{\bVi} \ar[d]_{\Theta_t}    &   \pModAi \ar[d]^{\mathcal U^{\infty}_{\mathrm{odd}}}  \\
      \T_{\lambda} \ar[r]_/-0.75em/{\bVi}             &   \pModAi.  
  }
\end{xy}
\]
\item[(e)] The same as in (a)+(b) works for the walls $\mu$ as well by using $\bVi^t$.
\item[(f)] We have the following commuting diagrams.
\[
\begin{xy}
  \xymatrix{
      \T_{\lambda}(\bbin) \ar[r]^/-0.1em/{\bVi} \ar[d]_{\mathcal{T}_{\lambda}^{-1}}    &   \pModAi \ar[d]^{\mathcal U_{\mathrm{even}}^t}  \\
      \T_{-1}(\bbin) \ar[r]_/-0.1em/{\bVi^t}             &   \pModAtt
  }
\end{xy}
\quad\raisebox{-0.75cm}{\text{and}}\quad
\begin{xy}
  \xymatrix{
      \T_{\lambda}(\bbin) \ar[r]^/-0.1em/{\bVi} \ar[d]_{\mathcal{T}_{\lambda}^{l-1}}    &   \pModAi \ar[d]^{\mathcal U_{\mathrm{odd}}^t}  \\
      \T_{l-1}(\bbin) \ar[r]_/-0.1em/{\bVi^t}             &   \pModAtt.  
  }
\end{xy}
\]
\item[(g)] We have the following commuting diagrams.
\[
\begin{xy}
  \xymatrix{
      \T_{-1}(\bbin) \ar[r]^/-0.1em/{\bVi^t} \ar[d]_{\mathcal{T}^{\lambda}_{-1}}    &   \pModAtt \ar[d]^{\mathcal U^{\mathrm{even}}_t}  \\
      \T_{\lambda}(\bbin) \ar[r]_/-0.1em/{\bVi}             &   \pModAi  
  }
\end{xy}
\quad\raisebox{-0.75cm}{\text{and}}\quad
\begin{xy}
  \xymatrix{
      \T_{l-1}(\bbin) \ar[r]^/-0.1em/{\bVi^t} \ar[d]_{\mathcal{T}^{\lambda}_{l-1}}    &   \pModAtt \ar[d]^{\mathcal U^{\mathrm{odd}}_t}  \\
      \T_{\lambda}(\bbin) \ar[r]_/-0.1em/{\bVi}             &   \pModAi.  
  }
\end{xy}
\hspace{2.3cm}
\raisebox{-1.55cm}{\makeqed}
\hspace{-2.3cm}
\]
\end{itemize}
\end{thm}

\begin{proof}
The proofs of (a)-(g) is similar to 
the proofs given in Theorem~\ref{thm-struktursatz1} 
by using arguments as in Proposition~\ref{prop-isoofalg2}. We leave it to the reader.
\end{proof}

\begin{rem}\label{rem-othermaps}
We encourage the reader to work out the 
translations onto and out of the wall 
commuting diagrams similar to part (c) of the theorems above.
\end{rem}

We note the following corollary for completeness.
\begin{cor}\label{cor-struktursatz}
For all $\lambda,\lambda^{\prime}\in\mathcal A_0$ we have $\T_{\lambda}\cong \T_{\lambda^{\prime}}$. Thus,
\[
\T(\leq\bbm)\cong \bigoplus_{i=0}^{l-2}\pModAm\oplus\bigoplus_{i=0}^{1}\pModAt,
\quad\T\cong \bigoplus_{i=0}^{l-2}\pModAi\oplus\bigoplus_{i=0}^{1}\pModAtt,
\]
with semisimple categories $\pModAt$ and $\pModAtt$.\makeqed
\end{cor}

\begin{proof}
This follows from Lemma~\ref{lem-block} 
and Theorems~\ref{thm-struktursatz1} 
and~\ref{thm-struktursatz2} because the 
right-hand sides are always the same independent of $\lambda\in\mathcal A_0$.
\end{proof}

\begin{ex}(\textbf{How to compute the images of the translation functors on $\Uq$-morphisms})\label{ex-howtofunctor}
Combining the commuting diagrams of the 
Theorems~\ref{thm-struktursatz1} and~\ref{thm-struktursatz2} 
we obtain an explicit way to compute $\Theta_s$ and $\Theta_t$ 
not just on the $\Tn{\bbi}$'s (as we already know by part (c) 
of Proposition~\ref{prop-functors}), but also 
on the $\Uq$-intertwiners in $\Hom_{\Uq}(\Tn{\bbi},\Tn{\bbi^{\prime}})$.

This follows by using~\eqref{eq-homps} together 
with $\Theta_s\leftrightsquigarrow \mathcal U_{\mathrm{even}}$ 
and $\Theta_t\leftrightsquigarrow \mathcal U_{\mathrm{odd}}$. 
For example, given $u_0\colon\Tn{\bbz}\to\Tn{\bbo}$, we can 
compute $\Theta_t(u_0)$ as follows. First note that 
$\bVi\Theta_t(\Tn{\bbz})\cong {}_1P$ and $\bVi\Theta_t(\Tn{\bbo})\cong {}_1P\oplus {}_1P$ by~\eqref{eq-tlcat} and Theorem~\ref{thm-struktursatz2}. By using~\eqref{eq-homps} 
we see
\[
\Theta_t(u_0)=\begin{pmatrix}
 0 \\
 \mathrm{id}
\end{pmatrix}\colon \Tn{\bbo}\to\Tn{\bbo}\oplus \Tn{\bbo}.
\]
We see that this is actually
\[
\Theta_t(u_0)=\begin{pmatrix}
 0 \\
 \mathrm{id} 
\end{pmatrix}\colon \Tn{\bbo}\to\Tn{\bbo}\langle-1\rangle\oplus \Tn{\bbo}\langle+1\rangle,
\]
by using the results we are going to explain in Subsection~\ref{sub-grad}. Indeed, 
the grading is very helpful to calculate the images of the $\Uq$-interwiners, e.g. 
$\Theta_t(u_0)$ can only be zero on the first component, since we will show 
that there will not be  suitable $\Uq$-intertwiners lowering the degree.
\end{ex}

\begin{rem}\label{rem-cato}
Let $\mathfrak{g}=\mathfrak{sl}_{m+1}$ and 
$\mathfrak{p}=\mathfrak{p}_1$ be a maximal 
parabolic subalgebra for $S_1\times S_{m-1}$ 
(for a definition see for example~\cite[Chapter~9]{hum}). 
Denote by $\mathcal O^{\mathfrak{p}}$ the corresponding 
parabolic category $\mathcal O$ (for a definition and 
some properties see e.g.~\cite[Section~9.3]{hum}). Moreover, 
denote by $\mathcal O^{\mathfrak{p}}_0$ the principle block. 
As Khovanov and Seidel show in~\cite[Proposition~2.9]{ks1} 
there is an equivalence of categories 
$\mathcal O^{\mathfrak{p}}_0\cong \ModAm$. 
Thus, by Theorems~\ref{thm-struktursatz1} 
and~\ref{thm-struktursatz2}, we see that 
$\T_{\lambda}$ governs $\mathcal O^{\mathfrak{p}}_0$ for all $m$.
Similarly for the projective endofunctors: 
as explained in~\cite[below Proposition~2.9]{ks1}, 
the functors $\mathcal U_i$ correspond to the translation 
through the $i$-th wall (note that, under the equivalence 
above, the $\mathfrak{sl}_{m+1}$ has Weyl group generators 
$s_1,\dots,s_m$ and therefore translation functors indexed 
by $i=1,\dots,m$). Moreover, $\mathcal U_i^t$ corresponds 
to the translation onto and $\mathcal U^i_t$ corresponds to 
the translation out off the wall functors.
Hence, by Theorems~\ref{thm-struktursatz1} 
and~\ref{thm-struktursatz2}, we can say 
that these are governed by the combinatorics 
of $\mathcal T_{\lambda}^{\mu}$, 
$\mathcal T^{\lambda}_{\mu}$ 
and $\Theta^{\lambda}_{s\text{ or }t}$ from our tilting case.

Another point is, as Khovanov and Seidel explain 
after~\cite[Proposition~2.9]{ks1}, that the path length grading gives rise 
to a \textit{Koszul grading} in the sense of~\cite{bgs} on 
the $\Am$'s, see also~\cite[Example~1.1]{st1}. 
The same holds for $\Ai$. This can either be 
seen ``by hand'' or by using a more general 
theory for (quotients of) not necessary 
finite quiver algebras that can be found for example in~\cite{mos}.
\end{rem}

\begin{rem}\label{rem-webalg}
Khovanov and Seidel's quiver is also 
related to Khovanov's arc algebra 
$H^m$ (for $m$ even) that he introduced 
in~\cite{kh4} to give an algebraic interpretation 
of Khovanov homology. The arc algebra categorifies the 
invariant tensors of $V^{\otimes m}$ where $V$ is the 
two dimensional vector representation of $\Uv$.
In addition, together with Chen he extended in~\cite{chkh} 
this categorification to the full tensor 
product $V^{\otimes m}$ by defining a certain 
subquotient $\prod_k A^{k,m-k}$ (this quotient can 
be seen as the quasi-hereditary cover of $H^m$). 
All of these have a topological interpretation 
as algebras consisting of cobordisms, 
see e.g.~\cite[Section~2]{chkh}, and they are 
graded by \textit{the Euler characteristic} of these cobordisms.
As explained in~\cite[Section~3]{chkh}, 
the $A^{1,m-1}$-part of this subquotient is graded 
isomorphic to $\Am$. Thus, we see that $\T$ governs $A^{1,m-1}$ for all $m$.
Moreover, this gives a hint for a  
generalization of our work: we think 
that analogs of the Theorems~\ref{thm-struktursatz1} and~\ref{thm-struktursatz2} 
can be proven for certain subquotients of the $\mathfrak{sl}_n$ 
generalizations of $H^m$ studied in e.g.~\cite{mack1},~\cite{mpt},~\cite{tub3} or~\cite{tub4}.
\end{rem}

\subsection{The tilting category \texorpdfstring{$\T$}{T} and its graded counterpart \texorpdfstring{$\Tgr$}{Tgr}}\label{sub-grad}
In this subsection we discuss the 
consequences of Theorems~\ref{thm-struktursatz1} 
and~\ref{thm-struktursatz2} with respect to the 
question how the path length grading $l(\cdot)$ 
(see Definition~\ref{def-ksquiver1}) of 
Khovanov-Seidel's $\infty$-quiver algebra relates 
to the tilting category $\T$. We point out that this is non-trivial since Theorems~\ref{thm-struktursatz1} and~\ref{thm-struktursatz2} only say that 
$\T_{\lambda}(\leq \bbm)$ and $\T_{\lambda}$ are 
isomorphic to subcategories of $\pModAm$ and $\pModAi$.

Recall that, for every graded algebra $A$, there is a 
\textit{forgetful functor} 
$\mathrm{forget}\colon\ModgrA\to\ModA$ 
that forgets 
the grading. Using this 
functor, we say that an $A$-module 
$M\in\Ob(\ModA)$ is \textit{gradable} if there 
exists an $A$-module $\tilde M\in\Ob(\ModgrA)$ 
such that $\mathrm{forget}(\tilde M)=M$. Note 
that, by abuse of notation, we usually do not distinguish between $\tilde M$ and $M$.

Recall that $\Hom_{\ModA}(M,M^{\prime})$ 
consists of all (right) $A$-module 
homomorphisms $f\colon M\to M^{\prime}$ 
and $\Hom_{\ModgrA}(\tilde M,\tilde M^{\prime})$ 
consists of (right) $A$-module homomorphisms 
$\tilde f\colon \tilde M\to \tilde M^{\prime}$ as 
in~\eqref{eq-gradhom}. Hence, in the same vein as 
above, we can call an $A$-module homomorphism 
$f\colon M\to M^{\prime}$ \textit{gradable}, 
if there exists $\tilde f\colon\tilde M\to \tilde M^{\prime}$ 
such that $\mathrm{forget}(\tilde f)=f$.
Consequently, we define the following.

\begin{defn}\label{defn-gradable}
We call a (left) $\Uq$-tilting module 
$M\in\Ob(\T)$ \textit{gradable}, if $\bVi(M)$ is 
gradable viewed as a (right) $\Ai$-module. We call a 
$\Uq$-intertwiner $f\colon M\to M^{\prime}\in\Hom_{\T}(M,M^{\prime})$ 
\textit{gradable}, if $\bVi(f)\colon \bVi(M)\to \bVi(M^{\prime})$ is 
a gradable $\Ai$-module homomorphism.

We denote by $\Tgr$ the subcategory of $\T$ consisting 
of gradable $\Uq$-tilting modules and gradable 
$\Uq$-intertwiners. We use similar notations and 
conventions for $\leq\bbm$ and for fixed $\lambda$ or $\mu$.
\end{defn}

The question concerning which objects $M\in\Ob(\T)$ 
are gradable turns out to be surprisingly 
simple in our case. We note that, with respect to 
the rather complicated situation for 
graded category $\mathcal O_0^{\mathfrak{p}}$ 
discussed in e.g.~\cite[Theorem~4.1]{st2}, 
the following proposition is quite remarkable.

\begin{prop}\label{prop-grading}
All objects $M\in\Ob(\T_{\lambda}(\leq\bbm))$, 
$M\in\Ob(\T_{\lambda})$ and $M\in\Ob(\T)$ are gradable. The grading of 
each indecomposable module $\Tn{\bbi}$ is unique up to shifts.\makeqed
\end{prop}

\begin{proof}
We only prove the $\infty$ case and leave 
the other case to the reader. We start by 
considering $\lambda\in\mathcal A_0$ and 
discuss the semisimple case on the walls afterwards.

It is easy to check that, for any graded algebra 
$A$ and any two gradable $A$-modules $M$ and 
$M^{\prime}$, the direct sum $M\oplus M^{\prime}$ 
is also gradable. Thus, by Lemma~\ref{lem-tiltcat} 
part (a), it is enough to consider only the 
$\Tn{\bbi}$'s. As in the proof of the part (b) 
of Theorem~\ref{thm-struktursatz2} 
(or Theorem~\ref{thm-struktursatz1} in the finite case) 
we see that $\bVi(\Tn{\bbi})\cong {}_iP$. As discussed 
in Subsection~\ref{sub-ammod}, the ${}_iP$ can be given 
a grading coming from the path length $l(\cdot)$. 
Thus, all the $\Tn{\bbi}$'s are gradable.

On the walls: Since this is the semisimple 
case by Lemma~\ref{lem-block}, we can just 
assign to any simple $\bVi(\Tn{\mu_i})$ a 
degree by demanding that all elements of 
$\bVi(\Tn{\mu_i})$ are concentrated in 
this particular degree. Thus, all 
$\Tn{\mu_i}$ are gradable and concentrated in (up to shifts) degree zero.

The uniqueness (up to shifts) can be proven as in~\cite[Lemma~2.5.3]{bgs}.
\end{proof}

\begin{cor}\label{cor-gradings2}
We have isomorphisms of ungraded categories 
$\T_{\lambda}(\leq\bbm)\cong_{\mathrm{iso}}\Tgr(\leq\bbm)$ 
and $\T\cong_{\mathrm{iso}}\Tgr$. That is, 
also all morphisms are gradable and the gradings are unique up to a shift.\makeqed
\end{cor}

\begin{proof}
As before, for any 
two gradable $A$-homomorphisms $f$ and 
$f^{\prime}$, the direct sum $f\oplus f^{\prime}$ 
is also gradable. Thus, by Lemma~\ref{lem-tiltcat} 
part (a) again together with Proposition~\ref{prop-grading}, 
it is enough to consider only 
$\Hom_{\Uq}(\Tn{\bbi},\Tn{\bbi^{\prime}})$. 
By Corollary~\ref{cor-homspaces} we see that 
this space has a basis consisting of a subset 
of $\bbi,u_i,d_i$ and $\varepsilon_i$ (in the non-semisimple case) who are all gradable by Proposition~\ref{prop-isoofalg2}. The semisimple 
case and the uniqueness of the grading follow as 
above. This gives rise to an 
isomorphism after collapsing the grading of $\Tgr$ 
(since each object $M$ of $\T$ corresponds to objects of $\Tgr$ 
of the form $M\langle s\rangle$ for $s\in\bZ$).
\end{proof}

Using Corollary~\ref{cor-gradings2} and 
abuse of language, we do not distinguish 
between $\Uq$-tilting modules or their graded versions.
Moreover, as a consequence of 
Corollary~\ref{cor-gradings2}, we can 
choose a \textit{standard grading} 
(since it will be unique up to shifts) 
by demanding that simple $\Uq$-tilting modules 
are concentrated in degree zero and 
proceed inductively along the quiver. 
Note that, by Propositions~\ref{prop-weyl} 
and~\ref{prop-tilt} and Corollary~\ref{cor-weyl}, 
this gives inductively a graded lift of 
$\Ln{\bbi}$, $\Vn{\bbi}$ and $\dVn{\bbi}$ as well. 
This choice gives rise to the following graded 
refinements of Propositions~\ref{prop-weyl} 
and~\ref{prop-tilt} and Corollary~\ref{cor-weyl}.

\begin{prop}\label{prop-refine}
Suppose $i=al+b$ for some $a,b\in\bN$ with 
$b\leq l-2$. Set $i^{\prime}=(a+2)l-b-2$. 
Then there
exist short exact sequences
\begin{gather*}
0\longrightarrow \Ln{i}\langle+1\rangle\hooklongrightarrow\Vn{i^{\prime}}\twoheadlongrightarrow \Ln{i^{\prime}}\longrightarrow 0,
\hspace*{0.25cm}
0\longrightarrow \Ln{i^{\prime}}\hooklongrightarrow\dVn{i^{\prime}}\twoheadlongrightarrow \Ln{i}\langle-1\rangle\longrightarrow 0,
\\
0\longrightarrow \Vn{i^{\prime}}\hooklongrightarrow\Tn{i^{\prime}}\twoheadlongrightarrow \Vn{i}\langle-1\rangle\longrightarrow 0,
\hspace*{0.25cm}
0\longrightarrow \dVn{i}\langle+1\rangle\hooklongrightarrow\Tn{i^{\prime}}\twoheadlongrightarrow \dVn{i^{\prime}}\longrightarrow 0.
\end{gather*}
Moreover, all involved morphisms are degree preserving.\makeqed
\end{prop}

\begin{proof}
This follows from our choice for the grading convention 
and the degree (under the equivalence in 
Theorems~\ref{thm-struktursatz1} and~\ref{thm-struktursatz2}) 
of the morphisms in $\Hom_{\Uq}(\Tn{\bbi},\Tn{\bbi^{\prime}})$. 
To be more precise, we know that the $\Tn{\bbi}$'s will 
be, under the functor $\bVi$, 
mapped to ${}_iP$. Then for example, 
as explained in Proposition~\ref{prop-link}, the unique morphism
\[
\Tn{\bbi}\twoheadrightarrow\dVn{\bbi}\hookrightarrow\Tn{\bbiu}\leftrightsquigarrow\xymatrix{
  \raisebox{0.1cm}{\xy(0,0)*{\bullet};(0,-2.5)*{\scriptstyle \bbiu};\endxy}\!\!   &  \raisebox{0.1cm}{\xy(0,0)*{\bullet};(0,-2.5)*{\scriptstyle \bbi};\endxy} \ar[l]_/-0.2em/{u_i}\\
}
\]
is of degree $1$. By our convention 
we see that moving to the left along 
the KS $\infty$-quiver always increases 
the degree (starting in degree zero for the first simple module). Thus,
\[
0\longrightarrow \dVn{i}\langle+1\rangle\hooklongrightarrow\Tn{i^{\prime}}\twoheadlongrightarrow \dVn{i^{\prime}}\longrightarrow 0.
\]
All other cases follow similarly 
and are left to the reader. Note again 
that moving to the left along the KS 
$\infty$-quiver always increases the degree, 
but, by duality, moving right decreases it.
\end{proof}

\begin{rem}\label{rem-webalg2}
As we already mentioned in 
Remark~\ref{rem-webalg} the grading induced on 
$\T$ comes from an Euler characteristic on a 
certain cobordism category associated to 
Khovanov's arc algebra $H^m$. Thus, this is 
a ``natural'' grading from the viewpoint of topology.

As we mentioned above, we think that this 
should generalize in type $\boldsymbol{\mathrm{A}}$. The degree will 
there be given by an Euler characteristic on 
a certain ``foam'' category associated to the 
$\mathfrak{sl}_n$-web algebras which generalize Khovanov's arc algebra. See for example~\cite{lqr1},~\cite{mpt},~\cite{qr1} or~\cite{tub4}. 
Or alternatively, by a grading coming from certain cyclotomic KL-R algebras in the sense of~\cite{kl5} or~\cite{rou}.
\end{rem}

We conclude this section by lifting the translation 
onto $\mathcal T^{\mu}_{\lambda}$, the translation out 
of $\mathcal T_{\mu}^{\lambda}$ and the 
translation through the wall $\Theta_s$ and $\Theta_t$ 
to their graded versions (we use the same notation for these). 
To understand our notation, recall that a category 
$\tilde{\mathcal C}$ is called a \textit{graded} 
category, if there exists a functor 
$\mathrm{deg}\colon\tilde{\mathcal C}\to\mathcal Z$ 
where we consider $\bZ$ as a category $\mathcal Z$ 
with $\Ob(\mathcal Z)=\{\bullet\}$ and 
$\End_{\mathcal Z}(\bullet)=\bZ$. In addition, there is a 
$\bZ$-action $\langle s\rangle$ on the objects 
$O\in\Ob(\tilde{\mathcal C})$, called \textit{shift} 
by $s\in\bZ$, such that a given morphism 
$f\in\Hom_{\tilde{\mathcal C}}(O,O^{\prime})$ of 
degree $d$ is of degree $d+s^{\prime}-s$ in 
$\Hom_{\tilde{\mathcal C}}(O\langle s\rangle,O^{\prime}\langle s^{\prime}\rangle)$ 
(said otherwise, 
every morphism in $\tilde{\mathcal C}$ has a degree which behaves 
additively under composition and objects can be shifted via 
$\langle s\rangle$). Moreover, denote by $\mathcal C$ the 
\textit{ungraded version}. We say a functor 
$\mathcal F\colon\mathcal C\to\mathcal D$ is gradable 
if there exists a \textit{lift} 
$\tilde{\mathcal F}\colon\tilde{\mathcal C}\to\tilde{\mathcal D}$ such that
$\mathrm{forget}\circ\tilde{\mathcal F}=\mathcal F\circ\mathrm{forget}$.
We say that $\tilde{\mathcal F}$ 
and $\tilde{\mathcal F}^{\prime}$ are the 
\textit{same up to a shift $s\in\bZ$} if there 
is a natural isomorphism\footnote{In the $2$-category 
where objects are graded categories, where morphisms 
are graded functors between these categories and 
where $2$-morphisms are graded natural transformations.} 
between $\tilde{\mathcal F}$ and $\tilde{\mathcal F}^{\prime}\langle s\rangle$.
As before, we are usually very 
careless with our distinction 
of $\tilde{\mathcal F}$ and $\mathcal F$.

\begin{cor}\label{cor-funcgrad2}
The functors 
$\mathcal T_{\lambda}^{\mu,i},\mathcal T_{\mu}^{\lambda,i},\mathcal T_{\lambda}^{\mu},\mathcal T_{\mu}^{\lambda}\in\Ob(\Endo(\T(\leq\bbm)))$
are gradable (same in the infinite case). 
Moreover, any two lifts of 
$\mathcal T_{\lambda}^{\mu,i},\mathcal T^{\lambda,i}_{\mu}$ are the same up to shifts.\makeqed
\end{cor}

\begin{proof}
This follows by Remark~\ref{rem-othermaps} 
and the fact that $\mathcal U_{i}^t$ and $\mathcal U^i_t$ are 
graded functors (similarly in the $\infty$ case). 
The statement for 
$\mathcal T_{\lambda}^{\mu}$ and $\mathcal T_{\mu}^{\lambda}$ follows now 
by~\eqref{eq-projectivefunc2}, since clearly sums of gradable, 
additive functors in an additive category are 
gradable. For the infinite case note that any sum (finite or not) 
of gradable functors is gradable and the claim 
follows from~\eqref{eq-projectivefunc2} 
and the finite case.
The uniqueness up to shifts can again be verified 
as in~\cite[Lemma 2.5.3]{bgs}, since 
$\mathcal T_{\lambda}^{\mu,i}$ and 
$\mathcal T_{\mu}^{\lambda,i}$ are 
indecomposable by Proposition~\ref{prop-profunc}. 
\end{proof}

\begin{cor}\label{cor-funcgrad1}
The functors 
$\Theta^{\lambda,i}_s,\Theta^{\lambda,i}_t,\Theta^{\lambda}_s,\Theta^{\lambda}_t\in\Ob(\Endo(\T_{\lambda}(\leq\bbm)))$ 
are 
gradable (same in the infinite case). 
Moreover, any two lifts of $\Theta^{\lambda,i}_{s\text{ or }t}$ are the same up to a 
shift.\makeqed
\end{cor}

\begin{proof}
We can use Corollary~\ref{cor-funcgrad2} 
since all 
functors involved can be obtained via sums and compositions of the functors 
$\mathcal T_{\lambda,i}^{\mu}$ and $\mathcal T^{\lambda,i}_{\mu}$. The infinite case follows as above, but using~\eqref{eq-projectivefunc} 
instead of~\eqref{eq-projectivefunc2}. 
By Proposition~\ref{lem-profuncinde} the functors $\Theta^{\lambda,i}_{s\text{ or }t}$ are 
indecomposable and the uniqueness up to shifts follows as before. 
\end{proof}

We choose the graded versions of 
$\mathcal T_{\lambda}^{\mu,i}$ and out of $\mathcal T^{\lambda,i}_{\mu}$ 
in the evident way as induced by the gradings on $\mathcal U_i^t$ or 
$\mathcal U^i_t$.
We get the following refinement of Proposition~\ref{prop-functors}.
\begin{cor}\label{cor-funcgrad3}
In the graded category $\Tgr$ we have for all 
$i\in\bN$ the following.
\begin{itemize}
\item[(a)] The functors 
$\mathcal T_{\lambda}^{\mu}$ and $\mathcal T^{\lambda}_{\mu}$ are 
well-defined (their definition gives gradable $\Uq$-tilting modules in the 
right blocks), (up to shifts) 
adjoints (left and right) and 
exact. Thus, $\Theta^{\lambda}_s$ and $\Theta^{\lambda}_t$ are exact and graded self-adjoint.
\item[(b)] We have (recalling that $\Ln{\mu_i}\cong\Tn{\mu_i}$)
\begin{gather*}
\mathcal T_{\lambda}^{\mu}(\Tn{\lambda_i})\cong\begin{cases}\Tn{\mu_{i-1}}\langle+1\rangle\oplus\Tn{\mu_{i+1}}\langle+1\rangle,  & \text{if }\mu_i>\lambda_i,\\
  \Tn{\mu_i}\oplus\Tn{\mu_i}\langle+2\rangle,  & \text{if }\mu_i<\lambda_i,\end{cases}
\\
\mathcal T^{\lambda}_{\mu}(\Tn{\mu_i})\cong\begin{cases}\Tn{\lambda_{i+1}}\langle-1\rangle,  & \text{if }\mu_i>\lambda_i,\\
 \Tn{\lambda_{i}}\langle-1\rangle,  & \text{if }\mu_i<\lambda_i.\end{cases}
\end{gather*}
\item[(c)] The dead-end relations $\Theta^{\lambda}_s(\Tn{\lambda_0})\cong0$, $\Theta^{\lambda}_s(\Tn{\lambda_1})\cong\Tn{\lambda_2}$ and $\Theta^{\lambda}_t(\Tn{\lambda_0})\cong\Tn{\lambda_1}$. Moreover, we have
\[
\Theta^{\lambda}_{s\text{ or }t}(\Tn{\lambda_i})\cong\begin{cases}\Tn{\lambda_{i-1}}\oplus\Tn{\lambda_{i+1}} ,  & \text{if }i>1\text{ is odd for }s\text{ and even for }t,\\
  \Tn{\lambda_{i}}\langle-1\rangle\oplus\Tn{\lambda_{i}}\langle+1\rangle ,  & \text{if }i>0\text{ is odd for }t\text{ and even for }s.\end{cases}
\]
\end{itemize}
Here we set $\Tn{-1}=\Tn{\lambda_{-1}}=\Tn{\mu_{-1}}=0$.\makeqed
\end{cor}

\begin{proof}
This follows directly from 
Theorem~\ref{thm-struktursatz2} 
parts (d), (f) and (g) together 
with~\eqref{eq-ufunctors} and the 
discussion in Subsection~\ref{sub-endo}, for example Lemma~\ref{lem-adjoint}.
\end{proof}

\begin{prop}\label{prop-endograd}
All functors in $\Endo(\T)$ are gradable. 
Moreover, the grading is unique up to shifts 
on the indecomposable projective functors. 
Similar for $\Endo(\T_{\lambda})$.\makeqed
\end{prop}

\begin{proof}
The uniqueness, up to shifts, of the grading on the indecomposable 
factors can be proven as before. By the Krull-Schmidt property of $\Endo(\T)$ 
(see Lemma~\ref{lem-profunc}) and the fact that any direct sum of gradable, 
additive functors in an additive category is 
gradable, it suffices to show that indecomposable 
projective functors are gradable. By 
Lemma~\ref{lem-profuncinde} these are given 
by tensoring with a $\Tn i$ for some $i\in\bN$. 
By Proposition~\ref{prop-grading} all objects 
of $\T$ are gradable and, by Lemma~\ref{lem-tiltcat} 
part (b), the category $\T$ is preserved by finite 
tensor products. Thus, the indecomposable projective functors are gradable. 
For 
$\Endo(\T_{\lambda})$ this follows from Corollary~\ref{cor-funcgrad1}.
\end{proof}

Proposition~\ref{prop-endograd} motivates 
the definition of $\Endo(\Tgrl)$ which can be 
defined in the spirit of~\eqref{eq-gradhom} (thus, using 
\textit{graded} hom-spaces of natural 
transformations). This is the same as 
saying that the category $\Endo(\Tgrl)$ is graded in the above sense.

In order to state the proposition, we write 
$\bVi\mathcal F(\Tn{\bbin})$ short for 
$\bigoplus_i\bVi\mathcal F(\Tn{\bbi})$. This 
is an $\Ai$-bimodule via pre-composition 
(right) and post-composition (left) with 
$\mathcal F(\cdot)$ or $\mathcal F^{(\prime)}(\cdot)$.

Moreover, recall that the superscript fs 
indicates that we are only considering finitely 
supported homomorphisms. We have the following 
(similar of course for $\Endo(\Tgr_{\lambda}(\leq\bbm))$).

\begin{prop}\label{prop-endograd2}
Let $\mathcal F,\mathcal F^{\prime}$ be two 
functors in $\Endo(\Tgrl)$. Then there exists 
an isomorphism of graded $\bC$-vector spaces
\[
\Hom^{\mathrm{fs}}_{\Endo(\Tgrl)}(\mathcal F,\mathcal F^{\prime})\cong \Hom^{\mathrm{fs}}_{\AipModgrAi}(\bVi\mathcal F(\Tn{\bbin}),\bVi\mathcal F^{\prime}(\Tn{\bbin})).
\]
This is an isomorphism of graded rings if $\mathcal F=\mathcal F^{\prime}$.\makeqed
\end{prop}

\begin{proof}
It 
suffices to verify the statement for the 
translation onto $\mathcal T^{\mu}_{\lambda}$, 
out of $\mathcal T_{\mu}^{\lambda}$ and the 
through the wall functors $\Theta^{\lambda}_{s\text{ or }t}$. 
By Theorem~\ref{thm-struktursatz2} part (d), (f) and (g) 
it suffices to show the statement for the corresponding $\mathcal U$'s.

Without loss of generality, we only discuss the 
``even'' case, that is, $\mathcal U^t_{\mathrm{even}}$, $\mathcal U_t^{\mathrm{even}}$ 
and $\mathcal U^{\infty}_{\mathrm{even}}$ which correspond to 
$\mathcal T^{-1}_{\lambda}$, $\mathcal T_{-1}^{\lambda}$ and 
$\Theta^{\lambda}_{s}$ respectively. Thus, by additivity, it 
suffices to show the statement for $\mathcal U^t_{i}$, 
$\mathcal U_t^{i}$ and $\mathcal U_i$ for $i>0$ even.
Fix $i$ and assume that $i\ll m$. Recall 
that Khovanov and Seidel have shown 
in~\cite[Proposition~2.9]{ks1} that 
$\pModAm$ and $\mathcal O^{\mathfrak{p}}_0$ 
for $\mathfrak{sl}_{m+1}$ and parabolic 
$\mathfrak{p}$ of type $S_1\times S_{m-1}$ 
are equivalent. Moreover, as they explain 
below~\cite[Proposition~2.9]{ks1}, 
$\mathcal U^t_{i}$, $\mathcal U_t^{i}$ and 
$\mathcal U_i$ correspond to the translation onto, out off 
and through the $i$-th wall functors.

Note now that the same holds for the graded 
version of $\mathcal O^{\mathfrak{p}}_0$ 
introduced in~\cite{st2} as Stroppel explains in~\cite[Section~1]{st1}. 
Thus, we can use~\cite[Theorem~1.12]{st1} to finish the proof in the 
``cut-off case''. Taking the inductive limit now shows the statement.
\end{proof}

\begin{cor}\label{cor-center}
There is an isomorphism of graded rings
\[
\End^{\mathrm{fs}}_{\mathrm{gr}}(\mathrm{id})\cong Z(\Ai),
\]
where $\mathrm{id}$ is the identity functor on 
$\Tgrl$, $Z(\Ai)$ is the center of 
$\Ai$ and the endomorphism ring is to be taken in $\Endo(\Tgrl)$.\makeqed
\end{cor}

\begin{proof}
From Proposition~\ref{prop-isoofalg2} 
combined with Proposition~\ref{prop-endograd2}. 
Note that we get the center $Z(\Ai)$ and not $\Ai$ 
itself because the endomorphism ring from 
Proposition~\ref{prop-isoofalg2} is as 
$\Uq$-intertwiners while the one induced from 
Proposition~\ref{prop-endograd2} is as $\AipModgrAi$-bimodule intertwiners.
\end{proof}

\begin{rem}\label{rem-grothendieck1a}
It is known 
that the split Grothendieck group 
$K_0^{\oplus}(\cdot)=K_0^{\oplus}(\cdot)\otimes_{\bZ[v,v^{-1}]}\bC$ 
(viewed as a module over $\bZ[v,v^{-1}]$ where the 
formal parameter $v$ comes for the grading) of 
$\pModgrAm$ \textit{categorifies} the Burau 
representation $\mathcal B_{m+1}$ of the 
$m+1$-strand braid group $B_{m+1}$ 
(the action of $B_{m+1}$ is induced via 
functors constructed from $\mathcal U_i$), see~\cite[Proposition~2.8 and Subsection 2.d]{ks1}. 
Using Theorem~\ref{thm-struktursatz1}, we see that the 
same is true for $\Tgrl(\leq\bbm)$, that is 
$K_0^{\oplus}(\Tgrl(\leq\bbm))\cong\mathcal B_{m+1}$ as 
$B_{m+1}$-modules. Note now that the limit therefore 
categorifies the corresponding $\infty$ version of 
the Burau representation of the braid 
group $B_{\infty}$ with $\infty$-many strands.
\end{rem}

\begin{rem}\label{rem-grothendieck1}
Fix $\mathfrak{sl}_{m+1}$ and denote 
by $\mathfrak{p}_i$ the parabolic for 
$S_i\times S_{m-i}$. Moreover, we denote 
by $\mathcal O^{\mathfrak{p}}_{\mathrm{max}}=\bigoplus_{i=0}^m\mathcal O^{\mathfrak{p}_i}_{0}$ 
the direct sum. We note that, as Bernstein-Frenkel-Khovanov 
conjectured in~\cite[Conjectures~1-4]{bfk} and Stroppel proved 
in~\cite{st3}, parabolic category 
$\mathcal O^{\mathfrak{p}}_{\mathrm{max}}$ 
for $\mathfrak{sl}_{m+1}$ (compare to Remark~\ref{rem-cato}) 
can be used to \textit{categorify} the $m+1$-strand 
Temperley-Lieb algebra $TL^v_{m+1}$. The split 
Grothendieck group $K_0^{\oplus}$ of the category 
of projective endofunctors on 
$\mathcal O^{\mathfrak{p}}_{\mathrm{max}}$ gives\footnote{The parameter $q$ in the notation of Bernstein-Frenkel-Khovanov comes from the 
grading.} $TL^v_{m+1}$. 
The Grothendieck group behaves additive, i.e.
\[
TL^v_{m+1}\cong K_0^{\oplus}(\Endo(\mathcal O^{\mathfrak{p}}_{\mathrm{max}}))\cong \bigoplus_{i=0}^mK_0^{\oplus}(\Endo(\mathcal O^{\mathfrak{p}}_{i})).
\]
Thus, $\Endo(\mathcal O^{\mathfrak{p}}_{1})$ 
gives a summand $\overline{TL}^v_{m+1}$ of $TL^v_{m+1}$ 
that corresponds to the ``next to highest weight'' 
summand of $V^{\otimes m}$ (with notation as in 
Remark~\ref{rem-webalg}), see~\cite[Section~6]{chkh}.

Thus, as explained in Remark~\ref{rem-cato}, 
by Proposition~\ref{prop-endograd2}, the same 
is true for $\Endo(\Tgrl(\leq\bbm))$. Moreover, 
it is easy to see that the embedding of categories
\[
\Endo(\Tgrl(\leq\bbm))\hookrightarrow\Endo(\Tgrl(\leq\bbmu))
\]
gives rise to an embedding
\[
\overline{TL}^v_{m+1}\cong K_0^{\oplus}(\Endo(\Tgrl(\leq\bbm)))\hookrightarrow K_0^{\oplus}(\Endo(\Tgrl(\leq\bbmu)))\cong \overline{TL}^v_{m+2}.
\]
Thus, by Proposition~\ref{prop-isoofalg2}, $\Endo(\Tgrl)$ 
governs all of them at once.
\end{rem}
\section{Diagrams for the graded tilting category \texorpdfstring{$\Tgr$}{Tgr}}\label{sec-dia}
In this section we give the diagrammatic presentation of 
the (graded) tilting category $\Tgrl$ and the category of 
its projective endofunctors $\Endo(\Tgrl)$.

\subsection{The basics: diagrammatic categories, additive closures and Karoubi envelopes}\label{sub-basiccats}
We consider \textit{diagrammatic} categories 
in the following. We encourage the reader 
to take a look in the rather extensive literature 
(in the authors opinion, a good start is~\cite[Section 2]{lau1}).
Although these categories 
have a natural $2$-categorical structure, 
we phrase everything in terms of $1$-categories.

Furthermore, we note that we \textit{only} use 
$\bC$-linear categories. Recall that a category 
$\mathcal C$ is called $\bC$\textit{-linear} if each 
hom-space $\Hom_{\mathcal C}(O,O^{\prime})$ for 
$O,O^{\prime}\in\Ob(\mathcal C)$ has the 
structure of a $\bC$-vector space and the composition of morphisms is $\bC$-bilinear 
(this includes the existence of 
a zero morphism $0$). 
We use $\bC$\textit{-linear functors} for such categories: 
functors which induce $\bC$-linear maps on each hom-space. 
We, as in Subsection~\ref{sub-grad}, \textit{only} consider ($\bZ$-)graded categories.

\begin{defnn}(\textbf{Additive closure of $\mathcal C$})\label{defn-addclosure} 
Given a $\bC$-linear category $\mathcal C$, we define its \textit{additive closure}, 
denoted by $\Mat(\mathcal C)$, to be the $\bC$-linear category consisting of the following.
\begin{itemize}
\item The objects are finite (possibly empty), 
formal direct sums $\oplus_{i=1}^NO_i$ with $O_i\in\Ob(\mathcal C)$ 
(note that 
the empty object is the zero object).
\item Given two objects $O=\oplus_{k=1}^NO_k,\,O^{\prime}=\oplus_{k=1}^{N^{\prime}}O^{\prime}_k$, then a morphism $F\in\Hom_{\Mat(\mathcal C)}(O,O^{\prime})$ is a $N\times N^{\prime}$ matrix $F=(f_{i^{\prime}i})$ consisting of morphisms $f_{i^{\prime}i}\in\Hom_{\mathcal C}(O_i,O^{\prime}_{i^{\prime}})$.
\item One can add the matrices component-wise and 
scalar multiplication with elements from $\bC$ is also component-wise.
\item Composition of morphisms is multiplication 
of matrices.\maketriqed
\end{itemize}
\end{defnn}

It is easy to check that this definition gives a category.

\begin{defnn}(\textbf{Karoubi envelope of $\mathcal C$})\label{defn-karoubi}
Let $\mathcal C$ be a category and 
$A\in\Ob(\mathcal C)$ be an object of 
$\mathcal C$. Let $e,e^{\prime}\colon O\to O$ 
denote idempotents in $\End_{\mathcal C}(O)$. 
The \textit{Karoubi envelope of} $\mathcal C$, 
denoted by $\Kar(\mathcal C)$, is the following category.
\begin{itemize}
\item Objects are ordered pairs $(O,e)$ 
consisting of an object $O$ and an idempotent $e\in\End_{\mathcal C}(O)$.
\item Morphisms 
$f\colon(O,e)\to(O^{\prime},e^{\prime})$ in $\Kar(\mathcal C)$ 
are all morphisms $f\colon O\to O^{\prime}$ of $\mathcal C$ such that the equations $f=f\circ e=e^{\prime}\circ f$ hold.
\item Compositions are induced by compositions in $\mathcal C$.\maketriqed
\end{itemize}
\end{defnn}

Again, it is immediate that this 
is indeed a category.

The 
identity on $(O,e)$ is $e$ itself. 
Moreover, there is a faithful functor 
$\mathrm{im}\colon \mathcal C\to\Kar(\mathcal C)$,
called \textit{the image}, that sends $A$ to 
$(O,\mathrm{id})$ and $f\colon O\to O^{\prime}$ 
to $f\colon(O,\mathrm{id})\to(O^{\prime},\mathrm{id})$.
Categories $\mathcal C$ such that $C\cong\Kar(\mathcal C)$ 
are called \textit{idempotent complete}.

We point out that $\Mat(\mathcal{C})$ 
is ``combinatorial the same'' as $\mathcal{C}$ -- in contrast to 
$\Kar(\mathcal{C})$ which is usually very hard to describe 
combinatorially/diagrammatically.

\subsection{The dihedral cathedral \texorpdfstring{$\D$}{Dinfty}}\label{sub-dcath}
For the reader familiar with~\cite{el1}: be careful that Elias' 
``root of unity case'' \textit{does not} 
correspond to our root of unity case from Section~\ref{sec-tilting}. 
Moreover, we slightly rescale his 
``barbell forcing relation'', see \textbf{BF2}~\eqref{eq-barbforc}.

Following~\cite{el1}, we encode the two different generators of 
$W_l=\langle s,t\rangle$ using two colors. 
Our convention, that is different from the one used 
by Elias, is that $s$ is displayed 
in \textit{red} and $t$ in \textit{green}. 
The ``third color'' is for $1\in W_l$ which is displayed as 
\textit{white colored or empty}.

\begin{defn}(\textbf{Soergel graph})\label{defn-soergeldia}
A \textit{pre}-Soergel graph 
$\tilde G$ is a colored, planar graph
embedded in $[0,1]\times[0,1]$ 
whose bottom (or source) boundary is in 
$[0,1]\times\{0\}$ and whose top (or target) 
boundary is in $[0,1]\times\{1\}$. The only 
vertices are univalent (called \textit{dots}) 
or trivalent. The edges of a Soergel graph 
are colored by red (or $r$) and green (or $g$) 
such that at each trivalent vertex all adjacent 
edges have the same color. We display dots 
and trivalent vertices locally as
\[
\text{Dots:}\quad
\xy
(0,0)*{\includegraphics[scale=0.8]{res/figs/diacath/iso3}};
(0,-8.5)*{l(\cdot)=1};
\endxy\hspace*{0.3cm}\text{or}\hspace*{0.3cm}
\xy
(0,0)*{\includegraphics[scale=0.8]{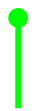}};
(0,-8.5)*{l(\cdot)=1};
\endxy\hspace*{1cm}\text{Trivalent vertices:}\quad
\xy
(0,0)*{\includegraphics[scale=0.8]{res/figs/diacath/iso4}};
(0,-8.5)*{l(\cdot)=-1};
\endxy\hspace*{0.3cm}\text{or}\hspace*{0.3cm}
\xy
(0,0)*{\includegraphics[scale=0.85]{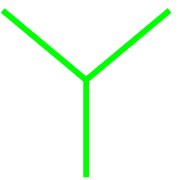}};
(0,-8.5)*{l(\cdot)=-1};
\endxy
\]
where we tend (as above) 
not to display the $r$ and $g$. 
Moreover, the illustration above 
is locally and Soergel graphs 
also include for example horizontal 
reflections of the ones above. 
In addition, given a pre-Soergel 
graph $\tilde G$, we define its 
\textit{degree $l(\tilde G)$} to be the sum of the local degrees as above.

A \textit{Soergel graph $G$} is an 
equivalence class of pre-Soergel graphs 
modulo boundary preserving isotopies of 
colored, planar graphs (which are 
all locally of the forms as in~\eqref{eq-iso}). Note that the 
degree is still well-defined for a 
Soergel graph and we denote it by $l(G)$. 
The empty diagram 
$\bbz$ is also a Soergel graph and of degree $l(\bbz)=0$.

Each Soergel graph gives rise 
to a bottom sequence $b(G)=b_i\dots b_2b_1$ and 
a top sequence $t(G)=t_{i^{\prime}}\dots t_2t_1$ 
of colors $b_k,t_k\in\{r,g\}$ by reading the colors 
at the bottom or top boundary from \textit{right to left} 
respectively. We also allow the empty sequence, if the 
boundary is empty (at bottom or top). We call a Soergel 
graph (or a local piece of it) \textit{floating}, if 
both boundaries are empty. We call floating Soergel 
graphs consisting of only one edge \textit{barbells}.
\end{defn}

Our reading conventions are thus 
from \textit{right to left} (we think of the pictures 
as applying functors to a module) and \textit{bottom to top}. 
Moreover, we think of Soergel graphs as embedded into a 
rectangle (although we never illustrate the rectangle), 
and thus, it make sense to speak about \textit{faces} of 
Soergel graphs. Each Soergel graph has a unique rightmost 
face $F_r$ and a unique leftmost face $F_l$.
\begin{ex}\label{ex-soergelgraph}
An example of two representatives of a Soergel graph $G$ are
\[
\xy
(0,0)*{\includegraphics[scale=0.85]{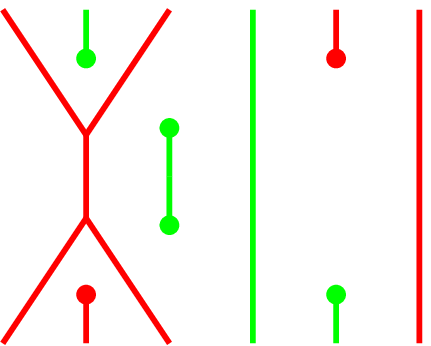}};
(-18,-16)*{\scriptstyle r};
(-10.75,-16)*{\scriptstyle r};
(-3.5,-16)*{\scriptstyle r};
(3.5,-16)*{\scriptstyle g};
(10.75,-16)*{\scriptstyle g};
(18,-16)*{\scriptstyle r};
(-18,16)*{\scriptstyle r};
(-10.75,16)*{\scriptstyle g};
(-3.5,16)*{\scriptstyle r};
(3.5,16)*{\scriptstyle g};
(10.75,16)*{\scriptstyle r};
(18,16)*{\scriptstyle r};
(-16,0)*{\scriptstyle F_l};
(20,0)*{\scriptstyle F_r};
\endxy=
\xy
(0,0)*{\includegraphics[scale=0.85]{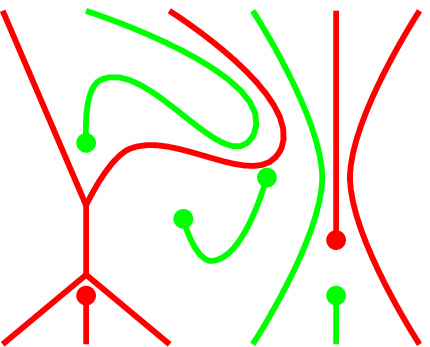}};
(-18,-16)*{\scriptstyle r};
(-10.75,-16)*{\scriptstyle r};
(-3.5,-16)*{\scriptstyle r};
(3.5,-16)*{\scriptstyle g};
(10.75,-16)*{\scriptstyle g};
(18,-16)*{\scriptstyle r};
(-18,16)*{\scriptstyle r};
(-10.75,16)*{\scriptstyle g};
(-3.5,16)*{\scriptstyle r};
(3.5,16)*{\scriptstyle g};
(10.75,16)*{\scriptstyle r};
(18,16)*{\scriptstyle r};
(-16,0)*{\scriptstyle F_l};
(16,0)*{\scriptstyle F_r};
\endxy
\]
Here $l(G)=4$, $b(G)=rrrggr$ and $t(G)=rgrgrr$. 
The graph above has one (green) barbell. We have 
marked the unique right- and leftmost faces with $F_r$ and $F_l$ respectively.
\end{ex}

Note that, since $G$ completely 
determines $b(G)$ and $t(G)$, we do not display these anymore.

\begin{defn}(\textbf{Elias' dihedral cathedral $\D$})\label{defn-diacath} We 
consider the $\bC$-linear, graded, monoidal category 
called the \textit{free} dihedral cathedral, 
denote by $\D_{f}$, consisting of the following.
\begin{itemize}
\item Objects $\Ob(\D_f)$ are finite 
sequences $x=x_i\dots x_2x_1$ with 
$x_i\in\{r,g\}$. Moreover, the empty sequence $\emptyset$ is also an object.
\item The space of morphisms 
$\Hom_{\D_f}(x,x^{\prime})$ for 
$x,x^{\prime}\in\Ob(\D_f)$ is the 
$\bC$-linear span of all Soergel graphs $G$ with $b(G)=x$ and $t(G)=x^{\prime}$.
\item Composition (vertical) of morphisms 
$G^{\prime}\circv G$ is defined by 
glueing $G^{\prime}$ on top of $G$.
\item The monoidal product (horizontal) $\circh$ is, 
for $x=x_i\dots x_2x_1$ and 
$x^{\prime}=x^{\prime}_{i^{\prime}}\dots x^{\prime}_2x^{\prime}_1$, 
given by concatenation $x^{\prime}\circh x=x^{\prime}x$, and 
for $G^{\prime}\circh G$ via placing 
$G^{\prime}$ to the left of $G$.
\item The spaces $\Hom_{\D_f}(x,x^{\prime})$ 
are graded $\bC$-vector spaces where the 
degree of a Soergel graph $G$ is given by $l(G)$. 
\end{itemize}
The category $\D$ is 
the quotient category of $\D_f$ obtained 
by taking the quotient by the following local 
relations and their 
\textit{color-inverted} (red\raisebox{-0.05cm}{$\leftrightarrows$}green) counterparts.

The \textit{two Frobenius relations}:
\begin{equation}\label{eq-frob}
\text{\textbf{Frob1}}:\quad
\xy
(0,0)*{\includegraphics[scale=0.85]{res/figs/diacath/frob2}};
\endxy=
\xy
(0,0)*{\includegraphics[scale=0.85]{res/figs/diacath/frob1}};
\endxy,\quad\quad\text{\textbf{Frob2}}:\quad
\xy
(0,0)*{\includegraphics[scale=0.85]{res/figs/diacath/frob3}};
\endxy\;=\;
\xy
(0,0)*{\includegraphics[scale=0.85]{res/figs/diacath/frob5}};
\endxy
\;=\;
\xy
(0,0)*{\includegraphics[scale=0.85]{res/figs/diacath/frob4}};
\endxy
\end{equation}
The \textit{needle relation}:
\begin{equation}\label{eq-needle}
\text{\textbf{Needle}}:\quad
\xy
(0,0)*{\includegraphics[scale=0.85]{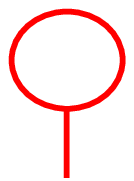}};
\endxy=0
\end{equation}
The \textit{barbell forcing relations}:
\begin{equation}\label{eq-barbforc}
\text{\textbf{BF1}}:\quad
\xy
(0,0)*{\includegraphics[scale=0.85]{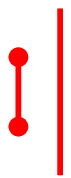}};
\endxy=2\cdot\xy
(0,0)*{\includegraphics[scale=0.85]{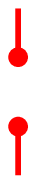}};
\endxy-\xy
(0,0)*{\includegraphics[scale=0.85]{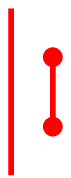}};
\endxy,\quad\quad\text{\textbf{BF2}}:\quad
\xy
(0,0)*{\includegraphics[scale=0.85]{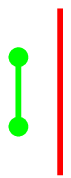}};
\endxy=\xy
(0,0)*{\includegraphics[scale=0.85]{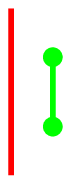}};
\endxy+2\cdot \xy
(0,0)*{\includegraphics[scale=0.85]{res/figs/diacath/forcing3}};
\endxy-2\cdot \xy
(0,0)*{\includegraphics[scale=0.85]{res/figs/diacath/forcing2a}};
\endxy
\end{equation}
Note that the isotopy invariance 
can be locally displayed via the \textit{isotopy relations}:
\begin{equation}\label{eq-iso}
\text{\textbf{Iso1}}:\quad
\xy
(0,0)*{\includegraphics[scale=0.85]{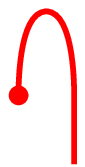}};
\endxy\;=\;
\xy
(0,0)*{\includegraphics[scale=0.85]{res/figs/diacath/iso3}};
\endxy
\;=\;
\xy
(0,0)*{\includegraphics[scale=0.85]{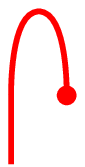}};
\endxy,\quad\quad\text{\textbf{Iso2}}:\quad
\xy
(0,0)*{\includegraphics[scale=0.85]{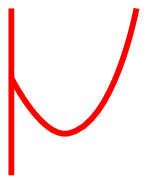}};
\endxy\;=\;
\xy
(0,0)*{\includegraphics[scale=0.85]{res/figs/diacath/iso4}};
\endxy
\;=\;
\xy
(0,0)*{\includegraphics[scale=0.85]{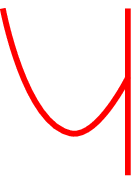}};
\endxy
\end{equation}
together with a horizontal reflection of the two relations.

We also note that all relations 
are homogeneous with respect to 
$l(G)$ and thus, $\D$ inherits 
the grading from $\D_f$. The monoidal product 
$\circh$ also carries over to $\D$ without difficulties.

Moreover, we use the following 
convention: we call elements of 
$\Hom_{\D_f}(x,x^{\prime})$ Soergel 
graphs and, on the other hand, elements 
of $\Hom_{\D}(x,x^{\prime})$ \textit{Soergel diagrams} 
(for all possible $x,x^{\prime}\in\Ob(\D_f)=\Ob(\D)$). 
We call a face of a Soergel diagram \textit{empty} 
if it has no internal floating Soergel diagrams (as for example barbells).
\end{defn}

\begin{ex}\label{ex-soergelgraph1}
An example of two representatives of Soergel diagrams are
\[
\xy
(0,0)*{\includegraphics[scale=0.85]{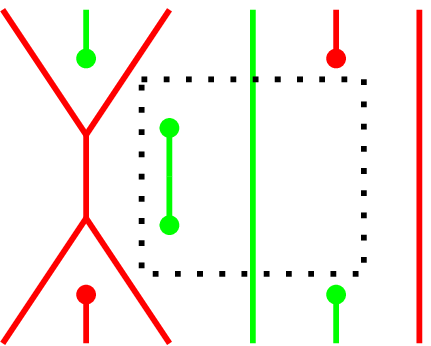}};
\endxy=
2\cdot\xy
(0,0)*{\includegraphics[scale=0.85]{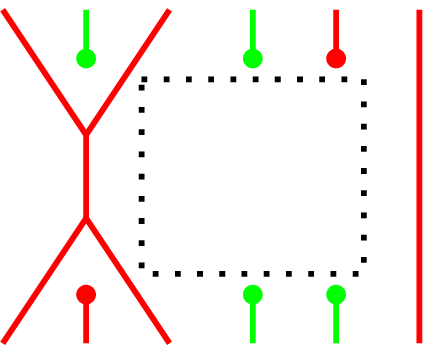}};
\endxy
-\xy
(0,0)*{\includegraphics[scale=0.85]{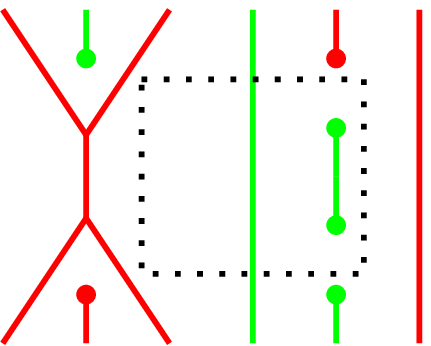}};
\endxy
\]
Here we used the left (green) barbell forcing relation~\eqref{eq-barbforc} on the marked piece.
\end{ex}

\begin{lem}\label{lem-cycle}
A Soergel diagram $G$ with an 
internal cycle which bounds an empty face is zero.\makeqed
\end{lem}

\begin{proof}
Use repeatedly the first Frobenius 
relations~\eqref{eq-frob} as illustrated 
below (note that each step reduces the number 
of adjacent edges of the face and we have 
indicated where to use the first Frobenius 
relation \textbf{Frob1}~\eqref{eq-frob} in the last step)
\[
\cdots\;
\overset{\text{\textbf{Frob1}}}{=}
\;\xy
(0,0)*{\includegraphics[scale=0.85]{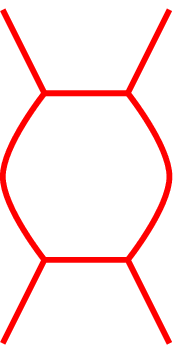}};
\endxy\;
\overset{\text{\textbf{Frob1}}}{=}
\;
\xy
(0,0)*{\includegraphics[scale=0.85]{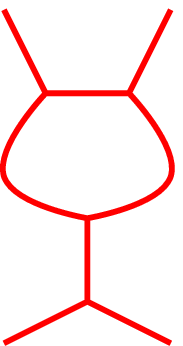}};
\endxy\;
\overset{\text{\textbf{Frob1}}}{=}
\;
\xy
(0,0)*{\includegraphics[scale=0.85]{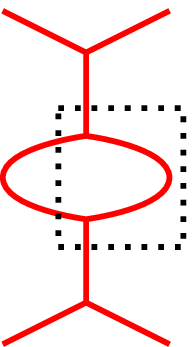}};
\endxy\;
\overset{\text{\textbf{Frob1}}}{=}
\;
\xy
(0,0)*{\includegraphics[scale=0.85]{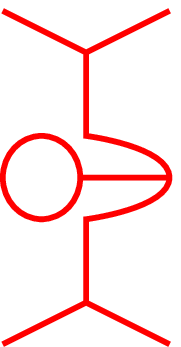}};
\endxy\;
\overset{\text{\textbf{Needle}}}{=}
\;0
\]
until we can use the needle 
relation~\eqref{eq-needle} to see that we get zero.
\end{proof}

A \textit{tree} is a connected Soergel 
graph $G$ with no cycles. We 
call a Soergel diagram \textit{reduced} if 
all of its connected components are trees or 
barbells and all barbells are in the rightmost face.

The following is~\cite[Proposition~5.19]{el1} converted to our conventions.
\begin{prop}\label{prop-soergeldiagrams}
Each Soergel diagram can be written 
as a $\bC$-linear sum of reduced 
Soergel diagrams. Thus, reduced Soergel 
diagrams form spanning sets of the hom-spaces of $\D$.\makeqed
\end{prop}

\begin{proof}
Note that it follows from a repeated 
application of the Frobenius 
relations~\eqref{eq-frob} that all floating 
trees are barbells and, by combining 
Lemma~\ref{lem-cycle} with \textbf{BF1} and 
\textbf{BF2}~\eqref{eq-barbforc}, all other 
floating Soergel diagrams are $\bC$-linear combinations of barbells.

Thus, we can restrict to Soergel diagrams whose only floating components are barbells.

Now, given a Soergel diagram $G$, we 
can first use \textbf{BF1} and \textbf{BF2}~\eqref{eq-barbforc} 
to express $G$ as a direct sum of Soergel 
diagrams $G_i$ (with $i=1,\dots,k$) with only 
barbells in the rightmost face. By 
Lemma~\ref{lem-cycle} all Soergel diagrams 
$G_i$ with an internal cycle are zero now.
\end{proof}

We point out the following corollary. 
This can be compared to~\cite[Corollary~5.20]{el1} and 
the corresponding statement, in our case for 
$\Tgr$, will later be completely different, 
see Lemma~\ref{lem-noaction}.
For this purpose, denote by $\bC[b_r,b_g]$ 
the graded $\bC$-algebras 
of polynomials in $b_r,b_g$, with the generators $b_r,b_g$ being of degree $2$. 

\begin{cor}\label{cor-endoring}
We have $\bC[b_r,b_g]\cong \Endbf_{\D}(\emptyset)$ as 
graded $\bC$-algebras.\makeqed
\end{cor}

\begin{proof}
By Proposition~\ref{prop-soergeldiagrams} 
the map that sends $b_r\mapsto$``red barbell'' 
and $b_g\mapsto$``green barbell'' is an isomorphism 
of graded $\bC$-algebras. To see this note that a bunch of 
barbells do not satisfy any extra relations 
(the barbell forcing relations~\eqref{eq-barbforc} do not give anything new).
\end{proof}

\begin{lem}\label{lem-decom1}
Soergel diagrams satisfy
\[
\xy
(0,0)*{\includegraphics[scale=0.85]{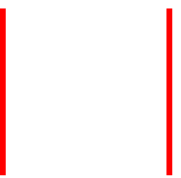}};
\endxy=\frac{1}{2}\cdot\xy
(0,0)*{\includegraphics[scale=0.85]{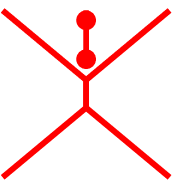}};
\endxy+\frac{1}{2}\cdot\xy
(0,0)*{\includegraphics[scale=0.85]{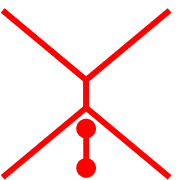}};
\endxy
\quad\text{and}\quad
\xy
(0,0)*{\includegraphics[scale=0.85]{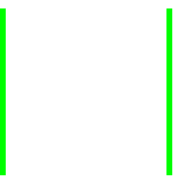}};
\endxy=\frac{1}{2}\cdot\xy
(0,0)*{\includegraphics[scale=0.85]{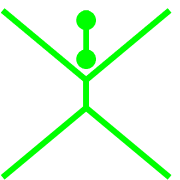}};
\endxy+\frac{1}{2}\cdot\xy
(0,0)*{\includegraphics[scale=0.85]{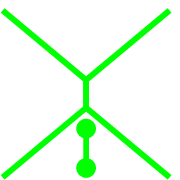}};
\endxy
\hspace{0.6cm}
\raisebox{-0.625cm}{\makeqed}
\hspace{-0.6cm}
\]
\end{lem}

\begin{proof}
Use \textbf{Frob1}~\eqref{eq-frob} on the 
right side followed by \textbf{BF1}~\eqref{eq-barbforc} for the middle edge.
\end{proof}

This implies the useful fact that we can focus 
on alternating sequences (the case where the 
sequence is a reduced word in $\D$) of red 
and green if we go to $\Mat(\D)$ (where we use $\circ$ 
for the matrix multiplication).

\begin{lem}\label{lem-decom2}
We have
\[
rr\cong r\langle-1\rangle\oplus r\langle+1\rangle\quad\text{and}\quad gg\cong g\langle-1\rangle\oplus g\langle+1\rangle.
\]
which are isomorphisms in $\Mat(\D)$.\makeqed
\end{lem}

\begin{proof}
The isomorphism for the red case is induced by
\[
\begin{xy}
  \xymatrix{
      rr \ar[rrrrr]|/-1.5em/{\begin{pmatrix}
 1\cdot\xy
(0,0)*{\includegraphics[scale=0.85]{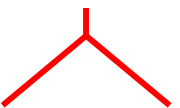}};
\endxy \\
 \frac{1}{2}\cdot\xy
(0,0)*{\includegraphics[scale=0.85]{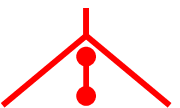}};
\endxy 
\end{pmatrix}}  & & & &  &  r\langle-1\rangle\oplus r\langle+1\rangle \ar[rrrrr]|/1.05em/{\begin{pmatrix}
 \frac{1}{2}\cdot\xy
(0,0)*{\includegraphics[scale=0.85]{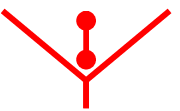}};
\endxy & 1\cdot\xy
(0,0)*{\includegraphics[scale=0.85]{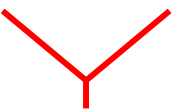}};
\endxy
\end{pmatrix}} & & & & & rr,
  }
\end{xy}
\]
where we 
check (with Lemma~\ref{lem-decom1} for 
``right$\circ$left'' and Lemma~\ref{lem-cycle} 
together with BF1~\eqref{eq-barbforc} for 
``left$\circ$right'') that the two matrices above are inverses.
The green case works analogous.
\end{proof}

\begin{defn}(\textbf{Jones-Wenzl projectors})\label{defn-jwpro}
Denote by $x_{k-tst}$ a 
sequence $x_k\dots x_2x_1$ of length $k$ 
that alternates in the colors red and green 
and starts with $x_1=g$ (thus, ends in $r$ iff $k$ is even).

We define for each $i\geq 0$ the 
$i$\textit{-th green Jones Wenzl} 
projector $\JWg_i$ recursively as the 
element of $\Endbf_{\D}(x_{i-tst})$ obtained 
via the convention that $\JWg_0$ is 
$\bbz\colon \emptyset\to \emptyset$, $\JWg_1$ 
is $\mathrm{id}\colon g\to g$, $\JWg_2$ is 
$\mathrm{id}\colon rg\to rg$ and recursion rule for $i>2$ given by
\[
\xy
(0,0)*{\includegraphics[scale=0.85]{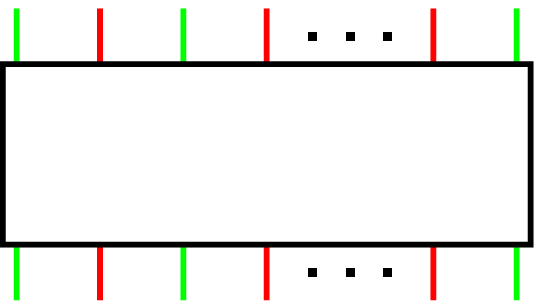}};
(0,0)*{\JWg_{i}};
\endxy=\xy
(0,0)*{\includegraphics[scale=0.85]{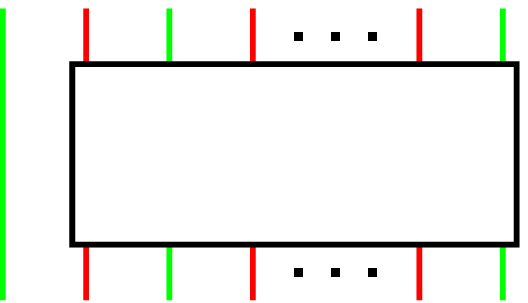}};
(2.5,0)*{\JWg_{i-1}};
\endxy-\frac{i-2}{i-1}\cdot\xy
(0,0)*{\includegraphics[scale=0.85]{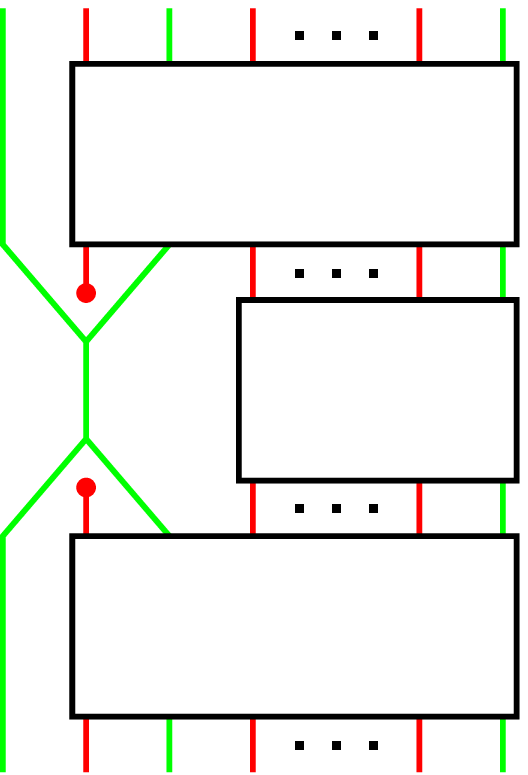}};
(10,0)*{\JWg_{i-3}};
(2.5,20)*{\JWg_{i-1}};
(2.5,-21)*{\JWg_{i-1}};
\endxy
\]
Here we have only illustrated the case for odd $i$.
Similar for even $i$ and red projectors $\JWr_i$.
\end{defn}

\begin{ex}\label{ex-jw}
The first four (green) Jones-Wenzl projectors are
\[
\JWg_0=\bbz,\quad\quad
\JWg_1=\xy
(0,0)*{\includegraphics[scale=0.85]{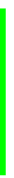}};
\endxy,\quad\quad\JWg_2=\xy
(0,0)*{\includegraphics[scale=0.85]{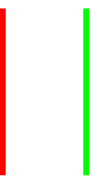}};
\endxy,\quad\quad\JWg_3=\xy
(0,0)*{\includegraphics[scale=0.85]{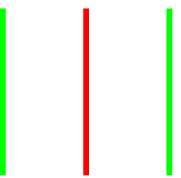}};
\endxy-\frac{1}{2}\xy
(0,0)*{\includegraphics[scale=0.85]{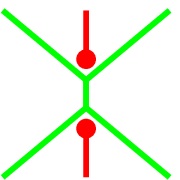}};
\endxy
\]
as follows by an easy calculation.
\end{ex}

For each $i^{\prime}>2$ we call a diagram of the form
\[
\xy
(0,0)*{\includegraphics[scale=0.85]{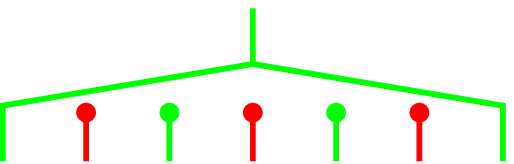}};
\endxy
\]
a \textit{green alternating $i^{\prime}$-pitchfork} 
(here $i^{\prime}=7$). We denote it by $P^g_{i^{\prime}}$ 
and similarly for $P^r_{i^{\prime}}$, which is the 
\textit{$i^{\prime}$-th red alternating pitchfork}. 
We note that $b(P^g_{i^{\prime}})$ as well as 
$b(P^r_{i^{\prime}})$ consist of $i^{\prime}$ 
elements. Moreover, we write $D_i(P^g_{i^{\prime}})$ 
for a bigger diagram with $i\geq i^{\prime}>2$ which somewhere 
at the bottom has a green alternating pitchfork, i.e.
\begin{equation}\label{eq-pitch}
D_i(P^g_{i^{\prime}})=D\circ (\mathrm{id}_x\,P^g_{i^{\prime}}\,\mathrm{id}_{x^{\prime}})
\end{equation}
for some sequences $x,x^{\prime}$ and some diagram $D$. Similarly for red again.

\begin{lem}\label{lem-pitch}
Let $i\geq i^{\prime}>2$. 
Then $D_i(P^g_{i^{\prime}})\circ \JWg_i=0$. Similarly for red.\makeqed
\end{lem}

\begin{proof}
This follows recursively. 
To see the first step $i=i^{\prime}=3$, we 
compose $P^g_{3}$ with the second term of 
$\JWg_3$ from Example~\ref{ex-jw} and obtain 
by using \textbf{BF2}~\eqref{eq-barbforc} 
and Lemma~\ref{lem-cycle} for the first 
step and \textbf{Frob2}~\eqref{eq-frob} for the second:
\begin{equation*}
-\frac{1}{2}\cdot\xy
(0,0)*{\includegraphics[scale=0.85]{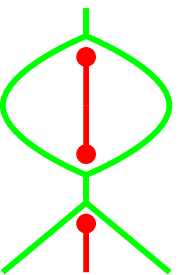}};
\endxy=-\frac{1}{2}\cdot 2\cdot\xy
(0,0)*{\includegraphics[scale=0.85]{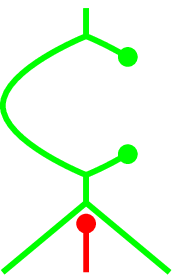}};
\endxy=-\xy
(0,0)*{\includegraphics[scale=0.85]{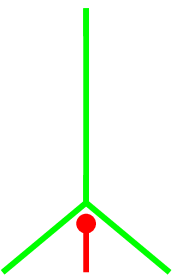}};
\endxy
\end{equation*}
giving minus the composition of $P^g_{3}$ 
with the first term of $\JWg_3$.
Thus, 
$D_3(P^g_{3})\circ \JWg_3=0$.

Now we use induction on $i$ and fix $i^{\prime}=3$. 
If $i>3$ and $D_i(P^g_{3})$ has $x\neq \emptyset$ 
(see~\eqref{eq-pitch}), then the recursion rule 
from Definition~\ref{defn-jwpro} (the $3$-pitchfork 
is on top of the smaller box that corresponds to the 
$i-1$-th Jones-Wenzl projector) shows by induction 
that we get zero. If, on the other hand, $x=\emptyset$, 
then we can 
use a similar argument as above. Hence, 
we always have $D_i(P^g_{3})\circ \JWg_i=0$.

But the case $i^{\prime}=3$ suffices as the following illustrates.
\[
\xy
(0,0)*{\includegraphics[scale=0.85]{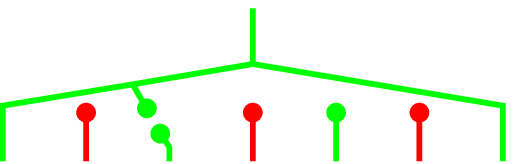}};
\endxy=\frac{1}{2}\cdot\xy
(0,0)*{\includegraphics[scale=0.85]{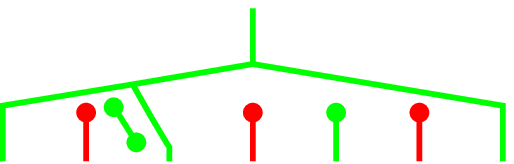}};
\endxy+\frac{1}{2}\cdot\xy
(0,0)*{\includegraphics[scale=0.85]{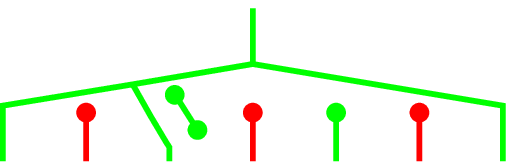}};
\endxy
\]
Here we used \textbf{Frob2}~\eqref{eq-frob} ``backwards'' 
on a bigger alternating pitchfork followed by 
\textbf{BF1}~\eqref{eq-barbforc}. We see now two smaller 
pitchforks and thus, the case $i^{\prime}=3$ 
suffices to verify the lemma.
\end{proof}

Lemma~\ref{lem-pitch} holds, mutatis mutandis, by attaching horizontally 
reflected pitchforks form the bottom as well.
Moreover, it is easy to deduce that the Jones-Wenzl 
projectors are (degree zero) idempotents 
of $x_{i-tst}$ and $x_{i-sts}$, respectively. 
We associate these to the corresponding 
alternating sequences. Furthermore, by Lemma~\ref{lem-decom2}, we can, by going to $\Mat(\D)$, associate to each $x\in\Ob(\D)$ a (shifted) direct sum of Jones-Wenzl 
projectors. We denote this sum that consists of 
these projectors for $x$ by $\JWx$. This motivates the following category.

\begin{defn}\label{defn-matdia}
Denote by $\Mat(\Dk)$ the full subcategory 
of $\Kar(\Mat(\D))$ generated by $\mathrm{im}(\JWx)$ for all $x\in\Ob(\D)$.
\end{defn}
Diagrammatically $\Mat(\Dk)$ is given similarly to 
$\Mat(\D)$, but has some extra relations. For example, 
Lemma~\ref{lem-pitch} says that, whenever we see a
pitchfork at the bottom (or its reflection at the top) 
boundary, then the corresponding Soergel diagram is zero.

\subsection{The quotient \texorpdfstring{$\QD$}{QDinfty}}\label{sub-qd}

We are now able to define the category 
whose combinatorics and diagrammatic 
govern the tilting category $\Tgr$. 
We call it for short \textit{the quotient}.

\begin{defn}(\textbf{The quotient})\label{defn-thequotient}
We denote by $\QD_f$ the $\bC$-linear full 
subcategory of $\D$, called \textit{``free quotient''}, consisting of the following.

\begin{itemize}
\item Objects $\Ob(\QD_f)$ are finite sequences 
$x=x_i\dots x_2x_1$ with $x_k\in\{r,g\}$. Moreover, 
the empty sequence $\emptyset$ is also an object.

\item The hom-space $\Hom_{\QD_f}(x,x^{\prime})$ 
for $x,x^{\prime}\in\Ob(\QD_f)$ is the $\bC$-linear 
span of all Soergel graphs $G$ with $b(G)=x$ and 
$t(G)=x^{\prime}$ whose rightmost face is marked with the \textit{dead-end symbol}.

\item Composition $G^{\prime}\circv G$ is defined 
by glueing $G^{\prime}$ on top of $G$. The dead-end symbol is just a 
marker and behaves as an idempotent under composition 
and can not be moved from the rightmost face.

\item The morphisms in $\Hom_{\QD_f}(x,x^{\prime})$ 
satisfy the relations of the corresponding morphisms 
(\textit{without} the dead-end marker) is 
$\Hom_{\D}(x,x^{\prime})$ given in Definition~\ref{defn-diacath}.

\item The spaces $\Hom_{\QD_f}(x,x^{\prime})$ are 
graded $\bC$-vector spaces with degree given by $l(G)$.
\end{itemize}

The \textit{quotient $\QD$} is the quotient of 
$\QD_f$ by the following extra relations which 
we call \textit{the dead-end relations}:
\begin{equation}\label{eq-dead}
\text{\textbf{DE1}}:\quad
\xy
(0,0)*{\includegraphics[scale=0.85]{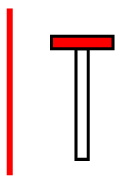}};
\endxy=
0,\quad\quad\text{\textbf{DE2}}:\quad
\xy
(0,0)*{\includegraphics[scale=0.85]{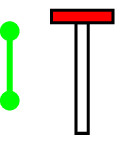}};
\endxy=0
\end{equation}
We stress that we do not allow the 
color inverted counterparts of~\eqref{eq-dead}. 
All relations are homogeneous with 
respect to $l(G)$ and thus, $\QD$ inherits the grading from $\D$. 
We call the morphisms in $\QD$ \textit{marked} Soergel diagrams.
\end{defn}

\begin{ex}\label{ex-soergelgraph2}
An example of a Soergel diagram that 
is zero as a marked Soergel diagrams is
\[
\xy
(0,0)*{\includegraphics[scale=0.85]{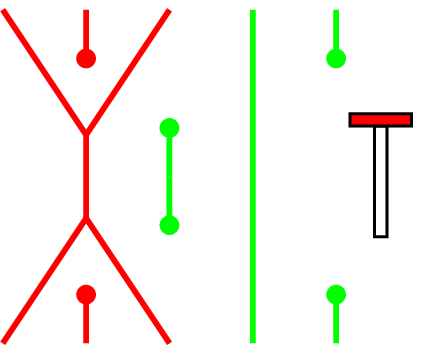}};
\endxy=
2\cdot\xy
(0,0)*{\includegraphics[scale=0.85]{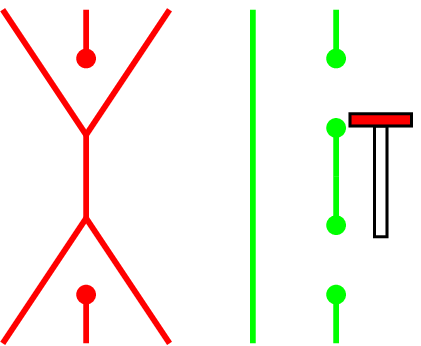}};
\endxy
-\xy
(0,0)*{\includegraphics[scale=0.85]{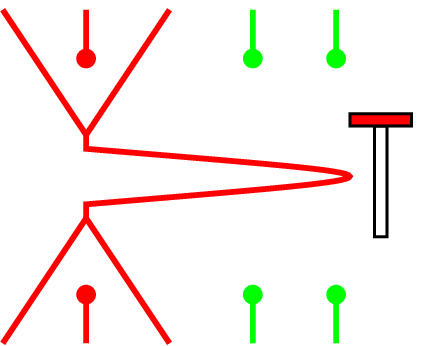}};
\endxy=0
\]
Here the middle diagrams is zero due to \textbf{DE2} and the right-hand due to \textbf{DE1}.
\end{ex}

Let \textit{reduced}, marked Soergel diagrams be the evident analogues of 
reduced Soergel diagrams.

\begin{lem}\label{lem-noaction}
Each marked Soergel diagram can be 
written as a $\bC$-linear sum of marked, 
reduced Soergel diagrams. Thus, marked, 
reduced Soergel diagrams without barbells 
form spanning sets of the hom-spaces of $\QD$. 
The right action of $\Endbf_{\D}(\emptyset)$ by 
placing barbells in the rightmost face is the trivial action.\makeqed
\end{lem}

\begin{proof}
Use repeatedly \textbf{BF1} and 
\textbf{BF2}~\eqref{eq-barbforc} to 
shift all existing barbells to 
the rightmost region. Then use 
the two dead-end relations \textbf{DE1} 
and \textbf{DE2}~\eqref{eq-dead} to see that 
only terms without barbells are not 
killed. The remaining diagrams are 
reduced by Proposition~\ref{prop-soergeldiagrams}.
\end{proof}

Note that, as a consequence of Lemma~\ref{lem-noaction}, 
a marked, reduced Soergel diagram does 
not have any floating components anymore.

\begin{defn}\label{defn-matdia2}
Denote by $\Mat(\QDk)$ the full 
subcategory of $\Kar(\Mat(\QD))$ generated by $\mathrm{im}(\JWx)$ for all $x\in\Ob(\QD)$.
\end{defn}

Thus, by using Lemma~\ref{lem-decom2} and 
\textbf{DE1}~\eqref{eq-dead}, we usually can focus 
on alternating sequences $x$ that start with $g$ in the following.

Given a marked, reduced Soergel diagram $G$, 
we call two edges $E_1,E_2$ \textit{face sharing}, 
if $E_1$ and $E_2$ are adjacent to some common face. 
Marked, reduced Soergel diagrams $G$ (with color 
alternating boundary) with face sharing edges of 
the same color come in four types, where we, by 
Lemma~\ref{lem-pitch}, ignore diagrams with pitchforks. The first two 
types are
\begin{equation}\label{eq-type12}
\text{\textbf{Type 1}}:\quad
\xy
(0,0)*{\includegraphics[scale=0.85]{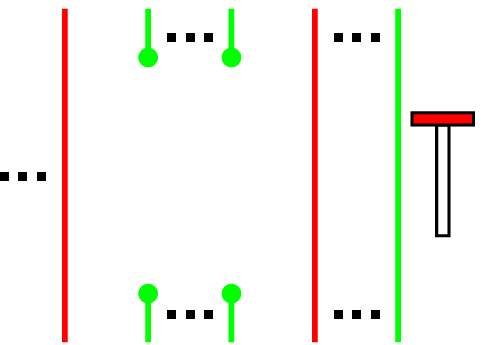}};
(-4.1,8)*{\scriptstyle \underbrace{}};
(-4.1,5.5)*{\scriptstyle k\geq 1};
(-4.1,-8)*{\scriptstyle \overbrace{•}};
(-4.1,-5)*{\scriptstyle k^{\prime}\geq 1};
\endxy,\quad\quad\text{\textbf{Type 2}}:\quad
\xy
(0,0)*{\includegraphics[scale=0.85]{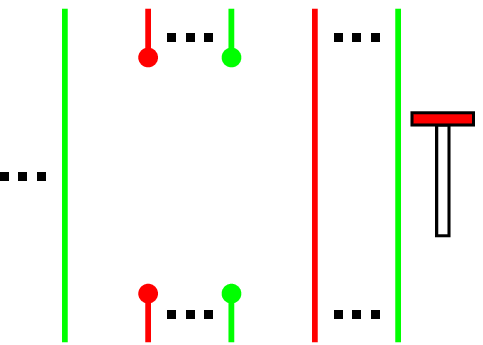}};
(-4.1,8)*{\scriptstyle \underbrace{•}};
(-4.1,5.5)*{\scriptstyle k\geq 2};
(-4.1,-8)*{\scriptstyle \overbrace{•}};
(-4.1,-5)*{\scriptstyle k^{\prime}\geq 2};
\endxy
\end{equation}

where the dots $\cdots$ mean an arbitrary (possible empty) 
diagram in between and we do not require the picture to 
by symmetric. Moreover, $k,k^{\prime}$ are odd in \textbf{Type 1} 
and even in \textbf{Type 2}. We should note that only the first 
edge has to be green, but we allow 
possible color inversion of the rest 
of the picture (that is, after the rightmost dots).
The other two types are
\begin{equation}\label{eq-type34}
\text{\textbf{Type 3}}:\quad
\xy
(0,0)*{\includegraphics[scale=0.85]{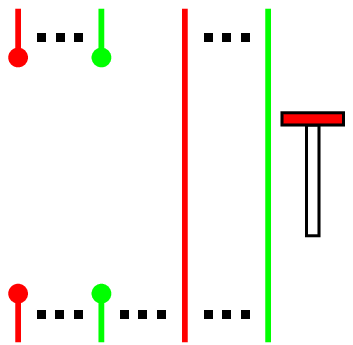}};
(-10,8)*{\scriptstyle \underbrace{•}};
(-10,5.5)*{\scriptstyle k\geq 2};
(-10,-8)*{\scriptstyle \overbrace{•}};
(-10,-5)*{\scriptstyle k^{\prime}\geq 0};
\endxy,\quad\quad\text{\textbf{Type 4}}:\quad
\xy
(0,0)*{\includegraphics[scale=0.85]{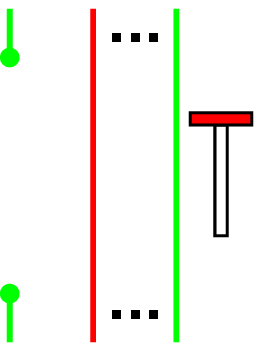}};
\endxy
\end{equation}
The difference between the types 
1,2 and 3,4 is that 1,2 have face sharing 
edges for a non-extremal face, while 
3,4 have face sharing edges for the leftmost 
face. We note that \textbf{Type 3} should 
also include the horizontal reflection (fixing the dead-end symbol) 
of the first diagram in~\eqref{eq-type34}.

\begin{lem}\label{lem-facesharing}
If a marked, reduced Soergel diagram 
$G\in\Hom_{\Mat(\QDk)}(x,x^{\prime})$ has 
two face sharing edges of one of the three types 1,2 or 3, then it is zero.\makeqed
\end{lem}

\begin{proof}
A case-by-case check. We assume that there is at 
least one such face.

\textbf{Type 1}. This follows via
\[
\xy
(0,0)*{\includegraphics[scale=0.85]{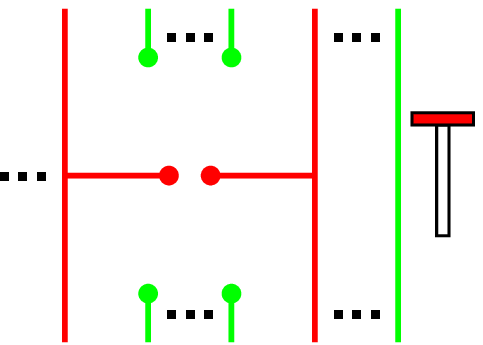}};
\endxy=\frac{1}{2}\cdot \xy
(0,0)*{\includegraphics[scale=0.85]{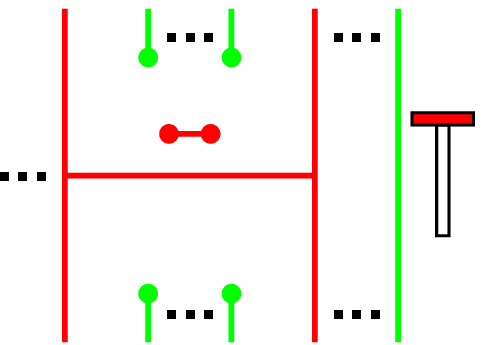}};
(-4,-5)*{\text{\tiny pitchfork}};
\endxy + \frac{1}{2}\cdot \xy
(0,0)*{\includegraphics[scale=0.85]{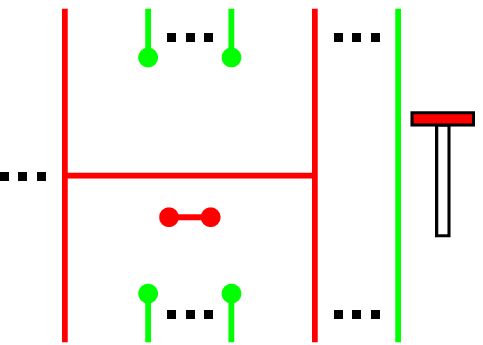}};
(-4,5)*{\text{\tiny pitchfork}};
\endxy
\]
That both diagrams are zero follows from Lemma~\ref{lem-pitch} 
(recalling that we have restricted ourself to alternating sequences of colors), since 
the marked faces above are, by~\eqref{eq-frob}, pitchforks.

\textbf{Type 2}. This follows similarly to \textbf{Type 3} below.

\textbf{Type 3}. Note that the number of bottom edges 
at the end does not matter and we thus, by 
simplicity, assume that there are none of these. 
In addition, we only illustrate the case with 
two dotted edges at the top since all the 
others will be similar. And, without loss 
of generality, we assume that this is the only 
face with adjacent edges of the same color.

We proceed by induction on $j$, where $j>0$ 
is the minimal number of cuts that a line 
drawn from the rightmost face to the face 
in question needs. The case $j=1$ can be deduced via
\[
\xy
(0,0)*{\includegraphics[scale=0.85]{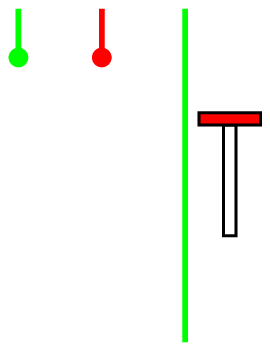}};
\endxy=\xy
(0,0)*{\includegraphics[scale=0.85]{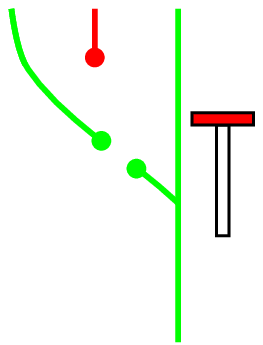}};
\endxy=\frac{1}{2}\cdot \xy
(0,0)*{\includegraphics[scale=0.85]{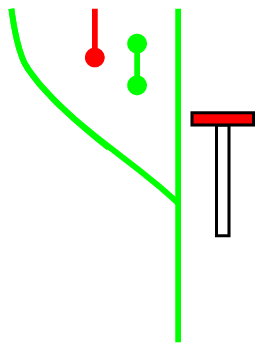}};
\endxy + \frac{1}{2}\cdot \xy
(0,0)*{\includegraphics[scale=0.85]{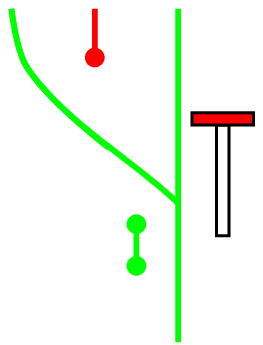}};
\endxy
\]
Note that the rightmost 
diagram is zero by Lemma~\ref{lem-pitch}. Thus, we 
continue as
\[
\frac{1}{2}\cdot\xy
(0,0)*{\includegraphics[scale=0.85]{res/figs/quotient/case3d}};
\endxy=-\frac{1}{2}\cdot\xy
(0,0)*{\includegraphics[scale=0.85]{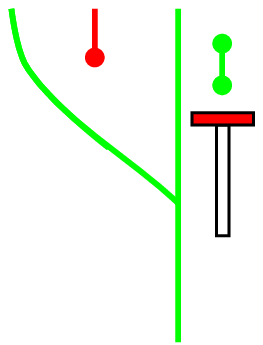}};
\endxy+\xy
(0,0)*{\includegraphics[scale=0.85]{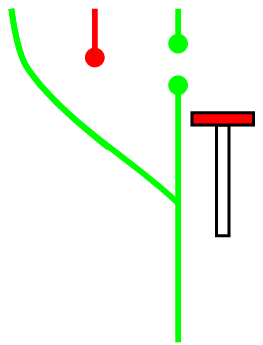}};
\endxy=\xy
(0,0)*{\includegraphics[scale=0.85]{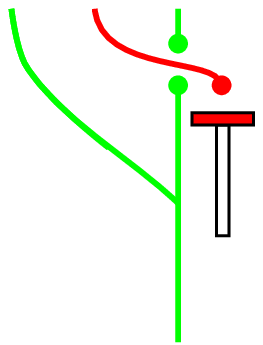}};
\endxy=0
\]
where the last diagram is zero due 
to \textbf{DE1}~\eqref{eq-dead}.

This shows that our starting diagram was 
already zero. The induction $j>1$ works 
now in a similar fashion where we perform 
the same sequence of moves until we are 
in a situation as the last step above. 
Then we can use induction, since we have 
created a face with edges of the same 
color, but ``closer'' to the rightmost face. This finishes the proof.
\end{proof}

\begin{cor}\label{cor-homspaces2}
Let $i,i^{\prime}\in\bN$. Then we have 
the following for $i,i^{\prime}\neq 0$.
\[
\Hom_{\Mat(\QDk)}(x_{i-tst},x_{i^{\prime}-tst})\cong\begin{cases}\bC\bbi\oplus\bC\varepsilon_i,&\text{if }|i-i^{\prime}|=0,\\
\bC u_{i},&\text{if }i^{\prime}-i=1,\\
\bC d_{i},&\text{if }i-i^{\prime}=1,\\
0,&\text{if }|i-i^{\prime}|>1,\end{cases}
\]
with diagrams for $i=1,2,\dots$ 
(and possible color inverted left sides)
\[
\bbi=\xy
(0,0)*{\includegraphics[scale=0.85]{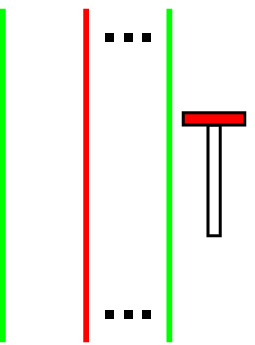}};
(-10,-16)*{\scriptstyle x_i};
(-2,-16)*{\scriptstyle x_{i-1}};
(4.5,-16)*{\scriptstyle x_1};
(-10,15.5)*{\scriptstyle x_i};
(-2,15.5)*{\scriptstyle x_{i-1}};
(4.5,15.5)*{\scriptstyle x_1};
\endxy\;,\;\varepsilon_i=\frac{1}{2^i}\cdot\xy
(0,0)*{\includegraphics[scale=0.85]{res/figs/quotient/case4b}};
(-10,-16)*{\scriptstyle x_i};
(-2,-16)*{\scriptstyle x_{i-1}};
(4.5,-16)*{\scriptstyle x_1};
(-10,15.5)*{\scriptstyle x_i};
(-2,15.5)*{\scriptstyle x_{i-1}};
(4.5,15.5)*{\scriptstyle x_1};
\endxy\;,\; u_{i-1}=\xy
(0,0)*{\includegraphics[scale=0.85]{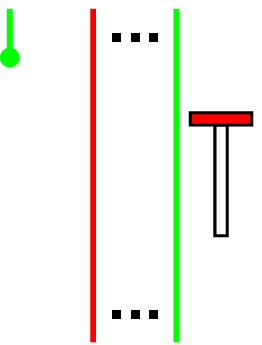}};
(-2,-16)*{\scriptstyle x_{i-1}};
(4.5,-16)*{\scriptstyle x_1};
(-10,15.5)*{\scriptstyle x_i};
(-2,15.5)*{\scriptstyle x_{i-1}};
(4.5,15.5)*{\scriptstyle x_1};
\endxy\;,\; d_i=\frac{1}{2^{i}}\cdot\xy
(0,0)*{\includegraphics[scale=0.85]{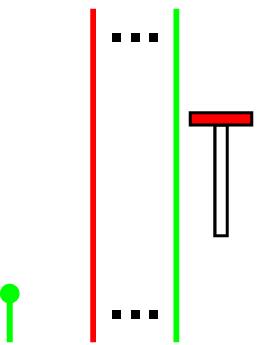}};
(-10,-16)*{\scriptstyle x_i};
(-2,-16)*{\scriptstyle x_{i-1}};
(4.5,-16)*{\scriptstyle x_1};
(-2,15.5)*{\scriptstyle x_{i-1}};
(4.5,15.5)*{\scriptstyle x_1};
\endxy
\]
of degree $l(\bbi)=0$, $l(u_i)=l(d_i)=1$ 
and $l(\varepsilon_i)=2$. In addition, 
$\Hom_{\Mat(\QDk)}(\emptyset,\emptyset)$ 
is one dimensional and spanned by the empty diagram $\bbz$ of degree $l(\bbz)=0$.\makeqed
\end{cor}

\begin{proof}
A case-by-case check using Lemma~\ref{lem-facesharing}. That $\Hom_{\Mat(\QDk)}(\emptyset,\emptyset)\cong\bC\bbz$ follows from \textbf{DE1} and \textbf{DE2}~\eqref{eq-dead} since no floating diagrams can exist.

If $|i-i^{\prime}|>1$, then each Soergel diagram has a pitchfork or is of \textbf{Type 3}~\eqref{eq-type34}. Both are zero by Lemmas~\ref{lem-pitch} and~\ref{lem-facesharing}.
If $|i-i^{\prime}|=1$, then, by Lemma~\ref{lem-facesharing} again, only diagrams of the form $u_i$ or $d_i$ as above can be non-zero and these are clearly non-zero.
If $|i-i^{\prime}|=0$, then, again by Lemma~\ref{lem-facesharing}, only diagrams of the form $\bbi$ or $\varepsilon_i$ as above can be non-zero (which they clearly are).

All these diagrams are linearly independent for degree reasons.
\end{proof}

\begin{rem}\label{rem-nonzero}
Let $\mathcal C$ be a category 
given by generators and relations. Usually 
it is very hard to see ``by hand'' that the 
calculus provided by $\mathcal C$ does not collapse. A way to 
show that this is not the case is to represent $\mathcal C$ 
on a category ``under control'', e.g. on module categories.

We formally have to do the same to see that the four maps from above are indeed non-zero. 
What we use in our case is Lemma~\ref{lem-welldef} which identifies 
the four maps with the evident $\Uq$-intertwiners (we point out that 
we do not need Corollary~\ref{cor-homspaces2} to prove Lemma~\ref{lem-welldef}).
\end{rem}

\begin{prop}\label{prop-quiverrel}
We have
\begin{gather*}
u_{i}\circ u_{i-1}=0=d_i\circ d_{i+1},\;i=1,2,\dots,\quad\quad d_{i+1}\circ u_i=\varepsilon_i=u_{i-1}\circ d_{i},\;i=1,2,\dots
\\
d_1\circ u_0=0, 
\end{gather*}
where we use the notation from Corollary~\ref{cor-homspaces2}.\makeqed
\end{prop}

\begin{proof}
The relation $d_1\circ u_0$ is \textbf{DE2}~\eqref{eq-dead}. 
For $u_{i}\circ u_{i-1}=0$ the corresponding diagrams are
\[
\xy
(0,0)*{\includegraphics[scale=0.85]{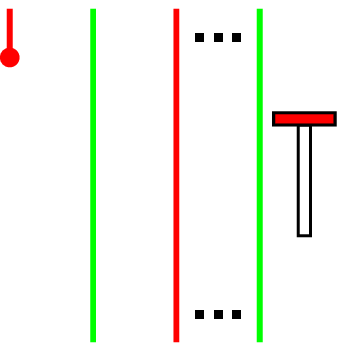}};
\endxy
\circ
\xy
(0,0)*{\includegraphics[scale=0.85]{res/figs/quotient/case4c}};
\endxy=\xy
(0,0)*{\includegraphics[scale=0.85]{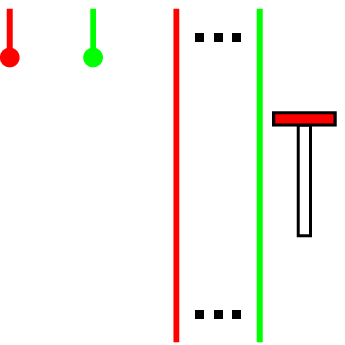}};
\endxy=0
\]
The relation $d_i\circ d_{i+1}=0$ follows similar and 
$u_{i-1}\circ d_{i}=\varepsilon_i$ can be verified directly as
\[
\xy
(0,0)*{\includegraphics[scale=0.85]{res/figs/quotient/case4c}};
\endxy
\circ\frac{1}{2^{i}}\cdot
\xy
(0,0)*{\includegraphics[scale=0.85]{res/figs/quotient/case4d}};
\endxy=\frac{1}{2^{i}}\cdot\xy
(0,0)*{\includegraphics[scale=0.85]{res/figs/quotient/case4b}};
\endxy
\]
while $d_{i+1}\circ u_i=\varepsilon_i$ follows via
\[
\frac{1}{2^{i+1}}\cdot
\xy
(0,0)*{\includegraphics[scale=0.85]{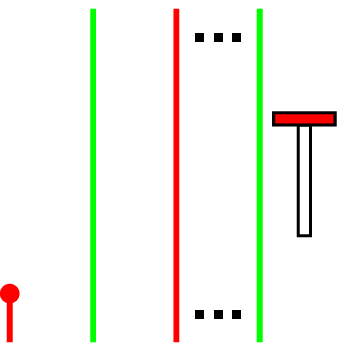}};
\endxy
\circ
\xy
(0,0)*{\includegraphics[scale=0.85]{res/figs/quotient/case4e}};
\endxy=\frac{1}{2^{i+1}}\cdot\xy
(0,0)*{\includegraphics[scale=0.85]{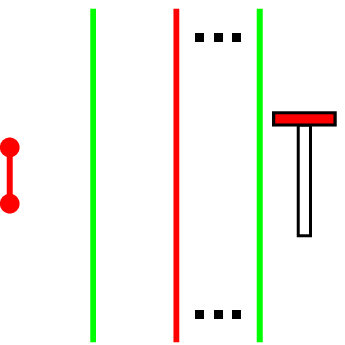}};
\endxy=\frac{1}{2^{i}}\cdot\xy
(0,0)*{\includegraphics[scale=0.85]{res/figs/quotient/case4b}};
\endxy
\]
where we use \textbf{BF2}~\eqref{eq-barbforc} 
together with Lemma~\ref{lem-facesharing} 
to see that the only surviving term is 
the ``broken'' one from \textbf{BF2}~\eqref{eq-barbforc}. This finishes the proof.
\end{proof}

\subsection{\texorpdfstring{$\QD$}{QDinfty} and the tilting category \texorpdfstring{$\Tgr$}{Tgr}}\label{sub-diatil}

Given a sequence $x\in\Ob(\QD)$ of the 
form $x_i\dots x_2x_1$ with $x_k\in\{r,g\}$ 
we denote by $\tilde{\ell}(x)$ the 
\textit{number of color changes} 
(including the first from empty to green), 
e.g. $\tilde{\ell}(g)=\tilde{\ell}(gg)=1$, 
$\tilde{\ell}(ggrgrgrr)=6$ and $\ell(x)=\tilde{\ell}(x)$ iff $x$ is alternating.

\begin{defn}(\textbf{Diagrammatic presentation})\label{defn-functor}
Define a $\bC$-linear functor
\[
\mathcal{D}_{\infty}\colon \Mat(\QDk)\to\Tgrl
\]
of graded categories by the following convention.

On objects (we treat degree 
shifted objects in the evident way):
\begin{itemize}
\item Send the empty sequence $\emptyset$ to $\Tn{\bbz}$. If $x\in\Ob(\QD)$ is of the 
form $x_i\dots x_2x_1$ with $x_k\in\{r,g\}$, then define
\[
\mathcal{D}_{\infty}(x)=p_{\Tn{\lambda_{\tilde{\ell}(x)}}}\circ\Theta_{x_i}\circ\dots\circ\Theta_{x_2}\circ\Theta_{x_1}\Tn{\bbz},
\]
where we use the convention 
$r=s, g=t$ and $p_{\Tn{\lambda_{\tilde{\ell}(x)}}}$ projects 
to the $\Tn{\lambda_{\tilde{\ell}(x)}}$ part.

\item Send a general $x\in\Ob(\Mat(\QDk))$ to the direct sum of the images of its components.
\end{itemize}

All marked Soergel diagrams are 
compositions of diagrams of the form 
$D^{\prime}GD$ with identity diagrams 
$D$ and $D^{\prime}$ and a generator 
$G$ in between. On morphisms:
\begin{itemize}
\item We describe the image of marked Soergel 
diagrams \textit{inductively from right to left}. 
That is, we describe what the functor does if 
one has already an identity diagram 
$D\colon x\to x$ and $\mathcal{D}_{\infty}(D)\colon M_k\to M_k$ 
(with $\mathcal{D}_{\infty}(x)\cong M_k$) and one adds a generator to the \textit{left}.

\item By construction of the image on objects 
and Corollaries~\ref{cor-theta} and~\ref{cor-funcgrad3}, 
$M_k$ will be just one $\Tn{\lambda_k}$ with 
some multiplicity and grading shifts and some 
$k\in\{0,\dots,i\}$ giving two cases, namely 
$k=\tilde{l}(x)$ even or odd. That is, we have
\[
M_k=\Tn{\lambda_k}\langle s_1\rangle\oplus \cdots\oplus \Tn{\lambda_k}\langle s_{k^{\prime}}\rangle.
\]
Note that the translation functors $\Theta_s$ 
and $\Theta_t$ act on the summands as in 
Corollaries~\ref{cor-theta} and~\ref{cor-funcgrad3} 
separately. To simplify notation, we write 
$\Theta_t(M_k)=M_{k+1}$ if $k$ is even and 
$\Theta_t(M_k)=M_{k}\langle-1\rangle\oplus M_{k}\langle+1\rangle=\tilde{M}_k$ 
if $k$ is odd and vice versa for $\Theta_s$. Moreover, 
in the case $k=0$ we send any red generator to zero.

\item The inductive description has four cases. 
We call these \textit{green-even}, \textit{green-odd}, 
\textit{red-even} and \textit{red-odd} where we only 
give the list for the first two since the other two 
are similar with exchanged conventions for even and 
odd. Each case has five sub-cases (for the five generators) 
giving twenty cases in total. We write $D_{0}$ (even) and 
$D_{1}$ (odd) for the two different cases. 
Since the $M_k$'s could already consist of multiple, shifted 
copies of $\Tn{\bbi}$'s, the entries in the matrices below 
are shorthand notations for matrices \textit{themselves}.

\item Basic pieces: send the empty 
sequence $\bbz$ to $\mathrm{id}\colon\Tn{\bbz}\to\Tn{\bbz}$. 
Moreover, for a green identity we assign
\[
\xy
(0,0)*{\includegraphics[scale=0.85]{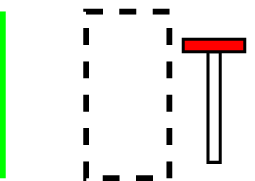}};
(0.5,0)*{D_{0}};
\endxy\mapsto\mathrm{id}\colon M_{k+1}\to M_{k+1},\quad\quad\xy
(0,0)*{\includegraphics[scale=0.85]{res/figs/functor/functor6}};
(0.5,0)*{D_{1}};
\endxy\mapsto\begin{pmatrix}
 \mathrm{id} & 0\\
 0 & \mathrm{id}
\end{pmatrix}\colon \tilde M_{k}\to \tilde M_{k}
\]
For the up dotted edge we assign 
(recalling the rescaling $\tilde{\varepsilon}_k=2^k\varepsilon_k$)
\[
\xy
(0,0)*{\includegraphics[scale=0.85]{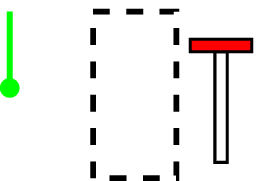}};
(0.5,0)*{D_{0}};
\endxy\mapsto u_k\colon M_{k}\to M_{k+1},\quad\quad\xy
(0,0)*{\includegraphics[scale=0.85]{res/figs/functor/functor7}};
(0.5,0)*{D_{1}};
\endxy\mapsto\begin{pmatrix}
 \tilde{\varepsilon}_k\\
 \mathrm{id}
\end{pmatrix}\colon M_{k}\to \tilde M_{k}
\]
For the down dotted edge we assign 
(recalling the rescaling $\tilde{d}_k=2^kd_k$)
\[
\xy
(0,0)*{\includegraphics[scale=0.85]{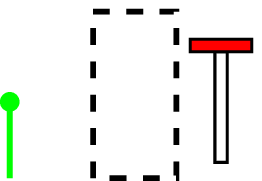}};
(0.5,0)*{D_{0}};
\endxy\mapsto \tilde{d}_{k+1}\colon M_{k+1}\to M_{k},\quad\quad\xy
(0,0)*{\includegraphics[scale=0.85]{res/figs/functor/functor8}};
(0.5,0)*{D_{1}};
\endxy\mapsto\begin{pmatrix}
 \mathrm{id} & \tilde{\varepsilon}_k
\end{pmatrix}\colon \tilde M_{k}\to M_{k}
\]
For the merges and splits in 
the even case we assign
\[
\xy
(0,0)*{\includegraphics[scale=0.85]{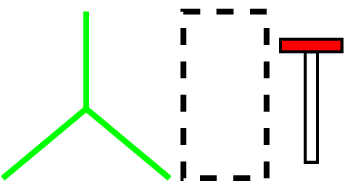}};
(4.5,0)*{D_{0}};
\endxy\mapsto \begin{pmatrix}
 0 & \mathrm{id}
\end{pmatrix}\colon \tilde M_{k+1}\to M_{k+1},\quad\quad\xy
(0,0)*{\includegraphics[scale=0.85]{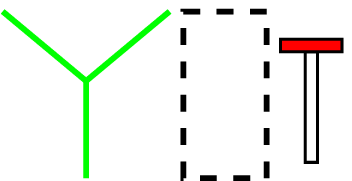}};
(4.5,0)*{D_{0}};
\endxy\mapsto\begin{pmatrix}
 \mathrm{id}\\ 0
\end{pmatrix}\colon M_{k+1}\to \tilde M_{k+1}
\]
For the merges in the odd case we assign
\[
\xy
(0,0)*{\includegraphics[scale=0.85]{res/figs/functor/functor9}};
(4.5,0)*{D_{1}};
\endxy\mapsto \begin{pmatrix}
 0 & 0 & \mathrm{id} & 0\\ 0 & 0 & 0 & \mathrm{id}
\end{pmatrix}\colon \tilde{\tilde{M}}_{k}\to \tilde M_{k}
\]
(with $\tilde{\tilde{M}}_{k}=M_k\langle-2\rangle\oplus M_k\langle0\rangle\oplus M_k\langle0\rangle\oplus M_k\langle+2\rangle$) and last but not least
\[
\xy
(0,0)*{\includegraphics[scale=0.85]{res/figs/functor/functor10}};
(4.5,0)*{D_{1}};
\endxy\mapsto \begin{pmatrix}
 \mathrm{id} & 0\\ 0 & \mathrm{id}\\ 0& 0\\0& 0
\end{pmatrix}\colon \tilde M_k\to\tilde{\tilde{M}}_{k}.
\]

\item ``Fill up the generators to the left'': for each 
red identity strand to the left of the local generators 
from above apply a $\Theta_s$ to the elements of the 
matrices component-wise and likewise for green strands 
and $\Theta_t$. This can be made explicit as explained in Example~\ref{ex-howtofunctor}.

\item If $f$ is a homomorphism $f\in\Hom(\QD)$, 
then decompose it into generators and define 
$\mathcal{D}_{\infty}(f)$ to be the composition of its local pieces.
\end{itemize}
We extend linearly for a general $F\in\Hom(\Mat(\QDk))$.
\end{defn}

\begin{ex}\label{ex-functor}
It is easy to see 
how the assignment is 
for alternating sequences. For example
\[
\xy
(0,0)*{\includegraphics[scale=0.85]{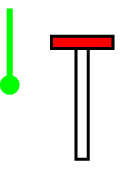}};
\endxy\mapsto u_0\colon \Tn{\bbz}\to \Tn{\bbo},\quad\quad\xy
(0,0)*{\includegraphics[scale=0.85]{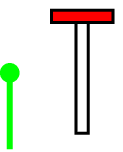}};
\endxy\mapsto \tilde{d}_1\colon\Tn{\bbo}\to\Tn{\bbz}
\]
and similar for all other diagrams of this kind. Furthermore, we have
\begin{gather*}
\begin{aligned}
\xy
(0,0)*{\includegraphics[scale=0.85]{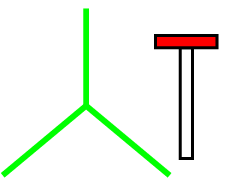}};
\endxy &\mapsto \begin{pmatrix}
 0 & \mathrm{id}
\end{pmatrix}\colon \Tn{\bbo}\langle-1\rangle\oplus \Tn{\bbo}\langle+1\rangle\to \Tn{\bbo}
\\
\xy
(0,0)*{\includegraphics[scale=0.85]{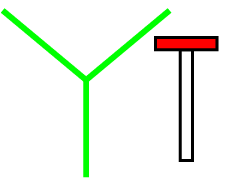}};
\endxy &\mapsto \begin{pmatrix}
 \mathrm{id}\\ 0
\end{pmatrix}\colon\Tn{\bbo}\to\Tn{\bbo}\langle-1\rangle\oplus \Tn{\bbo}\langle+1\rangle
\end{aligned}
\end{gather*}
where we see that the degree 
is preserved. Moreover, the entries in the matrices 
of the list above are in general already matrices. 
For example, the matrix that corresponds to (here $M_1=\Tn{\bbo}$)
\[
\xy
(0,0)*{\includegraphics[scale=0.85]{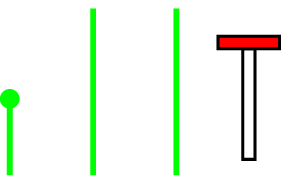}};
\endxy\mapsto \begin{pmatrix}
 \mathrm{id} & \tilde{\varepsilon}_1 & 0 & 0\\ 0& 0&\mathrm{id} & \tilde{\varepsilon}_1
\end{pmatrix}\colon\tilde{\tilde{M_1}}\to\tilde{M_1}
\]
Likewise for the horizontally flipped diagram.
\end{ex}

\begin{lem}\label{lem-welldef}
$\mathcal{D}_{\infty}\colon \Mat(\QDk)\to\Tgrl$ is well-defined, $\bC$-linear and
degree preserving.\makeqed
\end{lem}

\begin{proof}
It is immediate that $\mathcal{D}_{\infty}$ is $\bC$-linear and degree preserving.

Moreover, note that, by part (b) of 
Corollary~\ref{cor-theta} and part (c) of 
Corollary~\ref{cor-funcgrad3}, the assignment is 
well-defined on objects $x$, since we send $x$ to a 
repeated application of $\Theta_s$ and $\Theta_t$ to 
the trivial $\Uq$-module $\Tn{\bbz}$ together with a 
projection onto the leading  
factor $\Tn{\bbi}$. Thus, without taking relations into account, 
we get a well-defined functor from $\Mat(\D_f)$ to $\Tgrl$ 
since our assignment for the generating Soergel diagrams are all $\Uq$-intertwiners.

Thus, the main part is to show that 
the relations are satisfied. This is 
now a case-by-case check where we only 
do a few as examples and leave the rest to the reader.

\textbf{Frobenius and isotopy relations}: 
we do not have to check the Frobenius relations 
\textbf{Frob1} and \textbf{Frob2}~\eqref{eq-frob} and the isotopy relations \textbf{Iso1} and \textbf{Iso2}~\eqref{eq-iso}, since, 
by part (a) of Corollary~\ref{cor-funcgrad3}, we see that $\Theta_s$ and $\Theta_t$ 
are (graded) self-adjoint functors. 
By a more general principle such functors always give rise to a Frobenius structure, see e.g.~\cite[Lemma~3.4]{mueg}.

\textbf{Needle}: for simplicity of 
notation, we assume that we are in the even case. We show
\[
\mathcal D_{\infty}\left(\xy
(0,0)*{\includegraphics[scale=0.85]{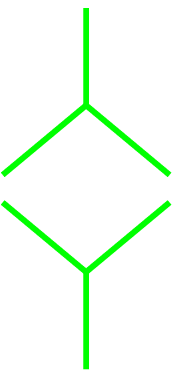}};
(0,-0.5)*{\circ};
\endxy\right)=0
\]
This follows from
\[
\xymatrix{
  M_{k+1}  \ar[rr]|/0.0em/{\begin{pmatrix}
  \mathrm{id} \\ 0
\end{pmatrix}}  &   & \tilde M_{k+1}  \ar[rr]|/0.0em/{\begin{pmatrix}
  0 & 
\mathrm{id}
\end{pmatrix}}  & &  M_{k+1}\\
}
\]
The odd case is similar. 
This clearly implies the needle relation.

\textbf{The dead-end relations 1+2}: \textbf{DE1} follows by construction. The \textbf{DE2} relation is
\[
\mathcal D_{\infty}\left(\xy
(0,0)*{\includegraphics[scale=0.85]{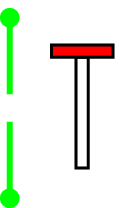}};
(-4.1,-0.2)*{\circ};
\endxy\right)=0
\]
This follows from
\[
\xymatrix{
  \Tn{\bbz}  \ar[r]^/0.0em/{u_0}  & \Tn{\bbo}  \ar[r]^/0.0em/{\tilde{d}_1}  &  \Tn{\bbz}\\
}
\]
which is the dead-end relation in the tilting case.

\textbf{Barbell forcing 1}: we only do one case (the one with equal colors) 
and leave the other (similar) case to the reader. 
Again, for simplicity of notation, we focus on
\[
\mathcal D_{\infty}\left(\xy
(0,0)*{\includegraphics[scale=0.85]{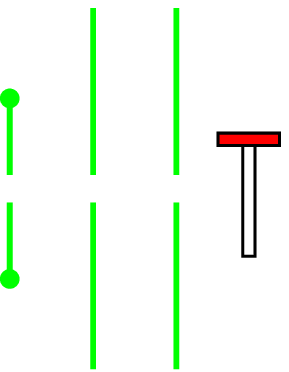}};
(-4.1,-0.25)*{\circ};
\endxy\right)=2\cdot \mathcal D_{\infty}\left(\xy
(0,0)*{\includegraphics[scale=0.85]{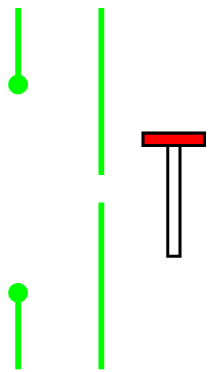}};
(-0.5,-0.25)*{\circ};
\endxy\right)-\mathcal D_{\infty}\left(\xy
(0,0)*{\includegraphics[scale=0.85]{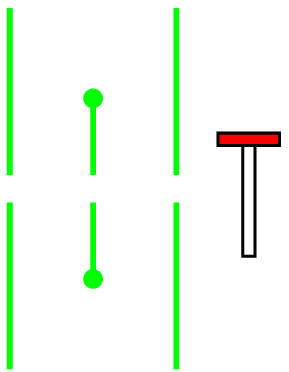}};
(-4.4,-0.25)*{\circ};
\endxy\right)
\]
The general case follows completely similar. 
We have already explained in Example~\ref{ex-functor} 
how to obtain the left two matrices. Thus, we obtain
\[
\begin{pmatrix}
 \mathrm{id} & \varepsilon_1 & 0 & 0\\ 0& 0&\mathrm{id} & \tilde{\varepsilon}_1
\end{pmatrix}\circ \begin{pmatrix}
 \varepsilon_1 & 0\\ \mathrm{id} & 0\\0& \tilde{\varepsilon}_1\\ 0 & \mathrm{id}
\end{pmatrix}=\begin{pmatrix}
 2\tilde{\varepsilon}_1 & 0\\ 0& 2\tilde{\varepsilon}_1
\end{pmatrix}
\]
for the left side. The middle (without the factor $2$) can be directly read off as
\[
\begin{pmatrix}
 \tilde{\varepsilon}_1\\ \mathrm{id}
\end{pmatrix}\circ \begin{pmatrix}
 \mathrm{id} & \tilde{\varepsilon}_1
\end{pmatrix}=\begin{pmatrix}
 \tilde{\varepsilon}_1 & 0\\ \mathrm{id}& \tilde{\varepsilon}_1
\end{pmatrix}
\]
The rightmost term is the result of 
applying $\Theta_t$ to $2\tilde{\varepsilon}_1$ 
(which is the composite of the diagram for the 
first two strands). This can be computed as 
explained in Section~\ref{sec-quiver}. Thus, we obtain
\[
\begin{pmatrix}
 2\tilde{\varepsilon}_1 & 0\\ 0& 2\tilde{\varepsilon}_1
\end{pmatrix}=2\cdot \begin{pmatrix}
 \tilde{\varepsilon}_1 & 0\\ \mathrm{id}& \tilde{\varepsilon}_1
\end{pmatrix}- \begin{pmatrix}
 0 & 0\\ 2\cdot\mathrm{id}& 0
\end{pmatrix}
\]

\textbf{Barbell forcing 2}: this is very similar to \textbf{BF1} 
(only two terms are different) and we only sketch it 
here. As before we only consider one particular case 
(the same as above for \textbf{BF1}) and leave the others for the reader to verify. 
The different two terms are now
\[
\mathcal D_{\infty}\left(\xy
(0,0)*{\includegraphics[scale=0.85]{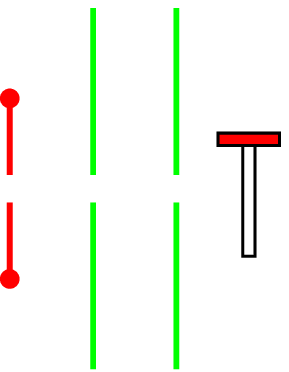}};
(-4.1,-0.25)*{\circ};
\endxy\right),\quad\quad\mathcal D_{\infty}\left(\xy
(0,0)*{\includegraphics[scale=0.85]{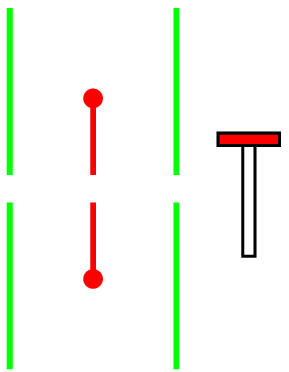}};
(-4.4,-0.25)*{\circ};
\endxy\right)
\]
These two cases give the matrices
\[
\begin{pmatrix}
 2\cdot\tilde{\varepsilon}_1 & 0\\ 0& 2\cdot\tilde{\varepsilon}_1
\end{pmatrix}\quad\text{and}\quad\begin{pmatrix}
 0 & 0\\ 2\cdot\mathrm{id}& 0
\end{pmatrix}
\]
and the equation \textbf{BF2}~\eqref{eq-barbforc} then reads as
\[
\begin{pmatrix}
 2\cdot\tilde{\varepsilon}_1 & 0\\ 0& 2\cdot\tilde{\varepsilon}_1
\end{pmatrix}=\begin{pmatrix}
 0 & 0\\ 2\cdot\mathrm{id}& 0
\end{pmatrix}+2\begin{pmatrix}
 \tilde{\varepsilon}_1 & 0\\ \mathrm{id}& \tilde{\varepsilon}_1
\end{pmatrix}- 2\begin{pmatrix}
 0 & 0\\ 2\cdot\mathrm{id}& 0
\end{pmatrix}
\]
This finishes the proof.
\end{proof}

\begin{thm}(\textbf{Diagram categories for $\Tgrl$})\label{thm-diafortilt}
The functor $\mathcal{D}_{\infty}\colon \Mat(\QDk)\to\Tgrl$ is an equivalence of graded categories.\makeqed
\end{thm}

\begin{proof}
We have to show that $\mathcal{D}_{\infty}$ is 
essentially surjective, full and faithful.

\textbf{Essentially surjective}: we have to show 
that, given some arbitrary object $M\in\Ob(\Tgrl)$, 
then there is an object $x\in\Ob(\Mat(\QDk))$ such 
that $\mathcal{D}_{\infty}(x)\cong M$. To see this, 
note that, by part (c) of Proposition~\ref{prop-tilt}, 
it is enough to verify this for indecomposable 
$\Uq$-modules $\Tn{\bbi}$ in the ungraded setting 
(using Proposition~\ref{prop-grading} the same is still true in the graded setting).
Now, because of our construction and part 
(b) of Corollary~\ref{cor-theta}, we see that 
$\mathcal{D}_{\infty}(x_{i-tst})\cong\Tn{\bbi}$ 
which shows that the functor is essentially surjective.

\textbf{Fully faithful}: by Lemmas~\ref{lem-decom2} and~\ref{lem-welldef} 
it is enough to show that
\[
\Hom_{\Mat(\QDk)}(x,x^{\prime})\cong\Hom_{\Tgrl}(\mathcal{D}_{\infty}(x),\mathcal{D}_{\infty}(x^{\prime}))
\]
holds as graded $\bC$-vector spaces for alternating sequences $x=x_{i-tst}$ and $x^{\prime}_{i^{\prime}-tst}$.
We have already computed both sides: the 
right side was computed in 
Corollary~\ref{cor-homspaces} and the 
left in Corollary~\ref{cor-homspaces2}. 
By our construction, the (graded) isomorphism 
is induced by the assignment 
$\bbi\mapsto\bbi$, $u_i\mapsto u_i$, 
$d_i\mapsto \tilde{d}_i$ and 
$\varepsilon_i\mapsto\tilde{\varepsilon}_i$ 
whenever this makes sense.
\end{proof}

\begin{cor}\label{cor-diafrotilt}
The category $\Mat(\QDk)$ is idempotent complete.\makeqed
\end{cor}

\begin{proof}
This follows from Theorem~\ref{thm-diafortilt} 
because in $\Tgrl\cong\T_{\lambda}$ every module 
decomposes into a direct sum of (shifted copies of) $\Tn{\bbi}$'s. 
\end{proof}

\begin{cor}(\textbf{Diagram categories for the center})\label{cor-diafrocenter}
The graded $\bC$-algebra $\End^{\mathrm{fs}}_{\mathrm{gr}}(\mathrm{id})$ of 
natural transformations in $\Endo(\Tgrl)$ of the 
identity functor $\mathrm{id}$ is given by the 
diagonal part of finitely supported (only a finite number of non-zero entries) matrices
\begin{equation}\label{eq-fcenter}
F\colon \emptyset\oplus g\oplus rg\oplus grg\oplus rgrg\oplus\dots\to\emptyset\oplus g\oplus rg\oplus grg\oplus rgrg\oplus\dots
\end{equation}
of marked Soergel diagrams with no identity 
diagram entry. In contrast, 
$\End^{\mathrm{fs}}_{\Uq}(\Tn{\bbin})\cong \Ai$ is 
given by all such matrices and not just the diagonal part.\makeqed
\end{cor}

\begin{proof}
An alternating sequence of red and green of 
length $i$ corresponds under the equivalence 
from Theorem~\ref{thm-diafortilt} to $\Tn{\bbi}$, 
since, by Proposition~\ref{prop-quiverrel}, the 
quiver relations are satisfied. This module 
corresponds under the isomorphisms from 
Corollary~\ref{cor-center} to ${}_iP$. Thus, the 
claim follows by Theorem~\ref{thm-diafortilt} and 
Proposition~\ref{prop-endograd2}, since (graded) 
equivalent categories have (graded) isomorphic 
centers and the observation that an $F$ as 
in~\eqref{eq-fcenter} is a natural transformation 
iff it commutes with all other such $F$'s 
(compare also to~\eqref{eq-natural}). That 
the center of the matrix ring is as stated 
above is then an easy to deduce fact: the 
$\bbi$ are not central (as already noted 
after Definition~\ref{def-ksquiver1}), while 
the $\varepsilon_i$ compose with everything to zero 
(and are therefore in the center).
\end{proof}

\begin{rem}\label{rem-overfinitefields}
Relaxing the conditions 
to work over another field and not $\bC$ is actually a 
problem in Section~\ref{sec-tilting} since the 
indecomposable tilting modules and their hom-spaces 
are much more complicated in positive characteristic. 
For the KS quiver algebras from Section~\ref{sec-quiver} working over 
$\bZ$ is not a huge problem, while for the diagrammatic category 
$\QD$ we need to work over $\bQ$.
\end{rem}

\subsection{Diagram categories for \texorpdfstring{$\Endo(\Tgrl)$}{Endo(Tgrl)}}\label{sub-projdia}

We denote by $\overline{\Mat}(\mathcal C)$ 
literally the same category as $\Mat(\mathcal C)$, 
but we allow \textit{countable} direct sums as 
objects and \textit{finitely supported} matrices as morphisms.
We carefully distinguish between 
red $r$ and green $g$ on one hand 
and $s,t\in W_l$ on the other. We 
think of an \textit{alternating} 
sequence (of length $i\geq 0$) of the 
former $x_{i-grg}$ to correspond (under 
Theorem~\ref{thm-diafortilt}) to $\Tn{\bbi}$ and 
of \textit{any} sequence of the latter $x=x_i\dots x_2x_1$ 
to correspond to $\Theta_x=\Theta_{x_i}\circ\cdots\circ\Theta_{x_1}$. 
Using this interpretation, we are going to make 
Proposition~\ref{prop-endograd2} explicit.
We define $\vec{\infty}$ to be the sequence of alternating $x_{i-grg}$ sequences
\[
\vec{\infty}=(\dots,rgrg,grg,rg,g,\emptyset),\;\vec{\infty}_i=x_{i-grg}\quad\text{for}\quad i=0,1,\dots,
\]
where we read the vector from right to left.

Moreover, given a finite sequence $x=x_i\dots x_2x_1$ with $x_i\in\{s,t\}$ we define
\[
x\cdot \vec{\infty}=(\dots,xrgrg,xgrg,xrg,xg,x),\;(x\cdot \vec{\infty})_i=xx_{i-grg}\quad\text{for}\quad i=0,1,\dots,
\]
where each entry is given by concatenation. In addition, we define
\[
xx_{i-grg}=\bigoplus_{i^{\prime}}x_{i^{\prime}-grg}^{m_{i^{\prime}}}\langle s_{i^{\prime}}\rangle
\]
where the multiplicity 
$m_{i^{\prime}}$ is the 
multiplicity of $\Tn{\bbi^{\prime}}$ in 
$\Theta_x(\Tn{\bbi})$ given inductively 
by part c of Corollary~\ref{cor-funcgrad3} 
(from where we also get the shifts $s_{i^{\prime}}$). 
Thus, we see $xx_{i-grg}$ as an object 
$\Ob(\overline{\Mat}(\QDk))$. Note that 
$xx_{0-grg}=0$ if $x$ starts with $r$ because of $\Theta_s(\Tn 0)=0$.

\begin{defn}\label{defn-diaforend1}
Denote by $\Mati(\QDk)$ the $\bC$-linear,
 graded category consisting of:
\begin{itemize}
\item Objects $\Ob(\Mati(\QDk))$ are 
finite sequences $x=x_i\dots x_2x_1$ 
with $x_i\in\{s,t\}$ (plus shifts). 
Moreover, the empty sequence $\emptyset$ is also an object.

\item The space of morphisms $\Hom_{\Mati(\QDk)}(x,x^{\prime})$ 
for $x,x^{\prime}\in\Ob(\Mati(\QDk))$ is the $\bC$-linear 
span of finitely supported (only a finite number of non-zero entries) matrices
\[
F=(F_{k^{\prime}k})_{k,k^{\prime}\in\bN}\in\Hom_{\overline{\Mat}(\QDk)}(x\cdot \vec{\infty},x^{\prime}\cdot \vec{\infty})
\]
of marked Soergel graphs $F_{k^{\prime}k}\colon (x\cdot \vec{\infty})_k\to (x^{\prime}\cdot \vec{\infty})_{k^{\prime}}$.

\item The diagonal part of $F$ consists 
of matrices $F^i\colon (x\cdot \vec{\infty})_i\to (x^{\prime}\cdot \vec{\infty})_i$.

\item Composition of morphisms is multiplication of matrices.

\end{itemize}
The spaces $\Hom_{\Mati(\QDk)}(x,x^{\prime})$ 
are graded $\bC$-vector spaces where the degree is 
given by $l(F_{k^{\prime}k})$ (the degree or length 
from Definition~\ref{defn-diacath}) for each entry of matrices $F$.
\end{defn}

Denote by $\Endo_{\Theta}^{\mathrm{fs}}(\Tgrl)$ the full 
subcategory of $\Endo(\Tgrl)$ consisting of 
only compositions of $\Theta_s$ and $\Theta_t$ (together with the 
condition of being finitely supported). 
Note that, by Proposition~\ref{prop-profunc}, 
we have
\[
\Kar(\Endo_{\Theta}^{\mathrm{fs}}(\Tgrl))\cong\Endo^{\mathrm{fs}}(\Tgrl).
\] 
Moreover, we denote by $\Mati(\QDk)_c$ the 
subcategory of $\Mati(\QDk)$ whose hom-spaces are
\begin{equation}\label{eq-homfunctors}
\Hom_{\Mati(\QDk)_c}(x,x^{\prime})=\{F\in\Hom_{\Mati(\QDk)}(x,x^{\prime})\mid FG=G^{\prime}F\},
\end{equation}
for $G\in\Hom_{\Mati(\QDk)}(x,x)$ and 
$G^{\prime}\in\Hom_{\Mati(\QDk)}(x^{\prime},x^{\prime})$ 
such that entry-wise there exists an $\Uq$-intertwiner 
$f\colon M\to M^{\prime}$ with 
$\mathcal{D}_{\infty}(G_{i^{\prime}i})=\Theta_x(f)$ and $\mathcal{D}_{\infty}(G^{\prime}_{i^{\prime}i})=\Theta_{x^{\prime}}(f)$. 
The reader may check that $\Mati(\QDk)_c$ is really a subcategory.

\begin{defn}\label{defn-diaforend2}
We define a functor
\[
\mathcal{D}^{\infty}\colon\Endo_{\Theta}^{\mathrm{fs}}(\Tgrl)\to\Mati(\QDk)_c
\]
on objects (we treat shifts again in the evident way) and morphisms as
\[
\Theta_x\mapsto x,\quad\text{and}\quad\eta\colon\Theta_x\to\Theta_{x^{\prime}},\,\eta_{\Tn{\bbi}}\colon\Theta_x\Tn{\bbi}\to\Theta_{x^{\prime}}\Tn{\bbi} \mapsto \mathrm{diag}(F)=(F^i)_{i\in\bN}\colon x\to x^{\prime},
\]
where $\mathrm{diag}(F)$ is a diagonal matrix 
consisting of the various $F^i$'s and, for 
each $\bbi$, the matrix of marked Soergel 
diagrams $F^i$ is component-wise given by 
$\bbi\mapsto\bbi$, $u_i\mapsto u_i$, 
$d_i\mapsto d_i$ and $\varepsilon_i\mapsto\varepsilon_i$ 
for all suitable indices $i$ and marked Soergel diagrams as in Corollary~\ref{cor-homspaces2}.
\end{defn}

\begin{lem}\label{lem-welldefagain}
The functor $\mathcal{D}^{\infty}\colon\Endo_{\Theta}^{\mathrm{fs}}(\Tgrl)\to\Mati(\QDk)_c$ 
is a well-defined functor between $\bC$-linear, graded categories.\makeqed
\end{lem}

\begin{proof}
The $\bC$-linearity is clear. 
In addition, the assignment is clearly degree preserving 
and well-defined component-wise (since the quiver relations are satisfied, 
see Proposition~\ref{prop-quiverrel}).

To see that it is well-defined note that 
$\eta_{\Tn{\bbi}}$ is a matrix consisting 
of $\Uq$-intertwiners between the direct 
summands of $\Theta_x\Tn{\bbi}$ and 
$\Theta_{x^{\prime}}\Tn{\bbi}$ (for each $\Tn{\bbi}$). 
These are sums of $\bbi,u_i,d_i$ and $\varepsilon_i$ 
by the isomorphism of Proposition~\ref{prop-isoofalg2}. 
To see that the relations are satisfied, note that, 
given any $\Uq$-intertwiner $f\colon M\to M^{\prime}$, 
the naturality of a transformation $\eta\colon \Theta_x\to \Theta_{x^{\prime}}$ says that
\begin{equation}\label{eq-natural}
\raisebox{0.9cm}{\begin{xy}
  \xymatrix{
      \Theta_x M \ar[r]^{\Theta_x f} \ar[d]_{\eta_M}    &   \Theta_x M^{\prime} \ar[d]^{\eta_{M^{\prime}}}  \\
      \Theta_{x^{\prime}} M \ar[r]_{\Theta_{x^{\prime}} f}             &   \Theta_{x^{\prime}} M^{\prime}  
  }
\end{xy}}
\end{equation}
commutes. Thus, the matrices coming from our assignment 
satisfy the condition~\eqref{eq-homfunctors}.
\end{proof}

\newpage

\begin{thm}(\textbf{Diagram categories for $\Endo(\Tgrl)$})\label{thm-diafortilt1}
The functor $\mathcal{D}^{\infty}$ is a graded equivalence.\makeqed
\end{thm}

\begin{proof}
As in Theorem~\ref{thm-diafortilt}, we have to show that $\mathcal{D}^{\infty}$ is 
essentially surjective, full and faithful.

That $\mathcal{D}^{\infty}$ is essentially surjective 
follows from the definition of the objects in 
$\Mati(\QDk)_d$. That it is faithful is a 
direct consequence of Proposition~\ref{prop-endograd2} 
combined with Theorem~\ref{thm-diafortilt}. 
That $\mathcal{D}^{\infty}$ is full is just a 
direct comparison of~\eqref{eq-homfunctors} and~\eqref{eq-natural}. 
\end{proof}

\begin{rem}\label{rem-possibleextension}
A possible 
generalization of our diagrammatic categories could follow from work of 
Elias and Libedinsky on Coxeter groups, 
see~\cite{el2}. The
underlying Coxeter group $W$ will for $A_n$ be the affine Weyl group, that is, the one given by the cyclic Dynkin diagram with $n+1$-nodes
\[
\xy 0;/r.17pc/:
(-5,8.66)*+{\bullet}="1";
(-5,-8.66)*+{\bullet}="2";
(10,0)*+{{\color{green}\bullet}}="3";
{\ar@{-} "1";"2" };
{\ar@{-} "2";"3" };
{\ar@{-} "3";"1" };
(0,-12.5)*{n=2};
\endxy\hspace*{1.5cm}
\xy 0;/r.17pc/:
(0,10)*+{\bullet}="1";
(-10,0)*+{\bullet}="2";
(0,-10)*+{\bullet}="3";
(10,0)*+{{\color{green}\bullet}}="4";
{\ar@{-} "1";"2" };
{\ar@{-} "2";"3" };
{\ar@{-} "3";"4" };
{\ar@{-} "4";"1" };
(0,-12.5)*{n=3};
\endxy\hspace*{1.5cm}
\xy 0;/r.17pc/:
(3.09,9.51)*+{\bullet}="1";
(-8.09,5.88)*+{\bullet}="2";
(-8.09,-5.88)*+{\bullet}="3";
(3.09,-9.51)*+{\bullet}="4";
(10,0)*+{{\color{green}\bullet}}="5";
{\ar@{-} "1";"2" };
{\ar@{-} "2";"3" };
{\ar@{-} "3";"4" };
{\ar@{-} "4";"5" };
{\ar@{-} "1";"5" };
(0,-12.5)*{n=4};
\endxy\hspace*{1.5cm}
\xy 0;/r.17pc/:
(5,8.66)*+{\bullet}="1";
(-5,8.66)*+{\bullet}="2";
(-10,0)*+{\bullet}="3";
(-5,-8.66)*+{\bullet}="4";
(5,-8.66)*+{\bullet}="5";
(10,0)*+{{\color{green}\bullet}}="6";
{\ar@{-} "1";"2" };
{\ar@{-} "2";"3" };
{\ar@{-} "3";"4" };
{\ar@{-} "4";"5" };
{\ar@{-} "5";"6" };
{\ar@{-} "1";"6" };
(0,-12.5)*{n=5};
\endxy\hspace*{1.5cm}\cdots
\]
where the green (rightmost) nodes indicate the affine nodes.
\end{rem}

\bibliographystyle{plainurl}
\bibliography{dia-tiltings}
\end{document}